\documentclass[12pt]{report}
\usepackage[procnames]{listings}
\usepackage[font={footnotesize}]{caption}
\usepackage{color}
\usepackage{urcsbiblio,urcsthesis}
\usepackage{amsmath,amsfonts,amsthm,amssymb,bbm,tikz}
\usetikzlibrary{decorations.pathmorphing}
\usepackage[margin=1.25in]{geometry}
\usepackage{setspace}
\usepackage{amsmath}
\usepackage[english]{babel}
\usepackage{graphicx}
\usepackage{hyperref}
\newtheorem{thm}{Theorem}[section]

\newtheorem{prop}[thm]{Proposition}
\newtheorem{lem}[thm]{Lemma}

\newtheorem{theorem}{Theorem}[section]

\newtheorem{lemma}[theorem]{Lemma}

\newtheorem{corollary}[theorem]{Corollary}

\newtheorem{conj}[theorem]{Conjecture}

\newcommand{\R}{\mathbb{R}}
\newcommand{\C}{\mathbb{C}}
\newcommand{\E}{\mathbb{E}}
\newcommand{\be}{\begin{equation}}
\newcommand{\ee}{\end{equation}}

\newcommand{\Mat}{\operatorname{Mat}}

\newcommand{\reals}{\mathbb{R}}

\newcommand{\Dform}[2]{\mathcal{Q}\left(#1,#2 \right)}
\newcommand{\tripNorm}[1]{|||#1|||_\infty}

\theoremstyle{definition}

\begin{document}

\title{Concentration of Measure Techniques and Applications}
\author{Meg Walters}
\thesissupervisor{Professor Shannon Starr \\ Professor Carl Mueller}
\date{2015}
\maketitle

\begin{curriculumvitae}
Meg Walters grew up in Gainesville, FL.  She moved to Rochester in 2005 to study bassoon performance at the Eastman School of Music and enrolled at the University of Rochester as an applied math major in 2007.  She graduated with a Bachelor of Music degree from Eastman and a Bachelor of Science in Applied Mathematics from the University of Rochester in 2010.  In the fall of 2010, she started her graduate studies in mathematics at the University of Rochester.  She received her Master of Arts degree in 2012 and began studying probability and mathematical physics under the supervision of Shannon Starr and Carl Mueller.  
\vfill

{\bf Publications:}
\begin{itemize}
\item Ng, S., and Walters, M. (2014). Random Operator Compressions. {\em arXiv preprint arXiv:1407.6306.}
\item Walters, M., and Starr, S. (2015). A note on mixed matrix moments for the complex Ginibre ensemble. {\em Journal of Mathematical Physics}, 56(1), 013301.
\item Starr, S., and Walters, M. (2015). Phase Uniqueness for the Mallows Measure on Permutations. {\em arXiv preprint arXiv:1502.03727}.
\end{itemize}
\end{curriculumvitae}

\begin{acknowledgments}
I would first and foremost like to thank my advisor, Shannon Starr, for his encouragement, patience, and guidance during my time as his student.  This work would not have been possible without him.  I would also like to thank Carl Mueller for all of his assistance after Professor Starr's relocation to Alabama.  

I am extremely grateful for all of the assistance that Joan Robinson, Hazel McKnight, and Maureen Gaelens have provided throughout my graduate studies.  

I would also like to thank my parents and David Langley for all of their love, support, and patience throughout the years.  
  
\end{acknowledgments}

\begin{abstract}
Concentration of measure is a phenomenon in which a random variable that depends in a smooth way on a large number of independent random variables is essentially constant.  The random variable will "concentrate" around its median or expectation.  In this work, we explore several theories and applications of concentration of measure.  The results of the thesis are divided into three main parts.  In the first part, we explore concentration of measure for several random operator compressions and for the length of the longest increasing subsequence of a random walk evolving under the asymmetric exclusion process, by generalizing an approach of Chatterjee and Ledoux.  In the second part, we consider the mixed matrix moments of the complex Ginibre ensemble and relate them to the expected overlap functions of the eigenvectors as introduced by Chalker and Mehlig.  In the third part, we develop a $q$-Stirling's formula and discuss a method for simulating a random permutation distributed according to the Mallows measure.  We then apply the $q$-Stirling's formula to obtain asymptotics for a four square decomposition of points distributed in a square according to the Mallows measure.  All of the results in the third part are preliminary steps toward bounding the fluctuations of the length of the longest increasing subsequence of a Mallows permutation.  
\end{abstract}

\begin{contributorsandfunding}
This work was supervised by a dissertation committee consisting of Shannon Starr (advisor) of the Department of Applied Mathematics at the University of Alabama Birmingham, Carl Mueller (co-advisor) and Alex Iosevich of the Department of Mathematics, and Yonathan Shapir of the Department of Physics and Astronomy at the University of Rochester.  The chair of the committee was Daniel \v Stefankovi\v c of the Department of Computer Science. \\ \\
The results obtained in chapter 3 were obtained in collaboration with Stephen Ng (Exelis Geospatial Systems), and the results obtained in chapter 4 were obtained in collaboration with Shannon Starr (UAB).  In addition, the results in chapters 5 were problems suggested to me by Shannon Starr and were obtained independently by me with the guidance and suggestions of Shannon Starr.  \\ \\
This work was partially funded by NSA Grant H98230-12-1-0211.
\end{contributorsandfunding}

\tableofcontents
\listoffigures
\chapter{Introduction}

The idea of concentration of measure was first introduced by Milman in the asymptotic theory of Banach spaces \cite{Milman}.  The phenomenon occurs geometrically only in high dimensions, or probabilistically for a large number of random variables with sufficient independence between them. For an overview of the history and some standard results, see \cite{LedouxCOM}.

 A illustrative geometric example of concentration of measure occurs for the standard $n$-sphere $\mathbb{S}^n$ in $\mathbb{R}^{n+1}$.  If we let $\mu_n$ denote the uniform measure on $\mathbb{S}^n$, then for large enough $n$, $\mu_n$ is highly concentrated around the equator.  

To see exactly what we mean by "highly concentrated", let us consider any measurable set $A$ on $\mathbb{S}^n$ such that $\mu_n(A)\geq 1/2$.  Then, if we let $d(x,A)$ be the geodesic distance between $x\in\mathbb{S}^n$ and $A$, we define the expanded set 
$$
A_t=\{x\in\mathbb{S}^n\; ; \; d(x,A)<r\}
$$
$A_t$ contains all points of $A$ in addition to any points on $\mathbb{S}^n$ with a geodesic distance less than $r$ from $A$.  The precise inequality that can be obtained says that 
$$
\mu_n(A_r)\geq 1-e^{-(n-1)r^2/2}
$$
In other words "almost" all points of the sphere are within distance of $\frac{1}{\sqrt{n}}$ from our set $A$.  Obviously as $n\rightarrow\infty$, this quantity because infinitesimal.  This example is due to Gromov, and more discussion can be found in \cite{Gromov}.  

Gromov's work on concentration on the sphere was inspired by L\'{e}vy's work \cite{Levy} on concentration of functions.  Suppose we have a function $F$, which is continuous on $\mathbb{S}^n$ with a modulus of continuity given by $\omega_F(t)=\sup\{|F(x)-F(y)|\; : \;d(x,y)\leq t\}$.  Let $m_F$ be a median for $F$, which by definition means that $\mu_n(F\geq m_F)\geq 1/2$ and $\mu_n(F\leq m)\geq 1/2$.  Then we have
$$
\mu_n(\{|F-m_F|\geq \omega_F(t)\})\leq 2e^{-(n-1)t^2}
$$

While these geometric examples give a nice introduction to the phenomenon, in this work we will mainly be interested in concentration of measure in a probabilistic setting. Let us give a simple example that will give some intuition about how concentration of measure comes up in probability.  Suppose we have independent random variables $X_1,X_2,\dots, X_n$.  Suppose that they take the values $1$ and $-1$, each with probability $1/2$.  For each $n\geq 1$, let $S_n=\sum_{i=1}^n X_i$.  Since $\mathbb{E}(X_i)<\infty$ (in fact $\mathbb{E}(X_i)=0$), the strong law of large numbers tells us that $S_n/n$ converges almost surely to $\mathbb{E}(X_i)$ as $n\rightarrow\infty$.  Remember that this means that 
$$
\mathbb{P}\left(\lim_{n\rightarrow\infty} \frac{S_n}{n}=\mathbb{E}(X_i)\right)=1
$$ 
Moreover, by the central limit theorem, we know that 
$$
\left(\frac{S_n}{\sqrt{n}}\right)\xrightarrow{d} N(0,\sigma^2)
$$
where $\sigma^2$ is the variance of each $X_i$, which in this case is $1$.  This shows us that the fluctuations of $S_n$ are of order $n$.  However, notice that $|S_n|$ can take values as large as $n$.  If we measure $S_n$ using this scale, then $\frac{S_n}{n}$ is essentially zero.  The actual bound looks like 
$$
\mathbb{P}\left(\frac{|S_n|}{n}\geq r\right)\leq 2e^{-nr^2/2}
$$
for $r>0$.  See \cite{TalagrandInd} for a proof.  As Talagrand points out, concentration of measure appears in a probabilistic setting by showing that one random variable that depends in a smooth enough way on many other independent random variables is close to constant, provided that it does not depend too much on any one of the independent random variables.  As we will see later in this work, it turns out that this idea still holds true if we have a random variable that depends on a large number of "almost" independent random variables.  We will later see an instance of a random variable that depends on many weakly correlated random variables.  It requires a little more work to prove concentration of measure, but often, it is still possible. 

This work is divided into chapters.  Chapter 2 introduces Talagrand's Gaussian concentration of measure inequality, Talagrand's isoperimetric inequality, Ledoux's concentration of measure on Markov chains, and the Euler-Maclaurin formula.  A statement and proof of each theorem (with the exception of the Euler-Maclaurin formula) is also given, to make this work as self-contained as possible.  In later chapters, we will see new applications of each of these results.  Chapter 3 introduces several new results using Ledoux's concentration of measure inequality on reversible Markov chains.  We are able to generalize a method first used by Chatterjee and Ledoux \cite{ChatLed} to prove concentration of measure for two different random operator compressions.  We also show how to use this method to obtain concentration of measure bounds for the length of the longest increasing subsequence of a random walk evolving under the asymmetric exclusion process.  To give more meaning to our fluctuation bounds, we also derive a lower bound for the length of this longest increasing subsequence.  It turns out that we can use Talagrand's isoperimetric inequality to do this, even though our random variables have weak correlations.  In Chapter 4, we discuss a method for calculating the mixed matrix moments in the Ginibre random matrix ensemble using techniques from spin glasses.  In addition, we use the mixed matrix moments to compute asymptotics of the overlap functions (introduced by Chalker and Mehlig \cite{CM}) for eigenvectors corresponding to eigenvalues near the edge of the unit circle. We propose an adiabatic method for computing explicit formulas for the eigenvector overlap functions.   In Chapter 5, we use the Euler-Maclaurin formula to prove a $q$-deformed Stirling's formula.  We demonstrate a use of the $q$-Stirling's formula to obtain asymptotics for point counts in a four square problem.  We also discuss techniques and algorithms to simulate a Mallows random permutation.  

\chapter{Concentration of Measure Results and other Necessary Background}
\section{Talagrand's Gaussian Concentration of Measure Inequality}
Michel Talagrand has made numerous contributions to the theory of concentration of measure.  The first concentration of measure result that we will present applies to Lipschitz functions of Gaussian random variables, so we will refer to it henceforth as Talagrand's Gaussian concentration of measure inequality, to distinguish it from other results of Talagrand that we will use.  Before stating the theorem, recall that a Lipschitz function $F$ on $\mathbb{R}^M$, with Lipschitz constant $A$, satisfies 
$$
|F({\bf x})-F({\bf y})|\leq A \|{\bf x}-{\bf y}\|
$$
where  $\|\bf{x}-\bf{y}\|$ is the Eucliean distance between $\bf{x}$ and $\bf{y}$. 
The following theorem is due to Talagrand \cite{Talagrandbook}
\begin{thm}
\label{GCOM}
Consider a Lipschitz function $F$ on $\mathbb{R}^M$, with Lipschitz constant $A$.  Let $x_1,\dots,x_M$ denote independent standard Gaussian random variables, and let ${\bf x}=(x_1,\dots,x_M)$.  Then for each $t>0$, we have 
\be
\mathbb{P}(|F({\bf x})-\mathbb{E}F({\bf x})|\geq t)\leq 2\exp\left(-\frac{t^2}{4A^2}\right)
\ee
\end{thm}
\begin{proof}
For this proof, we will assume that $F$ is not only Lipschitz, but also twice differentiable.  This is the case in most applications of this theorem, and if it is not the case, we can regularize $F$ by convoluting with a smooth function to solve the problem.  We begin with a parameter $s$ and consider a function $G$ on $\mathbb{R}^{2M}$ defined as 
$$
G(z_1,\dots,z_{2M})=\exp\left ( s(F(z_1,\dots z_M)-F(z_{M+1},\dots, z_{2M}))\right)
$$
Let $u_1,\dots u_{2M}$ be $2M$ independent standard Gaussian random variables.  Let $v_1,\dots v_{2M}$ also be $2M$ random variables (independent of the $u_1,\dots, u_M$) such that first $M$ ($v_1,\dots, v_M$) are independent standard Gaussians and such that the second $M$ variables ($v_{M+1}, \dots, v_{2M}$) are copies of the first $M$ $v's$, in order.  (i.e. $v_i=v_{i+M}$ if $i\leq M$.)
Notice that due to the independence of the $u$'s and the first $M$ $v$'s, we have 
$$
\mathbb{E}u_iu_j-\mathbb{E}v_iv_j =0
$$
except when $j=i+M $ or $i=j+M$, in which case we have 
$$
\mathbb{E}u_iu_j-\mathbb{E}v_iv_j=0-1=-1
$$
We consider a function ${\bf f}(t)=(f_1,\dots f_{2M})(t)$ given by 
$$
f_i(t)=\sqrt{t}u_i+\sqrt{1-t}v_i
$$
Note that ${\bf f}(0)={\bf v}$ and that ${\bf f}(1)={\bf u}$.  
Also, consider 
$$
\phi(t)=\mathbb{E}G({\bf f}(t))
$$
so that 
$$ \phi '(t)=\mathbb{E}\sum_{i=1}^{2M}\frac{d}{dt}f_i(t)\frac{\partial G}{\partial x_i}({\bf f}(t))
$$
To simplify $\phi '(t)$, recall the Gaussian integration by parts formula.  For Gaussian random variables $y,y_1,\dots, y_n$, and a function $F$ (of moderate growth at infinity), we have
$$
\mathbb{E}yF(y_1,\dots,y_n)=\sum_{i=1}^n \mathbb{E}(y y_i)\mathbb{E}\frac{\partial F}{\partial x_i}(y_1,\dots, y_n)
$$
(See \cite{Talagrandbook} Appendix 6 for a proof).

Using the fact that 
$$
\frac{d}{dt} f_i=\frac{1}{2\sqrt{t}}u_i-\frac{1}{2\sqrt{1-t}}v_i
$$
and applying Gaussian integration by parts, gives 
$$
\phi'(t)=\sum_{i,j=1}^{2M}\mathbb{E}\left(\frac{1}{2\sqrt{t}}u_i-\frac{1}{2\sqrt{1-t}}v_i\right)(\sqrt{t}u_j+\sqrt{1-t}v_j)\mathbb{E}\frac{\partial ^2 G}{\partial z_i\partial{z_j}}{\bf f}(t)
$$
Using the independence of the $u$'s and the $v$'s, we have that 
$$
\mathbb{E}\left(\frac{1}{2\sqrt{t}}u_i-\frac{1}{2\sqrt{1-t}}v_i\right)(\sqrt{t}u_j+\sqrt{1-t}v_j)=\frac{1}{2}(\mathbb{E}u_iu_j-\mathbb{E}v_iv_j)
$$
which we have already determined is equal to $0$ unless $j=i+M$ or $i=j+M$ (in which case it is $-1$), so we have 
$$
\phi'(t)=-\mathbb{E}\sum_{i=1^M}\frac{\partial ^2 G}{\partial z_i\partial z_{i+M}}({\bf u}(t))
$$
Computing the second derivative gives 
$$
\frac{\partial ^2 G}{\partial z_i\partial z_{i+M}}({\bf z})=-s^2\frac{\partial F}{\partial x_i}(z_1,\dots, z_M)\frac{\partial F}{\partial x_i}(z_{M+1},\dots, z_{2M}) G({\bf z})
$$
Since $F$ is Lipschitz, we know that for all ${\bf x}\in\mathbb{R}^M$, 
$$
\sum_{i=1}^M \left(\frac{\partial F}{\partial x_i}{\bf x}\right)^2\leq A^2
$$
so we can use the Cauchy-Schwarz inequality to get 
$$
\phi '(t)\leq s^2 A^2\phi(t)
$$
Notice that $\phi(0)=1$ (since at $t=0$ the $u$'s disappear and the second half of the $v$'s cancel the first half).
Hence we have
$$
\phi '(t)/\phi(t) \leq s^2A^2
$$
so 
$$
\log(\phi(t))\leq s^2A^2t + C
$$
or 
$$
\phi(t)\leq e^{s^2 A^2t}
$$
and 
$$
\phi(t)\leq \exp(s^2 A^2)
$$
Recalling that $f_i(1)=u_i$, this tells us that 
$$
\mathbb{E}\exp(s(F(u_1,\dots, u_M)-F(u_{M+1},\dots, u_{2M}))\leq e^{s^2 A^2}
$$
By independence of the $u$'s, we have that 
\begin{multline}
\mathbb{E}\exp(s(F(u_1,\dots, u_M)-F(u_{M+1},\dots, u_{2M}))) \\ =\mathbb{E}\exp(s(F(u_1,\dots, u_M)\mathbb{E}\exp(-sF(u_{M+1},\dots, u_{2M}))
\end{multline}
By Jensen's inequality, we know that 
$$
\mathbb{E}\exp(-sF(u_{M+1},\dots, u_{2M}))\geq \exp(-s\mathbb{E}F(u_{M+1},\dots,u_{2M}))
$$ for $s>0$.

Putting this all together, we have 
$$
\mathbb{E}\exp s(F({\bf x})-\mathbb{E}F({\bf x}))\leq e^{s^2A^2}
$$
where ${\bf x}$ is a length $M$ vector of independent standard Gaussian random variables.  
By Markov's inequality 
$$
\mathbb{P}(F({\bf x})-\mathbb{E}F({\bf x})\geq t)=\mathbb{P}(s(F({\bf z})-\mathbb{E}(F({\bf z}))\geq st)\leq e^{-st}\mathbb{E}e^{s(F-\mathbb{E}F)}
$$
Letting $s=t/2A^2$, we have 
$$
\mathbb{P}(F({\bf z})-\mathbb{E}F({\bf z})\geq t)\leq \exp\left (-\frac{t^2}{4A^2}\right)
$$
We can then apply the same inequality to $-F$ and we will have our result.
\end{proof}

It is worth noting that the method used to prove this result is quite important.  Talagrand \cite{Talagrandbook} refers to this method of proof as the "smart path method".  This method can be applied to a variety of problems.  Notice that we found a "path" (namely our function ${\bf f}$), which took us between the situation that we wanted to study and a simpler situation.  Beyond choosing an appropriate path, the only real work left to do was to get bounds on the derivatives along the path.  Talagrand points out that although this method leads to an elegant proof, the choice of path is highly important and nontrivial.  Often the choice is not obvious and can be found only after a careful study of the structure of the problem.
  
\section{Talagrand's Isoperimetric Inequality}
The concentration of measure inequality presented in this section is also due to Talagrand.  In \cite{Talagrand}, a theory of isoperimetric inequaliies on product spaces is developed.  The theorem presented here is just one of the many isoperimetric inequalities proved and applied in that work.  Once the necessary notions of distance are defined and the theorem proved, the applications of the theorem are vast and obtained quickly.  Before stating the theorem, we need to set up our product space and define a special notion of distance on the space.  

We will begin with a probability space which we will denote by $(\Omega, \mathcal{F}, P)$.  To give an idea of what we mean when we talk about a product probability space, we will give an example of the product of two probability spaces.  

Suppose that we have two probability spaces given by $(\Omega_1,\mathcal{F}_1, P_1)$ and $(\Omega_2, \mathcal{F}_2, P_2)$.  We want to form a product space which is the "product" of these two probability spaces.  For ease of notation, we will usually just denote the product space by $\Omega_1\times \Omega_2$, leaving the sigma algebras and the measures implicit.  Our new measure space is just the cross product $\Omega_1\times\Omega_2$.  The new sigma algebra is given by the tensor product $\mathcal{F}_1\otimes \mathcal{F}_2$.  We define the product measure $P_1\times P_2$ by $(P_1\times P_2)(F_1\times F_2)=P_1(F_1)P_2(F_2)$ for all $F_1\in\mathcal{F}_1$ and $F_2\in \mathcal{F}_2$.  We can then define a product of $n$ probability spaces by extending this notion.  

Given our probability space $(\Omega,\mathcal{F},P)$, we will be considering the product space $\Omega^n$.  Given $A\subseteq\Omega^n$, Talagrand's isoperimetric inequality gives us bounds on the measure of the set of points that are within a specified distance of this set $A$.  Before we can state the inequality, we need to develop a notion of distance. 

For $x\in\Omega^n$ and $A\subset\Omega^n$, we define Talagrand's convex distance to be 
$$
d_T(x,A)=\min\left\{ t \; : \;\forall \{\alpha_i\}, \;\exists y\in A \;\mathrm{such}\;\mathrm{that}\; \sum_{i=1}^n \alpha_i
\mathbbm{1}\{x_i\not = y_i\}\leq t\left(\sum_{i=1}^n \alpha_i ^2 \right)^{\frac{1}{2}}\right\}
$$
The similarity between $d_T(x,A)$ and Hamming's distance 
$$
d_H(x,A)=\inf_{y\in A} \sum_{i=1}^n \mathbbm{1}(x_i\not = y_i)
$$
should be noted.  Notice that if all $\alpha_i=n^{-1/2}$, then Talagrand's convex distance is always at least as large as $n^{-1/2}$ times the Hamming distance.  One of the main reasons that we use $d_T(x,A)$ instead of $d_H(x,A)$, is that $d_T(x,A)$ not only allows us to weight the summands differently, it allows us to choose weights that explicitly depend on the values of the $x_i$.  This flexibility allows the inequality to be applied to a much wider range of problems.

A second (and equivalent) way of defining Talagrand's convex distance is by 
\begin{equation}
\label{conv}
d_T(x,A)=\sup\left\{z_{\alpha}\; : \;z_{\alpha}=\inf_{y\in A}\sum_{i=1}^n \alpha_i(x)\mathbbm{1}(x_i\not = y_i) \;\mathrm{and}\; \sum_{i=1}^N\alpha_i^2(x)\leq 1\right\}
\end{equation} 

To gain a bit of understanding about the convex distance, let us look at a simple example.  Suppose that we are working in one dimension.  Let $x\in\mathbb{R}$ and let our set $A$ just be $\{y\}$, the set containing only the point $y\in\mathbb{R}$.  Then 
$$
d_T(x,\{y\})=\min\{t\geq 0\; : \; \forall \alpha\in\mathbb{R}_{>0}, \; \; \; \alpha\mathbbm{1}(x\not = y)\leq t\|\alpha\|\}
$$  
$$
=\left\{
     \begin{array}{lr}
       1 & \;\mathrm{if}\; y\not = x\\
       0 & \;\mathrm{if}\; y=x
     \end{array}
   \right.
$$

Given $A\subset\omega^n$, we define 
$$
A_t=\{x\in\Omega^n \; : \;d_T(x,A)\leq t\}
$$
In other words, $A_t$ is the set of all points that are within a distance $t$ of $A$.  The following inequality can be found in \cite{Talagrand} and tells us that for a set $A$ of "reasonable size", $P(A_t)$ is close to $1$.  
\begin{thm}
\label{tal}
For every $A\subset\Omega^n$, we have 
\be
\int_{\Omega^n}\exp\left(\frac{1}{4}d_T^2(x,A)\right)dP(x)\leq \frac{1}{P(A)}
\ee
and consequently
\be
P(d_T(x,A)>t)\leq \frac{e^{-t^2/4}}{P(A)}
\ee
and 
\be
P(A_t)\geq 1-\frac{e^{-t^2/4}}{P(A)}
\ee
\end{thm}
It should be noted that the proof given here more closely follows the proof as given in \cite{Steele} as opposed to \cite{Talagrand}.  The method is basically the same as Talagrand's original proof although some components are presently slightly differently and appear in a different order. 

Before we can begin the proof of the theorem, we need a deeper understanding and a different characterization of the convex distance.  We will begin by defining a set $U_A(x)$.  Elements of this set will be elements of $\mathbb{R}^n$ containing only $0$'s and $1$'s.  We will begin with the set $U'_A(x)$, which is the set of all vectors $u_y=(\mathbbm{1}(x_1\not= y_1), \mathbbm{1}(x_2\not = y_2),\dots ,\mathbbm{1}(x_n\not = y_n))$ for $y\in A$.  We then let $U_A(x)$ be the set which includes all of these vectors in addition to all vectors we can obtain from $U'_A(x)$ by switching some of the $0$'s to $1$'s.  In other words, $u\in U_A(x)$ if and only if $u-u_y\geq 0$ for all $y\in A$.  We then define the set $V_A(x)$ to be the convex hull of $U_A(x)$.  By convex hull, we mean the set of all convex combinations of vectors in $U_A(x)$.  We then have the following dual characterization of $d_T(x,A)$.
\begin{prop}
$$
d_T(x,A)=\min\{\|v\|_2: v\in V_A(x)\}
$$
\end{prop}
\begin{proof}
Begin with 
$$
\min_{y\in A}\sum_{i=1}^n \alpha_i\mathbbm{1}(x_i\not = y_i)
$$
Using the definition of $U_A(x)$, this is 
$$
\min_{u\in U_A(x)} \sum_{i=1}^n\alpha_i u_i
$$
Using the fact that the minimum of a linear functional on a convex set is equal to the minimum over the set of extreme points, we have that the above is equal to  
$$
\min_{v\in V_A(x)}\sum_{i=1}^n\alpha_i v_i
$$
Next, we apply the Cauchy-Schwarz inequality to get
$$
\leq \min_{v\in V_A(x)}\left\{\sum_{i=1}^n\alpha_i^2\right\}^{1/2}\left\{\sum_{i=1}^n v_i\right\}^{1/2}
$$
$$
=\left\{\sum_{i=1}^n\alpha_i ^2\right\} \min \{\|v\|_2: v\in V_A(x)\}
$$
Recalling Talagrand's convex distance, 
$$
d_T(x,A)=\min\left\{ t \; : \;\forall \{\alpha_i\}, \;\exists y\in A \;\mathrm{such}\;\mathrm{that}\; \sum_{i=1}^n \alpha_i
\mathbbm{1}\{x_i\not = y_i\}\leq t\left(\sum_{i=1}^n \alpha_i ^2 \right)^{\frac{1}{2}}\right\}
$$
we immediately have 
$$
d_T(x,A)\leq\min \{\|v\|_2\; : \;v\in V_A(x)\}
$$
by our last inequality.

Now we need to prove the reverse inequality.  
By the linear functional characterization of the Euclidean norm, there is an $\alpha$ with $\|\alpha\|_2=1$ such that for all $v\in V_A(x)$, we have 
$$
\sum_{i=1}^n\alpha_i v_i \geq \min \{\|v\|_2\; : \;v\in V_A(x)\}
$$
By definition of $V_A(x)$, this implies that for all $y\in A$, we have 
$$
\sum_{i=1}^n\alpha_i \mathbbm{1}(x_i\not = y_i)\geq \min \{\|v\|_2\; : \;v\in V_A(x)\}
$$
Using equation (2.2), this immediately applies the reverse inequality

\end{proof}
Now that this result is established, we can prove Talagrand's isoperimetric inequality.
\begin{proof}{(\em of Theorem 2.2.1)}
To prove the theorem, we will use induction on the dimension of the product space.  To prove the base case, we will start with $n=1$.  In this case we have 
$$
d_T(x,A)=\min \{\|v\|_2\; : \;v\in V_A(x)\}
$$
which is 
$$
\left\{
     \begin{array}{lr}
       1 & : x \notin A\\
       0 & : x \in A
     \end{array}
   \right.
$$
Plugging this into the integral from Talagrand's theorem, we have 
$$
\int_{\Omega}\exp\left(\frac{1}{4}d_T^2(x,A)\right)dP(x)=e^0\int_A dP(x)+e^{1/4}\int_{A^c}dP(x)
$$
$$
=P(A)+e^{1/4}(1-P(A))
$$
If we can show that this quantity is $\leq 1/P(A)$, the base case will be proved.  This will be relatively easy to show, using a little bit of calculus.  For ease of notation, let $p=P(A)$.  Then (multiplying on both sides of the equation by $p$), we need to prove that 
$$
p^2+e^{1/4}p(1-p)\leq 1
$$
To prove this, we will take derivatives to find the $p$ that maximizes $p^2+e^{1/4}p(1-p)-1$.
Taking the derivative and solving for $p$ gives $p\approx 2.26$.  Hence, on the interval $[0,1]$, $p^2+e^{1/4}p(1-p)$ obtains its maximum at $1$.  At $p=1$, the inequality is satisfied, so the base case is proved.  

Now we will proceed with the inductive step.  Assume that for any $A\subseteq\Omega^n$, we have 
$$
\int_{\Omega^n}\exp\left(\frac{1}{4}d_T(x,A)^2\right)dP(x)\leq \frac{1}{P(A)}
$$
where $P(A)$ now represents our product measure on $\Omega^n$.  We need to check and make sure that the inequality holds for dimension $n+1$. We start with an arbitrary $A\subseteq \Omega^{n+1}$.  We will begin by writing $\Omega^{n+1}$ as $\Omega^n\times \Omega$.  Let $x\in\Omega^n$ and $\omega\in \Omega$.  Then $(x,\omega)\in\Omega^{n+1}$.  We will consider two different sets.  Following Steele, we will define the following as the $\omega$ section of $A$, given by 
$$
A(\omega)=\{x \; :\;(x,\omega)\in A\}\subset\Omega^n
$$
and the projection of $A$, given by 
$$
B=\bigcup_{\omega\in\Omega}A(\omega)=\{x\; :\; \exists  (x,\omega)\in A\}\subset\Omega^n
$$
To prove the theorem, we will show that $d_T(x,A)$ (in $n+1$ dimensions) can be bounded in terms of the convex distances for the $\omega$ sections and the projections.  To do this, we prove the following lemma.
\begin{lem}
For all $0\leq t\leq 1$ and $A(\omega)$ and $B$ as defined above, we have
\be 
d_T^2((x,\omega), A)\leq t(d_T^2(x,A(\omega)))+(1-t)(d_T(x,B))^2 +(1-t)^2
\ee
\end{lem}
Using the alternative characterization of $d_T$ from Proposition 2.2.2, we can find vectors $v_1\in V_{A(\omega)}(x)$ and $v_2\in V_B(x)$ such that $d_T(x,A(\omega)=\|v_1\|$ and $d_T(x,B)=\|v_2\|$.  

First we note that $(v_1,0)\in V_A(x,\omega)$.  To see this, use the fact that since $v_1\in V_{A(\omega)}(x)$, we know that $v_1^{(i)}$ (the $i$th component of $v_1$) is equal to $\mathbbm{1}(x_i\not = y_i)$ for some $y\in A(\omega)$ for all $i$.  Then, if $y\in A(\omega)$, we know that $(y,\omega)\in A$.  Hence, $v_1^{n+1}=\mathbbm{1}(\omega\not =\omega)=0$.  

Next, we note that $(v_2,1)\in V_A(x,\omega)$. This follows immediately from the fact that starting from a vector in $U_A'(x,\omega)$, we can always change $0$'s to $1$'s and remain in $U_A(x,\omega)$, and hence in $V_A(x,\omega)$.

Since $V_A(x,\Omega)$ is convex, 
$$
t(v_1,0)+(1-t)(v_2,1)=(tv_1+(1-t)v_2,1-t)
$$
is also in $V_A(x,\Omega)$.  Notice that by our alternative characterization of $d_T$, $\|(tv_1+(1-t)v_2,1-t)\|$ is an upper bound on $d_T((x,\omega),A)$.  
$$
\|(tv_1+(1-t)v_2,1-t)\|^2=\sum_{i=1}^n(tv_1^{(i)}+(1-t)v_2^{(i)})^2+(1-t)^2
$$
$$
\leq t\|v_1\|_2^2+(1-t)\|v_2\|_2^2+(1-t)^2
$$
Using $d_T(x,A(\omega))=\|v_1\|$ and $d_T(x,B)=\|v_2\|$, we have proved the lemma.  

Keeping $\omega$ fixed, define $I_n(\omega)$ to be the $n$-fold integral given by 
$$
I_n(\omega)=\int_{\Omega^n}\exp\left(\frac{1}{4} d_T^2((x,\omega),A)\right)dP(x)
$$
Using Lemma 2.2.3, we have that 
$$
I_n(\omega)\leq \int_{\Omega^n}\exp\left(\frac{1}{4}(t d_T^2(x,A)+(1-t)d_T^2(x,B)+(1-t)^2)\right)dP(x)
$$
\be
=\exp\left(\frac{(1-t)^2}{4}\right) \int_{\Omega^n}\exp\left(\frac{t}{4}d_T^2(x,A)\right)\exp\left(\frac{1-t}{4} d_T^2(x,B)\right) dP(x)
\ee
Recall that Holder's inequality says that 
$$
\int |fg| \; d\mu\leq \left(\int |f|^p\;d\mu\right)^{1/p}\left(\int|g|^p\;d\mu\right)^{1/q}
$$
Applying this to (2.7) with $p=1/t$ and $q=1/(1-t)$ gives 
$$
I_n(\omega)\leq \exp\left(\frac{(1-t)^2}{4}\right)\left(\int_{\Omega^n}\exp\left(\frac{1}{4}d_T^2(x,A)\right)dP(x)\right)^t\left(\int_{\Omega_k}\exp\left(\frac{1}{4}d_T^2(x,B)\right)dP(x)\right)^{1-t}
$$
Now, by the induction hypothesis, we have 
$$
I_n(\omega)\leq \exp\left(\frac{(1-t)^2}{4}\right)\left(\frac{1}{P(A(\omega))}\right)^t\left(\frac{1}{P(B)}\right)^{1-t}
$$
\be
=\frac{1}{P(B)}\left(\frac{P(A(\omega))}{P(B)}\right)^{-t}\exp\left(\frac{(1-t)^2}{4}\right)
\ee
Since $A(\omega)\subset A(B)$, we know that 
$P(A(\omega))\leq P(B)$.  In order to complete the proof of the theorem, we need the following lemma.
\begin{lem}
For all $0\leq r\leq 1$, we have 
$$
\inf_{0\leq t\leq 1}r^{-t}\exp\left(\frac{1}{4}(1-t)^2\right)\leq 2-r
$$
\end{lem}
The proof of this lemma is essentially a calculus exercise, but we will give an outline.  Taking the derivative of $r^{-t}\exp\left(\frac{1}{4}(1-t)^2\right)$ with respect to $t$ gives 
$$
-\exp \left(\frac{1}{4}(1-t)^2\right)r^{-t}(\ln(r)+1/2(1-t)
$$
Optimizing in $t$ gives $t=1+2\ln(r)$, which can be shown to be a minimum.  Plugging back into $r^{-t}\exp\left(\frac{1}{4}(1-t)^2\right)$ gives 
$$
r^{-1+2\ln(r)}\exp(1/4(1-1-ln(r))^2)
$$
After some simplification, we get that the above
$$
=r^{2\ln(r)+1}
$$
To show that $r^{2\ln(r)+1}\leq 2-r$, we just need to show that $r+r^{2\ln(r)+1}\leq 2$.  Using calculus, one can show that $r+r^{2\ln(r)+1}$ is decreasing on $0\leq r\leq 1$ and therefore, the inequality is true.  This concludes the proof of the lemma.

Applying this lemma to equation (2.8) gives 
$$
I_n(\omega)\leq \frac{1}{P(B)}\left(2-\frac{P(A(\omega))}{P(B)}\right)
$$
This can now be integrated with respect to $\omega$, which gives 
$$
\int_{\Omega^{n+1}}\exp\left(\frac{1}{4}d_T((x,\omega),A)\right)dP(x)dP(\omega)\leq\frac{1}{P(B)}\left(2-\frac{P(A)}{P(B)}\right)
$$
$$
=\frac{1}{P(A)}\frac{P(A)}{P(B)}\left(2-\frac{P(A)}{P(B)}\right)
$$
Notice that if we can prove that 
$$
\frac{P(A)}{P(B)}\left(2-\frac{P(A)}{P(B)}\right)\leq 1
$$
then our proof will be complete.  Letting $x=\frac{P(A)}{P(B)}$, we want to determine for which $x$, $x(2-x)\leq 1$, or equivalently, for which $x$, $x^2-2x+1\geq 0$.  Since this factors as $(x-1)^2$, this inequality is true for all $x$, which completes the proof.  
\end{proof}
As previously mentioned, although the setup and proof of this theorem took a fair amount of work, most applications of the theorem are elegant and quick.  See \cite{Steele} for a discussion and explanation of common applications.  In addition, we will see an application in a later chapter.  

\section{Ledoux's Concentration of Measure on Reversible Markov Chains} 
A concentration of measure result proved by Ledoux \cite{LedouxCOM} turns out to be a key foundational piece for some of the results of chapter 3.  Before stating the result, we provide a few definitions.  Following the notation in \cite{LedouxCOM}, we will let ($\Pi,\mu$) denote a Markov chain on a finite or countable set $X$.  A Markov chain is a stochastic process which moves between elements of $X$ according to the following rules:  if the chain is at a given $x\in X$, the next position in the chain is chosen according to a fixed probability distribution $\Pi(x,\cdot)$.  In other words, given a starting position $x\in X$, the probability to move from $x$ to $y$ is $\Pi(x,y)$.  We call $X$ the state space and $\Pi$ the transition matrix.  Markov chains satisfy a "memoryless" property.  This property (called the Markov property) is stated in mathematical terms as follows.  

For notational purposes, let $S=(\Pi,\mu)$, so that $S_t$ is the current state of the chain at some discrete time $t>0$.  Then for all $x,y\in X$ and events $K_{t-1}=\cap_{i=0}^{t-1}\{S_i=x_i\}$ satisfying $\mathbb{P}(K_{t-1}\cap \{S_t=x\})>0$, we have 
$$
\mathbb{P}(S_{t+1}=y \; | \; K_{t-1}\cap\{S_t=x\})=\mathbb{P}(S_{t+1}=y \; | \; S_t=x)=\Pi(x,y)
$$
A simple explanation of this property is that the future depends only on the present, not on the past.  For a complete discussion of Markov chains and more properties, see \cite{LPW}.  

Furthermore, for this application, we require that the Markov chain be irreducible.  A Markov chain is irreducible if for any $x,y\in X$, there exists an integer $t>0$, such that $\Pi^t(x,y)>0$.  By the notation $\Pi^t(x,y)$, we mean that $\mathbb{P}(S_t=y \: | \: S_0=x )$.  In other words, there is a positive probability of going from any state to any other state.  

A probability measure $\mu$ on $X$ is called an invariant (or stationary) measure if 
$$
\sum_{x\in X} \mu(x)\Pi(x,y)=\mu(y)
$$
for all $y\in X$.  Regarding $\Pi$ as a matrix (where $\Pi(x,y)$ is the $(i,j)$th entry) and $\mu$ as a vector, this is equivalent to the condition 
$$
\mu=\mu\Pi
$$
which perhaps gives a more intuitive idea of the measure.  This is the $\mu$ that we will refer to in the notation $(\Pi,\mu)$ for the Markov chain.  

A Markov chain is reversible if 
$$
\mu(x)\Pi(x,y)=\mu(y)\Pi(y,x)$$ for all $x,y\in X$.  This is often called the detailed balance condition.  

From now on, we will assume that $(\Pi, \mu)$ is a reversible Markov chain with transition matrix $\Pi$ and invariant measure $\mu$.  For functions $f$ and $g$ on $X$, the Dirichlet form associated to $(\Pi,\mu)$ is given by 
$$
\mathcal{E}(f,g):=\langle (I-\Pi)f,g\rangle_{\mu}
$$
In particular, 
$$
\mathcal{E}(f,f)=\frac{1}{2}\sum_{x,y\in X}[f(x)-f(y)]^2\mu(x)\Pi(x,y)
$$
For a proof of this fact, see \cite{LPW}.

We will notate the eigenvalues of $\Pi$ in decreasing order by 
$$
1=\eta_1>\eta_2\geq\dots\geq \eta_{|X|}\geq -1
$$
Notice that $1$ must be an eigenvalue of $\Pi$, since (letting $\bf{e}$ temporarily represent a vector of all $1$'s) 
$$
\Pi\bf{e}=\bf{e}
$$
using the fact that the sum along each row and column of $\Pi$ is $1$.  To see why the rest of the eigenvalues must have magnitude less than $1$, note that if $|\eta_k|>1$ for some $k$, then $\Pi^nv_k=\eta_k^n v_k$ for an eigenvector $v_k$.  If $|\eta_k|^n$ is large, this contradicts the fact that all entries of $\Pi$ are between $0$ and $1$.  

The spectral gap of the Markov chain is defined by $\lambda_1:=1-\eta_2$.  The spectral gap relates to the Dirichlet form via the Poincare inequality, which says that for all functions $f$ on $X$,
$$
\lambda_1 \mathrm{Var}_{\mu}(f)\leq \mathcal{E}(f,f)
$$
In order to work with Ledoux's concentration of measure result, we need to define a triple norm on functions $f$ on $X$.  Let 
$$
\||f\||_{\infty}^2=\frac{1}{2}\sup_{x\in X}\sum_{y\in X}|f(x)-f(y)|^2\Pi(x,y)
$$
We are now in a position to state Ledoux's concentration of measure result on Markov chains \cite{LedouxCOM}.
\begin{thm}
Let $(\Pi,\mu)$ be a reversible Markov chain on $X$ with a spectral gap given by $\lambda_1>0$.  Then, whenever $\||F\||_{\infty}\leq 1$, $F$ is integrable with respect to $\mu$ and for every $r\geq 0$, 
$$
\mu\left (\{F\geq\int Fd\mu+r\}\right)\leq 3e^{-r\sqrt{\lambda_1/2}}$$
\end{thm}
\begin{proof}
We will begin by assuming that $F$ is a bounded function on $X$ with mean $0$ and $\||F\||_{\infty}\leq 1$.  We will let 
$$
\Lambda(\lambda)=\int e^{\lambda F}d\mu
$$
for $\lambda>0$.  By definition, 
$$
\mathcal{E}( e^{\lambda F/2},  e^{\lambda F/2})=\frac{1}{2}\sum_{x,y\in X}[ e^{\lambda F(x)/2}-e^{\lambda F(y)/2}]^2\Pi(x,y)\mu(\{x\})
$$
which, by symmetry is equal to 
$$
\sum_{F(y)<F(x)}[ e^{\lambda F(x)/2}-e^{\lambda F(y)/2}]^2\Pi(x,y)\mu(\{x\})
$$
$$
=\sum_{F(y)<F(x)}e^{\lambda F(x)}+e^{\lambda F(y)}-2e^{\lambda/2(F(x)+F(y))}\Pi(x,y)\mu(\{x\})
$$
Using the fact that $F(y)<F(x)$ in the region we are summing over, the above is 
$$
\leq \sum_{F(y)<F(x)} 2e^{\lambda F(x)}-2e^{\lambda/2(F(x)+F(y)}\Pi(x,y)\mu(\{x\})
$$
$$
=\sum_{F(y)<F(x)}2e^{\lambda F(x)}(1-e^{\lambda/2(F(y)-F(x))})\Pi(x,y)\mu(\{x\})
$$
Taylor expanding the exponential to second order gives 
$$
\leq \sum_{F(y)<F(x)} 2e^{\lambda F(x)}(1-(1+\lambda/2(F(y)-F(x))+(\lambda/2 F(y)-\lambda/2F(x))^2/2 \Pi(x,y)\mu(\{x\})
$$
The first order term cancels by symmetry once we go back to summing over the whole $x,y\in X$, leaving us with 
$$
\sum_{x,y\in X} e^{\lambda F(x)}(\lambda/2 F(y)-\lambda/2F(x))^2\Pi(x,y)\mu(\{x\})
$$
$$
=\frac{\lambda^2}{2} \||F\||_{\infty}^2 \int e^{\lambda F}d\mu
$$
so we have showed that 
$$
\mathcal{E}(e^{\lambda F/2},e^{\lambda F/2})\leq \||F\||_{\infty}^2\int e^{\lambda F}d\mu 
$$
The Poincare inequality says that 
$$
\lambda_1\mathrm{Var}_{\mu}(f)\leq \mathcal{E}(f,f)
$$
Using this, and the fact that $$\lambda_1\mathrm{Var}(e^{\lambda F/2})=\lambda_1(\Lambda(\lambda)-\Lambda^2(\lambda/2))$$
we have that 
$$
\lambda_1(\Lambda(\lambda)-\Lambda^2(\lambda/2))\leq \lambda^2 \||F\||_{\infty}^2\Lambda(\lambda)
$$
Recalling that $\||F\||_{\infty}^2\leq 1$ by assumption, we have the inequality 
$$
\Lambda(\lambda)-\Lambda\left(\frac{\lambda}{2}\right)^2\leq \frac{\lambda^2}{\lambda_1}
\Lambda(\lambda)
$$
Solving for $\Lambda(\lambda)$ gives 
$$
\Lambda(\lambda)\leq \frac{1}{1-\frac{\lambda^2}{\lambda_1}}\Lambda\left(\frac{\lambda}{2}\right)^2
$$
Now we use the same inequality on the $\Lambda\left(\frac{\lambda}{2}\right)$ term and iterate $n$ times, leaving us with 
$$
\Lambda(\lambda)\leq \prod_{k=0}^{n-1}\left(\frac{1}{1-\frac{\lambda^2}{4^k\lambda_1}}\right)^{2^k}\Lambda\left(\frac{\lambda}{2^n}\right)^{2^n}
$$
We now let $n\rightarrow\infty$.  The product will converge provided that $\lambda<\sqrt{\lambda_1}$.  This assumption does not cause any problems as we only required that $\lambda$ be nonnegative.  Recall that $\Lambda(\lambda)=\int e^{\lambda F} d\mu$ and that $F$ is bounded, so $\Lambda(\lambda)=1+o(\lambda)$.  This gives $\Lambda(\lambda/2^n)^{2^n}\rightarrow 1$ as $n\rightarrow\infty$.  Hence we are left with 
$$
\Lambda(\lambda)\leq \prod_{k=0}^{\infty} \left(\frac{1}{1-\frac{\lambda^2}{4^k\lambda_1}}\right)^{2^k}
$$
If we now set $\lambda=\frac{1}{2}\sqrt{\lambda_1}$, we have 
$$
\lambda\left(\frac{\sqrt{\lambda_1}}{2}\right)\leq \prod_{k=0}^{\infty}\left(\frac{1}{1-\frac{1}{4^{k+1}}}\right)\leq 3
$$
Recall this tells us that 
$$
\int e^{\lambda F}d\mu\leq 3 
$$
Markov's inequality states that for a nonnegative integrable random variable $X$ and and $r>0$, 
$$
\mathbb{P}(X>a)\leq \frac{\mathbb{E}(X)}{a}
$$
Applying this to the above equation, we have 
$$
\mathbb{P}(e^{\sqrt\frac{\lambda_1F}{2}}>e^r)\leq 3/e^r
$$
so 
$$
\mathbb{P}(\frac{\sqrt{\lambda_1}F}{2}> r)\leq 3/e^r
$$
giving 
\begin{equation}
\mathbb{P}(F>r)\leq 3e^{-\frac{\sqrt{\lambda_1}}{2}r}
\end{equation}
This is essentially the result we wanted to prove, except that to begin with, we assumed that $F$ was a mean zero function and bounded.  To get rid of the mean $0$ condition, we can simply replace $F$ in the beginning of the proof with $F'=F-\mathbb{E}(F)$, giving us a mean zero function and our desired result.  To relax the boundedness condition, we approximate $F$ by $F_n=\min(|F|, n)$.  Notice this still satisfies $\||F\||_{\infty}^2\leq 1$. 
Choose an $m$ such that $\mathbb{P}(|F|\leq m)\geq 1/2$ for all $n$ and an $r$ such that $3e^{-r\sqrt{\lambda_1}/2}<1/2$.  Since 
$$
\mathbb{P}(F_n>r)\leq 3e^{-\frac{\sqrt{\lambda_1}}{2}r}
$$
 we must have 
$$
\int F_n d\mu\leq m+r
$$
Then, by the monotone convergence theorem, we have
$$
\int |F|d\mu <\infty
$$
We can then apply (2.1) to $\min(\max(F,-n),n)$ and let $n\rightarrow\infty$ to get the final result.  

\end{proof}
In chapter 3, we will see multiple ways that this theorem can be applied to get concentration of measure results for a variety of interesting quantities merely by choosing an appropriate Markov chain with a known spectral gap.  
\section{The Euler-Maclaurin Formula}
The Euler-Maclaurin formula is a formula which enables us to make a connection between a sum and its corresponding integral, provided the function is sufficiently smooth.  Before we can state the formula, we need a few preliminary definitions and notation.  Let $\lfloor x\rfloor$ denote the greatest integer function, so that $\lfloor x\rfloor$ returns the greatest integer less than or equal to $x$.  For $s=1,2,\dots$, let $B_s(x)$ denote the Bernoulli polynomials.  The generating function for the Bernoulli polynomials is as follows:
$$
\frac{te^{xt}}{e^t-1}=\sum_{s=0}^{\infty} B_s(x)\frac{t^s}{s!}
$$
For $s\geq 1$, we will let $B_s:=B_s(0)$.  These $B_s$ are called the Bernoulli numbers.  The first few Bernouli numbers are given by $ B_1 = -1⁄2,\; B_2 = 1⁄6,\; B_3 = 0 ,\; B_4 = −1⁄30, \;B_5 = 0\;, B_6 = 1⁄42$.  See \cite{AndrewsAskeyRoy} for more details and alternative definitions.  We now have everything that we need to state the Euler-Maclaurin formula.
\begin{thm}
Suppose $f$ has continuous derivatives up to order $s$.  Then 
\begin{multline}
\sum_{m+1}^n f(x)=\int_m^n f(x) dx+\sum_{i=1}^s(-1)^i\frac{B_i}{i!}(f^{(i-1)}(n)-f^{(i-1)}(m)) \\ +\frac{(-1)^{i-1}}{i!}\int_m^n B_s(x-\lfloor x\rfloor)f^{(i)}(x) dx
\end{multline}
\end{thm}
Notice that this formula allows a sum to be estimated by its corresponding integral (or an integral by its sum), and gives an exact formula for the error in using this estimation.  In many applications, this error term can at least be bounded, if not computed exactly. The proof of the formula involves successively performing integration by parts, which gives a sequence of periodic functions relating to the Bernoulli polynomials.  For a proof, see \cite{AndrewsAskeyRoy}.   We will use a similar method of proof to prove a $q$-deformed version of Stirling's formula in a later section.  
\chapter{Random Operator Compressions}
\section{Background for First Result}
In a recent work of Chatterjee and Ledoux on concentration of measure for random submatrices \cite{ChatLed}, it is proved that for an arbitrary Hermitian matrix of order $n$ and $k\leq n$ sufficiently large, the distribution of eigenvalues is almost the same for any principal submatrix of order $k$.  Their proof uses the random transposition walk on the symmetric group $S_n$ and concentration of measure techniques.  To further generalize their results, we observe that it is important to use a Markov chain which does not change too many matrix entries all at once and whose spectral gap is known. Instead of looking at a Markov chain on $S_n$, we first consider a Markov chain on the special orthogonal group $SO(n)$.  $SO(n)$ is the group of $n\times n$ orthogonal matrices with determinant $1$.  As a linear transformation, every element of $SO(n)$ is a rotation and preserves distances.   We introduce Kac's walk on $SO(n)$ and demonstrate that it is sufficiently similar to the transposition Markov chain to allow for Chatterjee and Ledoux's results to carry over to the more general case of operator compressions.  It should be noted that a similar result has been proved by Meckes and Meckes \cite {MeckesMeckes} using different techniques.  In a more recent work \cite{MeckesMeckes2}, Meckes and Meckes have extended their techniques to include several other classes of random matrices and prove almost sure convergence of the empirical spectral measure.   The purpose of this paper is to highlight the fact that the methods of Chatterjee and Ledoux can be extended to include more general cases, provided the Markov chain used satisfies appropriate conditions.  To emphasize this point, we also apply the method to get a concentration of measure result for a compression by a matrix of Gaussians using Kac's walk coupled to a thermostat. We also show an application of this method applied to the length of the longest increasing subsequence of a random walk evolving under the asymmetric exclusion process.  The results of this section can be found in \cite{SNandMW}.\\ \\ Following the notation of Chatterjee and Ledoux, for a given Hermitian matrix $A$ of order $n$, with eigenvalues given by $\lambda_1,\dots,\lambda_n$, we let $F_A$ denote the empirical spectral distribution function of $A$.  This is defined as
$$
F_A(x):=\frac{\#\{i:\lambda_i\leq x\}}{n}
$$

\section{Kac's Walk on $SO(n)$}
The following model, introduced by Kac \cite{Kac}, describes a system of particles evolving under a random collision mechanism such that the total energy of the system is conserved.  Given a system of $n$ particles in one dimension, the state of the system is specified by $\vec{v}=(v_1,\dots v_n)$, the velocities of the particles.  At a time step $t$, $i$ and $j$ are chosen uniformly at random from $\{1,\dots, n\}$ and $\theta$ is chosen uniformly at random on $(-\pi, \pi]$.  The $i$ and $j$ correspond to a collision between particles $i$ and $j$ such that the energy,
$$
E=\sum_{k=1}^n v_k^2
$$
is conserved.  Under this constraint, after a collision, the new velocities will be of the form $v_i^{\mathrm{new}}=v_i\cos(\theta)+v_j\sin(\theta)$ and $v_j^{\mathrm{new}}=v_j\cos(\theta)-v_i\sin(\theta)$.    
For $i<j$, let $R_{ij}(\theta)$ be the rotation matrix given by:
\begin{equation*}
	R_{ij} (\theta) = \begin{pmatrix} I & & & & \\
		&\cos(\theta) &&\sin(\theta)& \\
		&&I&& \\
		&-\sin(\theta) &&\cos(\theta)& \\
		&&&& I \end{pmatrix}
\end{equation*}
where the $\cos(\theta)$ and $\sin(\theta)$ terms are in the rows and columns labeled $i$ and $j$, and the $I$ denote identity matrices of different sizes (possibly 0). We will use the convention that $R_{ii}{\theta}=I$.  After one step of the process, $\vec{v}_{new}=R_{ij}(\theta)\vec{v}$.  

In our case, we will be considering this process acting on $SO(n)$, so instead of vectors in $\mathbb{R}^n$, our states will be given by matrices $G\in SO(n)$.  Then we can define the one-step Markov transition operator for Kac's walk, $Q$, on continuous functions of $SO(n)$:
\begin{equation} 
	Qf(G) = \frac{1}{\binom{n}{2}} \sum_{i<j} \int_0^{2\pi} f(R_{ij}(\theta)G)\frac{1}{2\pi} d\theta
\end{equation}
for any $G\in SO(n)$, and where $f$ is a continuous function on $SO(n)$. Notice that this is a slightly different setup than we introduced in Chapter 2.  Instead of a finite state space Markov chain, we now have an infinite state space.  We will pause to discuss the differences between our previous case and this case. Since our state space is infinite, we cannot define our transition probabilities using finite dimensional matrices.  We instead define a Markov transition operator on continuous functions of our space.  In the context of our earlier discussion from before, 
$$
Qf(G)=\mathbb{E}(f(X_1)\; | \; X_0=G)
$$
In other words, $Qf(G)$ gives us the expected value after one step of the chain, conditioned on the fact that we start at $G\in SO(n)$.  It turns out that this fully specifies our Markov chain.  In order to generalize the methods of Chatterjee and Ledoux to this case, we need to know the invariant distribution and the spectral gap of  Kac's walk.  This is given in the following result.

\begin{thm}[\cite{CarlenCarvalhoLoss,Maslen}]
	Kac's walk on $SO(n)$ is ergodic and its invariant distribution is the uniform distribution on $SO(n)$. Furthermore, the spectral gap of Kac's walk on $SO(n)$ is $\frac{n+2}{2(n-1)n}$.
\end{thm}
Using our Markov transition operator, we can define the Dirichlet form, $\Dform{\cdot}{\cdot}$. As discussed in chapter 2, it is well known that for a Markov chain with spectral gap, $\lambda_1$, the Poincare inequality holds:
\begin{equation*}
	\lambda_1 \mathrm{Var}(f) \leq \Dform{f}{f}.
\end{equation*}
For the Kac's walk, we have 
\begin{equation*}
	\Dform{f}{f} = \frac 1 {2\binom{n}{2}} \sum_{1 \leq i<j\leq n} \int_0^{2\pi} \frac{1}{2\pi} \int_{SO(n)} \left(f(G) - f(R_{ij}(\theta)G)\right)^2 d\mu_n(G) d\theta,
\end{equation*}
where $\mu_n$ is the Haar measure on $SO(n)$ normalized so that the total measure is $1$.

Let us define the triple norm:
\begin{equation}
	\tripNorm{f}^2 = \frac 1 {2\binom{n}{2}} \sup_{G\in SO(n)}\sum_{1 \leq i<j\leq n} \int_0^{2\pi} \frac{1}{2\pi} \left|f(G) - f(R_{ij}(\theta)G)\right|^2 d\theta.
\end{equation}
The following result is analogous to Theorem 3.3 from Ledoux's Concentration of Measure Phenomenon book \cite{LedouxCOM} (discussed and proved in chapter 2) . We reproduce the proof of the theorem here to verify that even though our situation does not satisfy the conditions of the theorem, the exact same argument carries through for Kac's walk on $SO(n)$. We omit some details here as they are the same as the argument in chapter 2.  
\begin{thm}
	Consider Kac's walk on $SO(n)$, and let $F:SO(n)\to \reals$ be given such that $\tripNorm{F}\leq 1$. Then, $F$ is integrable with respect to $\mu_n$, and for every $r\geq 0$, 
	\begin{equation*}
		\mu_n\left (F \geq \int F d\mu_n + r\right) \leq 3 e^{-r\sqrt{\lambda_1}/2}
	\end{equation*}
	where $\lambda_1= \frac{n+2}{2(n-1)n}$ is the spectral gap of Kac's walk on $SO(n)$.
\end{thm}
\begin{proof}
	We first demonstrate that $\Dform{e^{\lambda F/2}}{e^{\lambda F/2}})\leq \frac{\lambda^2 \tripNorm{F}^2}{4} \int_{SO(n)} e^{\lambda F(G)} d\mu_n(G)$ by using symmetry (see chapter 2 for details).
	\begin{align*}
		\Dform{e^{\lambda F/2}}{e^{\lambda F/2}} &= \frac 1 {2\binom{n}{2}} \sum_{1 \leq i<j\leq n} \int_0^{2\pi} \frac{1}{2\pi} \int_{SO(n)} \left(e^{\lambda F(G)} - e^{\lambda F(R_{ij}(\theta)G)}\right)^2 d\mu_n(G) d\theta \\
		& =  \frac 1 {\binom{n}{2}} \sum_{1 \leq i<j\leq n} \int_0^{2\pi} \frac{1}{2\pi} \int_{F(G)>F(R_{ij}(\theta)G)} \left(e^{\lambda F(G)} - e^{\lambda F(R_{ij}(\theta)G)}\right)^2 d\mu_n(G) d\theta \\
		& \leq \frac{\lambda^2}{4} \frac 1 {2 \binom{n}{2}} \sum_{1 \leq i<j\leq n} \int_0^{2\pi} \frac{1}{2\pi} \int_{SO(n)} \left(F(G) - F(R_{ij}(\theta)G)\right)^2 e^{\lambda F(G)} d\mu_n(G) d\theta \\
		& = \frac{\lambda^2}{4} \tripNorm{F}^2 \int_{SO(n)} e^{\lambda F(G)} d\mu_n(G)
	\end{align*}
	Setting $\Lambda(\lambda) = \int_{SO(n)} e^{\lambda F(G)} d\mu_n(G)$, we combine this with the Poincare inequality to obtain
	\begin{equation*}
		\lambda_1 \mathrm{Var}(e^{\lambda F/2}) = \lambda_1 \left(\Lambda(\lambda)-\Lambda^2\left(\frac{\lambda}2\right)\right) \leq \Dform{e^{\lambda F/2}}{e^{\lambda F/2}} \leq \frac{\lambda^2}{4} \tripNorm{F}^2 \Lambda(\lambda).
	\end{equation*}
	Incorporating the assumption $\tripNorm{F}\leq1$ yields
	\begin{equation*}
		\Lambda(\lambda) \leq \frac{1}{1-\frac{\lambda^2}{4\lambda_1}} \Lambda^2(\lambda/2).
	\end{equation*}
	Iterating the inequality $n$ times gives
	\begin{equation*}
		\Lambda(\lambda) \leq \prod_{k=0}^{n-1} \left(\frac{1}{1-\frac{\lambda^2}{4^{k+1}\lambda_1}}\right)^{2^k} \Lambda^{2^n}(\lambda/2^n).
	\end{equation*}
	Since $\Lambda(\lambda) = 1+ o(\lambda)$, we see that $\Lambda^{2^n}(\lambda/2^n) \to 1$ as $n\to \infty$. This gives the upper bound
	\begin{equation*}
		\Lambda(\lambda)\leq \prod_{k=0}^\infty \left(\frac{1}{1-\frac{\lambda^2}{4^{k+1}\lambda_1}}\right)^{2^k}. 
	\end{equation*}
	By plugging in $\lambda=\sqrt{\lambda_1}$, using the crude estimate $\prod_{k=0}^\infty \left(\frac{1}{1-\frac{1}{4^{k+1}}}\right)^{2^k} < 3$, and applying Chebyshev's inequality (similarly to as in chapter 2), we obtain the result. 
\end{proof}

\section{Main Result}
Using these results, along with the method of Chatterjee and Ledoux, we are able to prove the following result:
\begin{thm}  Take any $1\leq k\leq n$ and an $n$-dimensional Hermitian matrix $G$.  Let $A$ be the $k\times k$ matrix consisting of the first $k$ rows and $k$ columns of the matrix obtained by conjugating $G$ by a rotation matrix $R^{\theta}_{ij}\in SO(n)$ chosen uniformly at random.  If we let $F$ be the expected spectral distribution of $A$, then for each $r>0$, 
\begin{equation*}
\mathbb{P}\left(\|F_A-F\|_{\infty}\geq \frac{1}{\sqrt{k}}+r\right)\leq 12\sqrt{k}\exp\left(-r\sqrt{\frac{k}{32}}\right)
\end{equation*}
\end{thm}
\begin{proof}
The proof of this theorem uses the method introduced by Chatterjee and Ledoux \cite {ChatLed} with appropriate changes made to apply to the situation we are considering.  \\ \\ 
Let $R_{ij}(\theta)\in\mathrm{SO}(n)$ and let $A$ be as stated above.  Note that since $A$ is a compression of a Hermitian operator, it will also be Hermitian.  Fix $x\in\mathbb{R}$.  Let $f(A):=F_A(x)$, where $F_A(x)$ is the empirical spectral distribution of $A$.  Let $Q$ be the transition operator as defined in (1) and let $\tripNorm{.}$ be as in (2).  Using Lemma 2.2 from \cite{Bai}, we know that for any two Hermitian matrices $A$ and $B$ of order $k$, 
\begin{equation*} \|F_A-F_B\|_{\infty}\leq\frac{\mathrm{rank}(A-B)}{k}
\end{equation*}
In our case, taking one step in Kac's walk is equivalent to rotation in a random plane by a random angle.  Hence $A$ and $R_{ij}^{\theta}A$ will differ in at most two rows and two columns, bounding the difference in rank by $2$, so
\begin{equation*}
\|f(A)-f(R_{ij}^{\theta}A)\|_{\infty}\leq\frac{2}{k}
\end{equation*}
Using (2), 
\begin{equation*}
\tripNorm{f}^2 = \frac 1 {2\binom{n}{2}} \sup_{A\in SO(n)}\sum_{1\leq i<j\leq n} \mathbb{E}[f(A)-f(R_{ij}^{\theta}A)]^2
\end{equation*}
\begin{equation*}
\leq \frac 1 {2} \left(\frac{2}{k}\right)^2\left(\frac{2k}{n}\right)=\frac{4}{kn}
\end{equation*} where the $\frac{2k}{n}$ comes from the probability that both $i$ and $j$ are greater than $k$, in which case, $A$ and $R^{\theta}_{ij}A$ will be the same.
From Theorems 2.1 and 2.2, we have that 
\begin{equation*}
\mathbb{P}(|F_A(x)-F(x)|\geq r)\leq 6\exp\left(-\frac{r}{2}\frac{\sqrt{\frac{1}{2}\frac{n+2}{(n-1)n}}}{\sqrt{\frac{4}{kn}}}\right)
\end{equation*}
\begin{equation*}
=6\exp\left(-r/2\sqrt{\frac{1}{8}\frac{k(n+2)}{n-1}}\right)\leq 6\exp\left(-r/2\sqrt{\frac{k}{8}}\right)
\end{equation*}
This is true for any $x$.  Now, if we let $F_A(x-):=\lim_{y\uparrow x}F_A(y)$, then we have $\mathbb{E}F_A(x-)=\lim_{y\uparrow x}F(y)=F(x-)$.  Hence, for $r>0$, 
\begin{equation*}
\mathbb{P}(|F_A(x-)-\mathbb{E}F_A(x-)|>r)\leq \lim_{y\uparrow x}\mathbb{P}(|F_A(y)-F(y)|>r)
\end{equation*}
\begin{equation*}
\leq 6\exp\left(-r/2\sqrt{\frac{k}{8}}\right)
\end{equation*}
This holds for all $r$, so we can replace $>$ by $\geq$.  Next we will fix $\ell\in\mathbb{Z}_{\geq 2}$.  For $1\leq i<\ell$, let
\begin{equation*}
t_i:=\inf\{x:F(x)\geq i/\ell\}
\end{equation*}
and $t_0=-\infty$, $t_\ell=\infty$.  Then for each $i$, $F(t_{i+1})-F(t_i)\leq 1/\ell$.  Let
$$
\triangle=(\max_{1\leq i<\ell}|F_A(t_i)-F(t_i)|)\wedge (\max_{1\leq i<\ell}|F_A(t_i-)-F(t_i-)|)
$$
Take any $x\in\mathbb{R}$.  Let $i$ be an index where $t_i\leq x<t_{i+1}$.  Then 
\begin{equation*}
F_A(x)\leq F_A(t_{i+1}-)\leq F(t_{i+1}-)+\triangle\leq F(x)+1/\ell +\triangle 
\end{equation*}
and 
\begin{equation*}
F_A(x)\geq F_A(t_i)\geq F(t_i)-\triangle\geq F(x)-1/\ell-\triangle
\end{equation*}
Using these two facts, we get that 
\begin{equation*}
\|F_A-F\|_{\infty}\leq 1/\ell+\triangle
\end{equation*}
Then for any $r>0$, we have 
\begin{equation*}
\mathbb{P}(\|F_A-F\|_{\infty}\geq 1/\ell+r)\leq 12(\ell-1)\exp\left(-r\sqrt{\frac{k}{32}}\right)
\end{equation*}
Letting $\ell=k^{1/2}+1$, we have
\begin{equation*}
\mathbb{P}(\|F_A-F\|_{\infty}\geq \frac{1}{\sqrt{k}}+r)\leq 12\sqrt{k}\exp\left(-r\sqrt{\frac{k}{32}}\right)
\end{equation*} which concludes the proof of our theorem.  
\end{proof}
 
\section{Kac's Model Coupled to a Thermostat}
Using a spectral gap result from \cite{BLV}, we are able to demonstrate the application of this method to a more complicated Markov chain.  In this system, the particles from Kac's system interact amongst themselves with a rate $\lambda$ and interact with a particle from a thermostat  with rate $\mu$.  The particles in the thermostat are Gaussian with variance $\frac{1}{\beta}$, so they have already reached equilibrium.  The Markov transition operator for Kac's walk is defined as in $(1)$ and the Markov transition operator for the thermostat is given by \begin{equation}
Rf(G)=\frac{1}{n}\sum_{j=1}^n \frac{1}{2\pi}\int_{0}^{2\pi}\int_{\mathbb{R}^n}
 \sqrt{\frac{\beta}{2\pi}}^ne^{-\frac{\beta}{2}\omega_{ij}^{*2}(\theta)}f(V_j(\theta, \omega)G)d\theta d\omega
\end{equation}
where $\omega=(\omega_1,\omega_2,\dots,\omega_n)$, $V_j(\theta, \omega)$ sends each element $g_{ij}$ in column $j$ to $g_{ij}cos(\theta)+\omega_i\sin(\theta)$ for $i=1$ to $n$ and $\omega_{ij}^*=-g_{ij}\sin(\theta)+\omega_i\cos(\theta)$. 
In \cite{BLV}, they consider the Markov chain acting on a vector.  We consider the Markov chain acting on a matrix by treating the matrix as $n$ independent vectors.  Using this adaption, the following theorem follows immediately from the results proved in \cite{BLV}.  
\begin{thm}  Kac's walk coupled to a thermostat has unique invariant measure given by $$\nu_n=\prod_{i,j} \sqrt{\frac{\beta}{2\pi}}e^{-\frac{\beta}{2}v_{ij}^2}$$ and has spectral gap $\frac{\mu}{2n}$  \end{thm} 
For the thermostat alone (letting $\lambda=0$), we can again prove a theorem analogous to Chatterjee and Ledoux's theorem 3.3.  Let $\mathcal{G}$ be the set of $n\times n$ matrices with independent and identically distributed $\mathcal{N}(0,1/\beta)$ entries.  We can define the Dirichlet form and the triple norm for the thermostat as 
\begin{equation*}
\mathcal{Q}(f,f)=\frac{1}{2n}\sum_{j=1}^n\frac{1}{2\pi}\int_0^{2\pi}\int_{\mathbb{R}^n}\int_{G\in\mathcal{G}}\left(\frac{\beta}{2\pi}\right)^{n/2}e^{-\frac{\beta}{2}w_{ij}^{*2}}(f(V_j(\theta,w))G-f(G))d\nu_ndwd\theta
\end{equation*}
\begin{equation}
\tripNorm{f}^2=\sup_{G\in\mathcal{G}}\; \;\frac{1}{2n}\sum_{j=1}^n\frac{1}{2\pi}\int_0^{2\pi}\int_{\mathbb{R}^n}\left(\frac{\beta}{2\pi}\right)^{n/2}e^{-\frac{\beta}{2}w_j^{*2}}|f(V_j(\theta,w))G-f(G)|^2dwd\theta
\end{equation}
Using these, we can prove a concentration of measure result for the thermostat analogous to Theorem 3.2.2
\begin{thm} Consider the Gaussian thermostat and let $F:\mathcal{G}\rightarrow\mathbb{R}$ be such that $\tripNorm{F}\leq 1$.  Then $F$ is integrable with respect to $\nu_n$ and for every $r\geq 0$, 
$$
\nu_n(F\geq Fd\nu_n+r)\leq 3e^{-r\sqrt{\lambda_1}/2}
$$
where $\lambda_1=\frac{\mu}{2n}$ is the spectral gap of the thermostat process.  
\end{thm}
We omit the proof here as it is similar to the proof of Theorem 3.2.2.  \\ \\
Using this result and Theorem 3.4.1, we can prove the following concentration of measure inequality.  
\begin{thm}
Take any $1\leq k\leq n$ and an $n$-dimensional Hermitian matrix $G$.  Let $S$ be an $n\times k$ matrix whose $k$ columns are the first $k$ columns of a random matrix with distribution $\nu_n$.  Let $A$ be the $k\times k$ matrix obtained by conjugating $G$ by $S$.  Letting $F$ denote the expected spectral distribution of $A$, then for each $r>0$, 
$$
\mathbb{P}(\|F_A-F\|_{\infty}\geq \frac{1}{\sqrt{k}}+r)\leq 12\sqrt{k}\exp\left(-r\sqrt{\frac{k\mu}{108}}\right)
$$ where $\mu$ is the rate of the interaction with the thermostat.  
\end{thm}

\begin{proof}
The proof of this theorem closely follows the proof of Theorem 3.3.1, with appropriate changes made.  Let $A$ be stated as above, and let $A'$ be $A$ after one step of the Markov chain.  Fix $x\in\mathbb{R}$ and let $f(x)=F_A(x)$, where where $F_A$ is the empirical spectral distribution of $A$.  Notice that rank($A-A')\leq 3$, since after one step of the chain, at most 3 columns of $A$ will be changed (two from Kac's Walk, and one from the thermostat).  Again using the inequality from \cite{Bai}, we know that 
$$
\|f(A)-f(A')\|_{\infty}\leq\frac{3}{k}
$$
$$
\tripNorm{f}^2 =\frac{1}{2{n\choose 2}{n}}\sup_A\sum_{1\leq i <j\leq n} \sum_{k=1}^n\mathbb{E}|f(A)-f(A')|^2
$$
where the first sum is over possible interactions in Kac's process and the second is over possible particle interactions with the thermostat.  The above is 
$$
\leq \frac{1}{2}\left(\frac{3}{k}\right)^2\left(\frac{3k}{n}\right)=\frac{27}{2kn}
$$
Using theorems 3.4.1 and 3.4.2, we have that 
$$
\mathbb{P}(|F_A(x)-F(x)|\geq r)\leq 6\exp\left(-\frac{r}{2}\sqrt{\frac{\frac{\mu}{2n}}{\frac{27}{2kn}}}\right)
$$
$$
=6\exp\left(-\frac{r}{2}\sqrt{\frac{k\mu}{27}}\right)
$$
Following the rest of the proof in 3.2.1 (with the appropriate numbers changed), we get 
$$
\mathbb{P}(\|F_A-F\|_{\infty}\geq \frac{1}{\sqrt{k}}+r)\leq 12\sqrt{k}\exp\left(-r\sqrt{\frac{k\mu}{108}}\right)
$$
\end{proof}

\section{An Additional Application: The Length of the Longest Increasing Subsequence of a Random Walk Evolving under the Asymmetric Exclusion Process}
Consider a random walk X on $\{1,\dots,n\}$.  Represent $X$ by some element in $\{0,1\}^n$, where  $X_i=0$ corresponds to a step down in the walk at position $i$ and $X_i=1$ corresponds to a step up.  We will assume that 
$$
\sum_{i=1}^n X_i =\frac{n}{2}
$$
so that we have the same number of up steps as down steps. 
We can now look at this random walk as the initial configuration of a particle process with $X_i=1$ corresponding to a particle in position $i$ and $X_i=0$ corresponding to no particle at position $i$.  Consider the asymmetric exclusion process acting on this configuration with the following dynamics.  At each step of the process, a number $i$ is chosen uniformly in $\{1,\dots,n-1\}$.  If $X_i=X_{i+1}$, then the configuration stays the same.  If $X_i=1$ and $X_{i+1}=0$, then the values of $X_i$ and $X_{i+1}$ switch with probability $1-q/2$ and if $X_i=0$ and $X_{i+1}=1$, then the values switch with probability $q/2$.  Viewed in this way, the asymmetric exclusion process can be viewed as a Markov process on the set of random walks.  See \cite{Liggett} for an in depth discussion of the asymmetric exclusion process.

\begin{thm}[\cite{KomaNach},\cite{Alcaraz},\cite{CapMar}] The spectral gap of the ASEP is $\lambda_n=1-\Delta^{-1}\cos(\pi/n)$, where $\Delta=\frac{q+q^{-1}}{2}$ for a parameter $q$ satisfying $0<q<1$.
\end{thm}
In our case, take $q=1-c/n^{\alpha}$, for a constant $c$, and $0<\alpha<1$, such that 
$q\approx e^{-c/n^{\alpha}}$.  Then Taylor approximating and simplifying gives 
$$
\lambda_n=c^2/2n^{2\alpha}
$$
Now let $M_X$ denote the height of the midpoint of the random walk at a fixed time during the process.  In other words, $M_X=X_{n/2}$, assuming $n$ is even.  Note that the range of this function is $[-n/2,n/2]$. Let $M_x'$ be the evolution of $M_x$ after one step of the process.  Notice that $$\|M_x-M_x'\|_{\infty}\leq 1$$ since switching the position of two adjacent particles can change the height of the midpoint by at most $1$.  Then 
$$
\||M||_{\infty}^2=\frac{1}{2}\max_{X}\mathbb{E}(M_x-M_x')^2
$$
$$
\leq \frac{1}{2}(1)^2\left(\frac{1}{n-1}\right)=\frac{1}{2(n-1)}
$$
The $\frac{1}{n-1}$ appears because the only choice of $i$ that will effect the midpoint is $i=n/2$.  \\ \\
Now plugging into the Chatterjee Ledoux theorem, we have the following result.
\begin{thm}
Letting $M_X$ denote the height of the midpoint of the random walk after evolution under the asymmetric exclusion process, for all $r>0$ and $q=1-c/n^{\alpha}$,
$$
\mathbb{P}(|M_X-\mathbb{E}M_X|\geq r)\leq 6\exp\left(-r/2\sqrt{\frac{c^2/2n^{2\alpha}}{1/(2(n-1))}}\right) =6\exp\left(-r/2\sqrt{\frac{c^2(n-1)}{n^{2\alpha}}}\right)
$$
\end{thm}
\noindent
Notice that this implies that the height of the midpoint has fluctuations bounded above by a constant $n^{\alpha-1/2}$  for $0<\alpha<1$. 

Consider the length of the longest increasing (non-decreasing) subsequence of the random walk.  This is defined as $$ L_X=\max\{k \;   : \; i_1<i_2<\dots<i_k \; \mathrm{and} \; X_{i_1}\leq X_{i_2}\leq\dots\leq X_{i_k}\}$$  See \cite{PeresRW} for a more in depth description of this topic and results for the simple random walk.  
\begin{figure}
\centering
\includegraphics[]{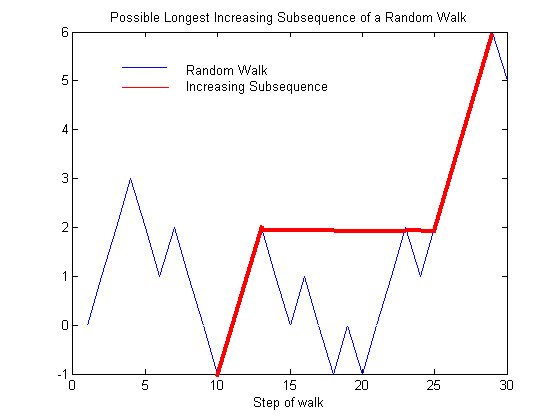}
\caption {A longest increasing subsequence of a random walk}
\end{figure}
Notice that the height of the midpoint gives a lower bound on the length of the longest increasing subsequence.  Using ASEP as our Markov process and the spectral gap above, we can prove concentration of measure for $L_X$.  Notice that switching the position of two adjacent particles via ASEP can only change $L_X$ by at most $1$.  As before, let $X'$ be the evolution of $X$ after one step of the process.  Then, bounding the probability above by $1$, we have 
$$
\||L\||_{\infty}^2=\frac{1}{2}\max_X\mathbb{E}(L_X-L_{X'})^2
$$
$$
\leq \frac{1}{2} (1)^2=\frac{1}{2}
$$
so plugging into the Chatterjee Ledoux formula, we get the following result.
\begin{thm}
Letting $L_X$ denote the length of the longest increasing subsequence of the random walk after evolution under the asymmetric exclusion process, for all $r>0$ and $q=1-c/n^{\alpha}$, 
$$
\mathbb{P}(|L_X-\mathbb{E}L_X|\geq r)\leq 6\exp\left(-r/2\sqrt{\frac{c^2}{n^{2\alpha}}}\right)
$$
\end{thm}
\noindent
This implies that the fluctuations are bounded above by a constant times $n^{\alpha}$.  In particular, for $q=1-c/\sqrt{n}$, the fluctuations are bounded above by a constant times $\sqrt{n}$. 
 
In order to give some context to the size of the fluctuations, we calculate height of the midpoint, which gives a lower bound on the length of the longest increasing subsequence of the walk under this distribution.
\begin{thm}
\label{height}
For $q<1-c/n$ and $c=-20\log(3/5)$, the height of the midpoint of the random walk is $kn$ for some constant $k>0$.  
\end{thm}

Before we give the proof, we will need the following lemma.
\begin{lem}
Consider a random walk with independent steps.  Assume that $\mathbb{P}(X_k=0)=\frac{1}{aq^k+1}$ and $\mathbb{P}(X_k=1)=\frac{aq^k}{aq^k+1}$ for some $a>0$, $q\in (0,1)$ and $k\in\mathbb{Z}_{+}$.  Consider $N_X=\sum_{i=1}^n X_i$.  This gives us the number of up steps in our random walk, or equivalently, the number of particles in our particle process.  The fluctuations of $N_X$ are at most order $\sqrt{n}$.
\end{lem}
\begin{proof}
We begin by calculating the variance of $N_X$.  We can then use Chebyshev's inequality to bound the fluctuations.  Since the $X_i$ are independent, 
$$
\mathrm{Var}(N_X)=\sum_{i=1}^n \mathrm{Var}(X_i)
$$
Using the probabilities given in the lemma, we know that 
$$
\mathrm{Var}(X_i)=\frac{aq^i}{aq^i+1}-\left (\frac{aq^i}{aq^i+1}\right)^2
$$
$$
=\frac{aq^i}{aq^i+1}\left (1-\frac{aq^i}{aq^i+1}\right)
$$
This gives 
$$
\mathrm{Var}(N_X)=\sum_{i=1}^n\frac{aq^i}{aq^i+1}\left (1-\frac{aq^i}{aq^i+1}\right)
$$
A derivative calculation show that $\frac{aq^i}{aq^i+1}\left (1-\frac{aq^i}{aq^i+1}\right)$ is decreasing in $i$, so 
$$
\mathrm{Var}(N_X)\leq n \left(\frac{aq}{aq+1}\right)\left(1-\frac{aq}{aq+1}\right)
$$
Since we only care about the order of the fluctuations, we can bound the positive value
$$
\left(\frac{aq}{aq+1}\right)\left(1-\frac{aq}{aq+1}\right)
$$
by $1$, giving us 
$$
\mathrm{Var}(N_X)\leq n
$$
Plugging into Chebyshev's inequality tells us that 
$$\mathbb{P}\left(|N_X-\mathbb{E}(N_X)|\geq k\right)\leq \frac{n}{k^2}$$
which proves our result.  

\end{proof}
We are now set to prove theorem \ref{height}
\begin{proof}
The basic idea of the proof of theorem \ref{height} is as follows.  We will begin by assuming that the steps of our random walk are independent, so that our measure is a product measure.  Recall, the steps are not independent, since we are conditioning on the fact that we have exactly $n/2$ steps up and $n/2$ steps down.  However, if $n$ is large, the steps are {\em close} to independent.  By bounding the fluctuations of the number of particles in our product system, we can then relate our non-independent state to the product state. 

Begin by assuming that 
$$
\frac{P(X_k=0)}{P(X_k=1)}=aq^k
$$
so that we have a product measure. 
Then we know that 
$$
P(X_k=0)=\frac{1}{aq^k+1}
$$ and $$P(X_k=1)=\frac{aq^k}{aq^k+1}$$
Then 
$$
\mathbb{E}\left(\sum_{i=1}^k X_i\right)=\sum_{i=1}^k\frac{aq^i}{aq^i+1}$$
Since the summand is decreasing in 
$i$, we get the bounds 
$$
k\left(\frac{aq^k}{aq^k+1}\right)\leq \mathbb{E}\left(\sum_{i=1}^k  X_i\right)\leq k\left(\frac{aq}{aq+1}\right)
$$
We will work in this generality for now, and add in appropriate values of $a$ and $k$ later.  Using this information, we can get bounds on the height of the random walk at point $k$.  Let $H_k$ be the height of the random walk at position $k$. For convenience later, we will assume that $X_i=1$ corresponds to a step down in the walk, and that $X_i=0$ corresponds to a step up.  Provided that we can prove that our height is $cn$ for $c<0$, our theorem will be proved.  We have  
$$
\mathbb{E}(H_k)=(-1)\sum_{i=1}^k X_i+\left(k-\sum_{i=1}^k X_i\right)=k-2\left(\sum_{i=1}^k X_i\right)
$$
Plugging in our bounds on $\mathbb{E}\left(\sum_{i=1}^k X_i\right)$, we get 
$$
-k\left(2\left(\frac{aq}{aq+1}\right)-1\right)\leq\mathbb{E}(H_k)\leq -k\left(2\left(\frac{aq^k}{aq^k+1}
\right)-1\right)
$$

At this point, we need a bound on the number of particles in the system.  
Since we are assuming the $X_i$ are independent, we can use the result from the previous lemma, which gives us
$$
\mathbb{P}\left(\left |\sum_{i=1}^n X_i-M\right |>u\right)\leq 4\exp(-u^2/4M)
$$ 
where $M$ is a median for the number of particles.  Estimating the median by the expectation of the number of particles, we see that $M$ should at least be close to $n/2\left(\frac{aq}{aq+1}\right)$.  If we choose $a$ appropriately corresponding to $q$, we should be able to make the constant order $1$, making our expectation order $n$.  Then, by the concentration of measure inequality, $\sum_{i=1}^n X_i$  has  fluctuations on the order of $\sqrt{n}$.  This is reasonably small compared with the expected number of particles in the system. 

Recall that we are actually concerned with finding the height of the midpoint, so plugging in $k=n/2$, we have that 
$$
-n/2\left(2\left(\frac{aq}{aq+1}\right)-1\right)\leq\mathbb{E}(H_{n/2})\leq -n/2\left(2\left(\frac{aq^{n/2}}{aq^{n/2}+1}
\right)-1\right)
$$
At this point, we can ignore the lower bound, using the fact that that a lower bound is $-n/2$ anyway, regardless of the configuration.  We will refer to our interface as the position in which $\mathbb{P}(X=0)=\mathbb{P}(X=1)$.  For now, we will put our interface at $9n/20$, which will be just to the left of the midpoint.   In other words, $a=q^{-9n/20}$ and at position $9n/20$, $\mathbb{P}(X=0)=\mathbb{P}(X=1)$. We will push it to the edge at $n/2$ at the end, since moving the interface to the right only increases the probability of more $X_i$ being equal to $1$, hence lowering the expectation of the midpoint.  Using this interface, we will first look at the height of the random walk at position $8n/20$.  Using the upper bound from above, we have that 
$$
\mathbb{E}(H_{8n/20})\leq \frac{-8n}{20}\left(2\left(\frac{q^{-n/20}}{q^{-n/20}+1}\right)-1\right)
$$
Beyond this point, if we assume that all of the remaining steps between $8n/20$ and $n/2$ are steps up, we have that 
$$
\mathbb{E}(H_{n/2})\leq \frac{-8n}{20}\left(2\left(\frac{q^{-n/20}}{q^{-n/20}+1}\right)-1\right) + \frac{2n}{20}
$$
The important thing to notice here, is this actually gives us an upper bound on the height of the midpoint in the fixed particle number (ASEP) random walk.  In the product state configuration, with our interface at $\frac{9n}{20}$, we know that the fluctuations in the number of down steps are less than $\frac{n}{20}$.  By assuming that all steps after site $\frac{8n}{20}$ are up, we have accounted for the worst case scenario where we actually have $\sqrt{n}$ less down steps then we expect.  If some of the steps after site $\frac{8n}{20}$ are actually down instead of up, this will only serve to lower the height of our midpoint.  Hence, we have, that in the ASEP (fixed number of down steps) random walk generated using the blocking measures, 
$$
\mathbb{E}(H_{n/2})\leq \mathbb{E}(H_{n/2})\leq \frac{-8n}{20}\left(2\left(\frac{q^{-n/20}}{q^{-n/20}+1}\right)-1\right) + \frac{2n}{20}
$$
We would like to show that for an appropriate choice of $q$, this is $cn$ for some constant $c<0$.  This is true provided that 
$$
\frac{8}{20}\left(2\left(\frac{q^{-n/20}}{q^{-n/20}+1}\right)-1\right)>\frac{2}{20}
$$
Solving this inequality gives a condition on q, which is 
$$
q>\left(\frac{3}{5}\right)^{\frac{20}{n}}
$$
or 
$$
q>e^{20/n\log(3/5)}
$$
Taylor expanding the exponential gives 
$$
q> 1+\frac{20}{n}\log(3/5)+\frac{400}{2n^2}(\log(3/5))^2+\dots
$$
As $n\rightarrow\infty$, taking $q>1-\alpha/n$ with $\alpha=-20\log(3/5)$ should be sufficient.  
As long as this condition is satisfied, our expectation is $cn$ for a constant $c<0$.  

At this point, we do want to move the interface to $a=q^{-n/2}$, such that $\mathbb{P}(X_{n/2}=0)=\mathbb{P}(X_{n/2}=1)$.  This simply increases our probability of down steps between $\frac{9n}{20}$ and $\frac{n}{2}$.  Since adding extra down steps only decreases the expectation of the height of the midpoint, the theorem is proved. 
\end{proof}
\section{Remarks}
By generalizing this method introduced by Chatterjee and Ledoux, we are able to show concentration of measure of the empirical spectral distribution not only for operator compressions via $SO(n)$ but also for operators that are "compressed" by conjugation with a Gaussian matrix.  It is likely that this method could be applied to a much wider range of Markov chains, given that the chain does not change too many entries at once, has an appropriate invariant distribution, and for which the spectral gap is known.  It is possible that better bounds for the Gaussian compression could be obtained by adapting the method to use the "second" spectral gap or the exponential decay rate in relative entropy found in \cite{BLV}. 

It is worth noting that Talagrand's isoperimetric inequality \cite{Talagrand} gives concentration of measure for the length of the longest increasing subsequence for random permutations, but it cannot be used in the context of this ASEP random walk, as it requires independence.  Using Chatterjee and Ledoux's method, independence is not needed.  We only need a spectral gap bound for the Markov chain.    
\\ \\

\chapter{Mixed Matrix Moments and Eigenvector Overlap Functions of the Ginibre Ensemble}
The purpose of this section is to make some observations about the mixed matrix moments for non-Hermitian random matrices.  The results in this chapter can be found in \cite{MWSS}.
Let $\Mat_n(\C)$ denote the set of $n\times n$ matrices with complex entries. We use this notation here because we will use $M_n$ for something else later.  

The model we will focus on most is the complex Ginibre ensemble, given by
\be
A_n \in \Mat_n(\C)\, ,\quad A_n = (a_n(j,k))_{j,k=1}^{n}\, ,\quad
a_n(j,k)\, =\, \frac{X(j,k)+iY(j,k)}{\sqrt{2n}}\, ,
\ee
where $(X(j,k))_{j,k=1}^{\infty}$, $(Y(j,k))_{j,k=1}^{\infty}$ are IID, $\mathcal{N}(0,1)$ real random variables.

Much of what we will say has already been explored by Chalker and Mehlig in a pair of papers \cite{CM,MC},
in particular, in their definition of expected overlap functions.
There are other models of interest which were explored by Fyodorov and coauthors \cite{FY1,FY2},
for which one can obtain more explicit formulas for the expected overlap functions.
Our main emphasis will be to relate Chalker and Mehlig's formulas for the overlap functions of the complex Ginibre
ensemble to the mixed matrix moments.

Our motivation in considering this problem is the following. There is a rough analogy between mean-field spin glasses and random matrices,
as far as the mathematical methods are concerned.
We indicate this in the table in Figure \ref{fig:1}.  We will give more details and references in a later discussion, but we would like to point out some of the analogies now.  
This analogy leads to a method to calculate moments, but there is still the question about how to relate
the moments to the spectral information for the matrix.
\
\begin{figure}[h]
\begin{center}
\begin{tikzpicture}
\begingroup \fontsize{9pt}{11pt}
\draw (-4,0.5) rectangle (12.125,-2.125);
\draw (0,0) node[] {\bf Random Matrices};
\draw (8.5,0) node[] {\bf Spin Glasses};
\draw (0,-0.75) node[] {
expectations of moments 
};
\draw (8.5,-0.75) node[] {
expectations of products of overlaps
};
\draw (0,-1.25) node[] {
recurrence relation for moments
};
\draw (7.5,-1.25) node[] {
stochastic stability equations: Ghirlanda-Guerra identities
};
\draw (0,-1.75) node[] {
formula for Stieltjes transform of limiting law
};
\draw (8.5,-1.75) node[] {
proof of Parisi's ultrametric ansatz
};
\endgroup
\end{tikzpicture}
\begingroup \fontsize{10pt}{11pt}
\caption{Some analogous elements in random matrix  and spin glass theory. (Proofs may differ considerably.)  \label{fig:1}} \endgroup
\end{center}
\end{figure}

Although the main subject of this subject is random matrices, we will give a very brief introduction to spin glasses, just to motivate our analogy.  Spin glasses are physical objects.  We will not say much about the physics behind them, as the subject of this paper is mathematics.  However, we will give a quote from Daniel Mattis's book \cite{Mattis} in his discussion of dilute magnetic alloys.  He says :
\\ \\
{\em "If the impurity atom does possess a magnetic moment this polarizes the
conduction electrons in its vicinity by means of the exchange interaction and
thereby influences the spin orientation of a second magnetic atom at some
distance. Owing to quantum oscillations in the conduction electrons’ spin polarization
the resulting effective interaction between two magnetic impurities
at some distance apart can be ferromagnetic (tending to align their spins)
or antiferromagnetic (tending to align them in opposite directions). Thus a
given magnetic impurity is subject to a variety of ferromagnetic and antiferromagnetic
interactions with the various neighboring impurities. What is
the state of lowest energy of such a system? This is the topic of an active
field of studies entitled “spin glasses,” the magnetic analog to an amorphous
solid."} \cite{Mattis} (p. 48)
\\ \\
Since this is a mathematics paper, we will consider a spin glass as a probabilistic model.  We can consider a system
$$
\Sigma_N=\{-1,1\}^N
$$
for a large integer $n$.  We call an element $\sigma\in \Sigma_n$ a configuration.  The components of $\sigma$ are called spins (and can each take the value either $\pm 1$).  The energy of the system in a configuration $\sigma$ is called the Hamiltonian, which is usually denoted $H_N(\sigma)$.  Given a parameter $\beta$ (the inverse temperature), we can define the Gibbs measure by 
$$
G_N(\{\sigma\})=\frac{\exp(-\beta H_N(\sigma))}{Z_N}
$$
where $Z_N$ is a normalizing factor, called the partition function.  The Gibbs measure is a probability measure which represents the probability of observing the configuration $\sigma$ after the system has reached equilibrium in a heat bath at temperature $1/\beta$.  $H_N(\sigma)$ relates to the interactions between the spins.  In the models that are often considered, the $H_N(\sigma)$ are random variables.  For a given $H_N(\sigma)$, the main problem is to understand the Gibbs measure.  See \cite{Talagrandbook} for a more in depth discussion of the probabilistic aspect of spin glasses.  

We will depart from our discussion of spin glasses now, to begin the discussion of random matrices.  The analogies between the two topics will be discussed more in depth later.  

We will start by briefly recalling the formula for the mixed matrix moments of the complex Ginibre ensemble, and we will emphasize the relation to spin glass techniques.
This formula is already known and we will give references.  

In later sections, we will describe the relationship between the mixed matrix moments and the expected overlap functions of Chalker and Mehlig.  This leads to some new problems.

\section{Mixed Matrix Moments}

Given any $n\times n$ matrix $A$, any positive integer $k$, and any nonnegative integers $p(1),q(1),\dots,p(k),q(k)$, we may define
\be
M_n(\mathbf{p};\mathbf{q})\, 
=\, \frac{1}{n}\, \operatorname{tr}[A_n^{p(1)} (A_n^*)^{q(1)} \cdots A_n^{p(k)} (A_n^*)^{q(k)}]\, ,
\ee
for $\mathbf{p}=(p(1),\dots,p(k))$, $\mathbf{q}=(q(1),\dots,q(k))$.  Notice that $M_0=1$.
As an example, consider 
\begin{multline}
M_n((2,2);(2,2))\, = \\ \, \frac{1}{n}\, \sum_{j_1,\dots,j_8=1}^n a_n(j_1,j_2) a_n(j_2,j_3) \overline{a}_n(j_4,j_3) \overline{a}_n(j_5,j_4) a_n(j_5,j_6) a_n(j_6,j_7)
\overline{a}_n(j_8,j_7) \overline{a}_n(j_1,j_8)\, .
\end{multline}
If we consider the Ginibre ensemble and let $a_n(j,k) = (X(j,k)+iY(j,k))/\sqrt{2n}$ as before, then we have
\be
\E[a_n(j,k) a_n(j',k')]\, =\, \E[\overline{a}_n(j,k) \overline{a}_n(j',k')], =\, 0\ \text{ and } \
\E[a_n(j,k) \overline{a}_n(j',k')]\, =\, n^{-1} \delta_{j,j'} \delta_{k,k'}\, .
\ee
Recall that Wick's rule says that for mean $0$ Gaussian random variables $X_1,\dots,X_n$,
$$
\mathbb{E}(X_1X_2\dots X_n)=\sum \prod_{i,j}(X_iX_j)
$$
where the sum is over all distinct ways of dividing $1,\dots,n$ into pairs.  
Using this, and defining $m_n(\mathbf{p},\mathbf{q}) = \E[M_n(\mathbf{p},\mathbf{q})]$, gives us:
\begin{equation}
\label{eq:recurrence}
m_n(\mathbf{p},\mathbf{q})\, =\, \sum_{(\mathbf{p}',\mathbf{q}',\mathbf{p}'',\mathbf{q}'') \in \mathcal{S}(\mathbf{p},\mathbf{q})} 
\E[M_n(\mathbf{p}',\mathbf{q}') M_n(\mathbf{p}'',\mathbf{q}'')]\, ,
\end{equation}
where $\mathcal{S}(\mathbf{p},\mathbf{q})$ is the set of all admissible pairs, which we describe now.
Let $R = p(1)+\dots+p(k)+q(1)+\dots+q(k)$, 
and define $\sigma = (\sigma(1),\dots,\sigma(R)) \in \{+1,-1\}^R$
as $\sigma\, =\, ((+1)^{p(1)},(-1)^{q(1)},\dots,(+1)^{p(k)},(-1)^{q(k)})$
viewed as spins on vertices arranged on a circle.
We will sometimes denote this as $\sigma_{\mathbf{p},\mathbf{q}}$.
Let $\Sigma(\mathbf{p},\mathbf{q})$ denote pairs $(\sigma',\sigma'')$ as follows.
We match up the first $+1$ and any $-1$.
Where these two are removed, we pinch the circle into two smaller circles. 
Then the remaining spins on the two smaller circles comprise $\sigma'$ and $\sigma''$.
E.g., for a particular example
\be
\sigma = (\underline{+1},+1,\underline{-1},-1,+1,+1,-1,-1) \mapsto (\sigma',\sigma'') = ((+1),(-1,+1,+1,-1,-1))\, .
\ee
The set $\Sigma(\mathbf{p},\mathbf{q})$ is the set of all possible pairs $(\sigma',\sigma'')$ obtainable in this way.
We then define $\mathcal{S}(\mathbf{p},\mathbf{q})$ to be the set of all pairs $(\mathbf{p}',\mathbf{q}')$ and $(\mathbf{p}'',\mathbf{q}'')$ by mapping
backwards $\Sigma(\mathbf{p},\mathbf{q})$ from $\sigma'$ and $\sigma''$, this way.

Using this, we wish to give the main ideas of the proof of the following theorem.  
\begin{theorem}
\label{thm:MMM}
For any $k$ and any $\mathbf{p},\mathbf{q}$, we have
$$
\lim_{n \to \infty} m_n(\mathbf{p},\mathbf{q})\, =\, m(\mathbf{p},\mathbf{q})\, ,
$$
where $m(\mathbf{p},\mathbf{q})$ is as follows.
Let $C_R$ denote the number of all non-crossing matchings of
$R$ vertices on a circle (Catalan's number).
Let $m(\mathbf{p},\mathbf{q})$ denote the cardinality of all such matchings
satisfying the following constraint: assigning spins to the $R$ vertices by 
 $\sigma_{\mathbf{p},\mathbf{q}}$,  
each edge has two endpoints with one $+1$ spin and one $-1$ spin.
\end{theorem}
As an example,  
$m((2,2);(2,2)) = 3$ where the matchings are indicated diagrammatically as
$$
\begin{tikzpicture}
\draw (0,0) circle (1.5cm);
\draw[very thick] (67.5:1.5cm) .. controls (75:0.75cm) and (105:0.75cm)  .. (112.5:1.5cm);
\draw[very thick] (22.5:1.5cm) .. controls (15:0.75cm) and (-15:0.75cm)  .. (-22.5:1.5cm);
\draw[very thick] (-67.5:1.5cm) .. controls (-75:0.75cm) and (-105:0.75cm)  .. (-112.5:1.5cm);
\draw[very thick] (157.5:1.5cm) .. controls (165:0.75cm) and (195:0.75cm)  .. (202.5:1.5cm);
\fill (157.5:1.5cm) circle (1mm);
\fill (112.5:1.5cm) circle (1mm);
\filldraw[fill=white, very thick] (67.5:1.5cm) circle (1mm);
\filldraw[fill=white, very thick] (22.5:1.5cm) circle (1mm);
\fill (-22.5:1.5cm) circle (1mm);
\fill (-67.5:1.5cm) circle (1mm);
\filldraw[fill=white, very thick] (-112.5:1.5cm) circle (1mm);
\filldraw[fill=white, very thick] (-157.5:1.5cm) circle (1mm);
\end{tikzpicture}
\hspace{1cm}
\begin{tikzpicture}
\draw (0,0) circle (1.5cm);
\draw[very thick] (67.5:1.5cm) .. controls (75:0.75cm) and (105:0.75cm)  .. (112.5:1.5cm);
\draw[very thick] (22.5:1.5cm) .. controls (90:0.25cm) ..  (157.5:1.5cm);
\draw[very thick] (-22.5:1.5cm) .. controls (-90:0.25cm) ..  (-157.5:1.5cm);
\draw[very thick] (-67.5:1.5cm) .. controls (-75:0.75cm) and (-105:0.75cm)  .. (-112.5:1.5cm);
\fill (157.5:1.5cm) circle (1mm);
\fill (112.5:1.5cm) circle (1mm);
\filldraw[fill=white, very thick] (67.5:1.5cm) circle (1mm);
\filldraw[fill=white, very thick] (22.5:1.5cm) circle (1mm);
\fill (-22.5:1.5cm) circle (1mm);
\fill (-67.5:1.5cm) circle (1mm);
\filldraw[fill=white, very thick] (-112.5:1.5cm) circle (1mm);
\filldraw[fill=white, very thick] (-157.5:1.5cm) circle (1mm);
\end{tikzpicture}
\hspace{1cm}
\begin{tikzpicture}
\draw (0,0) circle (1.5cm);
\draw[very thick] (67.5:1.5cm) .. controls (0:0.25cm) .. (-67.5:1.5cm);
\draw[very thick] (22.5:1.5cm) .. controls (15:0.75cm) and (-15:0.75cm)  .. (-22.5:1.5cm);
\draw[very thick] (112.5:1.5cm) .. controls (180:0.25cm)  .. (-112.5:1.5cm);
\draw[very thick] (157.5:1.5cm) .. controls (165:0.75cm) and (195:0.75cm)  .. (202.5:1.5cm);
\fill (157.5:1.5cm) circle (1mm);
\fill (112.5:1.5cm) circle (1mm);
\filldraw[fill=white, very thick] (67.5:1.5cm) circle (1mm);
\filldraw[fill=white, very thick] (22.5:1.5cm) circle (1mm);
\fill (-22.5:1.5cm) circle (1mm);
\fill (-67.5:1.5cm) circle (1mm);
\filldraw[fill=white, very thick] (-112.5:1.5cm) circle (1mm);
\filldraw[fill=white, very thick] (-157.5:1.5cm) circle (1mm);
\end{tikzpicture}
$$
Theorem \ref{thm:MMM} is a well-known result. We refer to \cite{Kemp} for a discussion.
We will motivate a proof of this result, without including all details, here.
Our reason is that we actually want to use this result to motivate the discussion of random matrices and spin glasses further, which we indicated earlier.

\subsection{Argument for the Proof of the Mixed Matrix Moments}
\label{subsec:Argument}
The first step in the argument for the proof of Theorem \ref{thm:MMM} is to use concentration of measure (COM) to replace (\ref{eq:recurrence})
with a nonlinear recurrence relation.
Here what we mean is non-linearity in the probability measure for the random entries of the matrix.
Since the expectation is linear, what we really mean is to obtain a product of two expectations.
If $M_n(\mathbf{p}',\mathbf{q}')$ and $M_n(\mathbf{p}'',\mathbf{q}'')$ were independent, then we could
replace the expectation by a product,
but they are not exactly independent.
Instead, they satisfy COM, which means that they are approximately non-random.
And, of course, non-random variables are exactly independent of every other random variable
(as well as themselves).

The easiest version of COM is just $L^2$-concentration. For example, the following lemma is very easy to prove:
\begin{lemma}
\label{lem:unifGrad}
Suppose $f : \R^n \to \R$ is a function such that \\
$\|\nabla f\|^2_{\infty} = \sup_{\mathbf{x} \in \R^n} \sum_{k=1}^n \left(\frac{\partial f}{\partial x_k}(\mathbf{x})\right)^2$ is finite.
Then if $U_1,\dots,U_n,V_1,\dots,V_n$ are IID $\mathcal{N}(0,1)$ random variables then
\be
\label{eq:lemG}
\E\left[\left(f(\mathbf{U})-f(\mathbf{V})\right)^2\right]\, \leq\, 2 \|\nabla f\|_{\infty}^2\, .
\ee
\end{lemma}
This can be proved using the basic, but important method of ``quadratic interpolation,'' which is sometimes
called the ``smart path method'' by some mathematicians working on spin glasses.
\begin{proof}
Let $\boldsymbol{Z} = (Z_1,\dots,Z_n)$ be an IID $\mathcal{N}(0,1)$ vector, independent of 
$\mathbf{U}$ and $\mathbf{V}$. Then define $\widetilde{\mathbf{U}}(\theta) = \sin(\theta)\, \mathbf{U} 
+ \cos(\theta)\, \mathbf{Z}$ and $\widetilde{\mathbf{V}}(\theta) = \sin(\theta)\, \mathbf{V} 
+ \cos(\theta)\, \mathbf{Z}$.
Then $\frac{d}{d\theta} \widetilde{\mathbf{U}}(\theta) = \widetilde{\mathbf{U}}(\theta+\frac{\pi}{2})$,
and $\E[\widetilde{\mathbf{U}}(\theta) \widetilde{\mathbf{U}}(\theta+\frac{\pi}{2})] = 0$.
This means that $\widetilde{\mathbf{U}}(\theta)$ is statistically independent of its $\theta$-derivative.
Similar results hold for $\widetilde{\mathbf{V}}(\theta)$.
On the other hand $\E[\widetilde{\mathbf{V}}(\theta) \widetilde{\mathbf{U}}(\theta+\frac{\pi}{2})] = 
-\sin(\theta)\cos(\theta)$.

Next, using the fundamental theorem of calculus,
\be
\E\left[\left(f(\mathbf{U})-f(\mathbf{V})\right)^2\right]\,
=\, \int_0^{\pi/2} \frac{d}{d\theta}
\E\left[\left(f(\widetilde{\mathbf{U}}(\theta))-f(\widetilde{\mathbf{V}}(\theta))\right)^2\right]\,
d\theta\, ,
\ee
and an easy calculation using Gaussian integration by parts (and the covariance formulas mentioned above) shows that 
\be
\frac{d}{d\theta}
\E\left[\left(f(\widetilde{\mathbf{U}}(\theta))-f(\widetilde{\mathbf{V}}(\theta))\right)^2\right]\,
=\, 2\sin(2\theta) \E\left[\nabla f(\widetilde{\mathbf{U}}(\theta)) \cdot 
\nabla f(\widetilde{\mathbf{V}}(\theta))\right]\, .
\ee
Then (4.7) follows by using the Cauchy-Schwarz inequality.
\end{proof}
This is only the simplest Gaussian COM result. Notice that the method of proof is similar to the method used to proved Talagrand's Gaussian concentration of measure inequality for Lipschitz functions as stated in chapter 2 Theorem \ref{GCOM} 

This lemma is a tool which can be applied to show that the various mixed matrix moments $M_n(\mathbf{p},\mathbf{q})$
do satisfy COM.  We present this lemma here, because it is easier to obtain concentration of measure for the matrix moments using this lemma than with Theorem \ref{GCOM}.  It should be noted, that \ref{GCOM} will also work in this case and will give a sharper concentration bound.  
Either way, it is an interesting calculation, and much of the combinatorics, especially involving matchings related to Catalan's number,
are first visible in the grad-squared calculation.

Since the goal of this section is to give a general outline of the proof of the formula for the mixed matrix moments and relate it to spin glass techniques, we will just state that the desired concentration of measure result is true.

Then we are able to boost (\ref{eq:recurrence}) to 
\begin{equation}
\label{eq:recur2}
\lim_{n \to \infty} m_n(\mathbf{p},\mathbf{q}) - 
\sum_{(\mathbf{p}',\mathbf{q}',\mathbf{p}'',\mathbf{q}'') \in \mathcal{S}(\mathbf{p},\mathbf{q})} 
m_n(\mathbf{p}',\mathbf{q}') m_n(\mathbf{p}'',\mathbf{q}'')\, =\, 0\, .
\end{equation}
Another easy fact is that, due to symmetry, $m_n(\mathbf{p},\mathbf{q})=0$ unless $p(1)+\dots+p(k)=q(1)+\dots+q(k)$.
And, of course, $m_0=1$.

Using this, and the method of induction, one can then prove Theorem \ref{thm:MMM}.

\subsection{Commentary on Proof Technique}

The quadratic interpolation technique is important in spin glasses. 
The first major use was by Guerra and Toninelli \cite{GT} and Guerra \cite{G}.
It is called the ``smart path method'' by Talagrand \cite{Talagrand2}.  This is the method which we used to prove Talagrand's Gaussian concentration of measure inequality in chapter 2.  

Using Wick's rule to obtain a recurrence relation is important in many subjects.
It is a standard approach to determining moments of random matrices.
See, for instance, \cite{AndersonGuionnetZeitouni}, chapter 1.
In the context of Gaussian spin glasses, this technique combined with stochastic stability
leads to the Aizenman-Contucci identities \cite{AC}.
When combined with concentration of measure it leads to the Ghirlanda-Guerra identities \cite{GG}.
See, for instance, the review \cite{ContucciGiardina}.

For random matrices, the problem of recombining the moments into useful information
about the limiting empirical spectral measure is also important.
For Hermitian random matrices, this is related to the classical moment method.
The standard approach is to put the moments together into the Stieltjes transform,
and then to proceed from there \cite{Pastur}.
Again, a good general reference is \cite{AndersonGuionnetZeitouni}, chapter 1.

For spin glasses, the problem of integrating the Ghirlanda-Guerra identities
into a useful result for mean-field models was solved only relatively recently.
Panchenko showed that the ``extended Ghirlanda-Guerra identities'' imply
Parisi's ultrametric ansatz \cite{Panchenko}.
This is an important work. One element of his proof is putting various terms together into a
a new exponential type generating function. 
This might be somewhat analogous to the Stieltjes transform step.
But after that, the proofs are very different.

For non-Hermitian random matrices, getting useful information from the moments
is the topic we focus on next.

\section{The Expected Overlap Functions}
\label{sec:Expected}

Since the moments $M_n(\mathbf{p},\mathbf{q})$ satisfy concentration of measure, one is primarily only interested in their expectations.
The next quantity we introduce is also defined just for the expectation.
(Studying its distribution may be interesting, but we will not comment on this, here.)
It is the expected overlap function of Chalker and Mehlig, introduced in \cite{CM} and further studied by them in \cite{MC}.

Given $A_n \in \Mat_n(\C)$, randomly distributed according to Ginibre's ensemble, almost surely it may be diagonalized.
This means that we can find eigenvalues $\lambda_1,\dots,\lambda_n \in \C$ as well as pairs of vectors $\psi_1,\phi_1,\dots,\psi_n,\phi_n \in \C^n$ such that
\be
A_n \psi_k\,=\, \lambda_k \psi_k\, ,\quad
\phi_k^* A_n\, =\, \lambda_k \phi_k^*\, ,\quad
\phi_k^* \psi_j\, =\, \delta_{jk}\, .
\ee
Using this, for any other vector $\Psi \in \C^n$, there is the formula
\be
A_n \Psi\, =\, \sum_{k=1}^{n} \lambda_k \langle \phi_k, \Psi \rangle \psi_k\, .
\ee
These are random because they depend on $A_n$, but we may take the expectation over the randomness.

Given any continuous function, $f$, with compact support on $\C$, one may define
\be
\omega^{(1)}_n[f]\, =\, \frac{1}{n}\, \E\left[\sum_{k=1}^{n} f(\lambda_k) \|\phi_k\|^2 \|\psi_k\|^2\right]\, .
\ee
Similarly, given any continuous function, $F$, with compact support on $\C \times \C$, we may define
\be
\omega^{(2)}_n[F]\, =\, \frac{1}{n}\, \E\left[\sum_{j=1}^{n} \sum_{k\neq j} F(\lambda_j,\lambda_k) \langle \psi_k,\psi_j \rangle \langle \phi_j, \phi_k\rangle\right]\, .
\ee
Regularity of the eigenvalues and eigenvectors with respect to the matrix entries guarantees existence of functions $\mathcal{O}_n^{(1)} : \C \to \C$ and
$\mathcal{O}_n^{(2)} : \C \times \C \to \C$ such that
\be
\omega^{(1)}_n[f]\, =\, \int_{\C} f(z) \mathcal{O}_n^{(1)}(z)\, d^2z\quad \text{and}\quad
\omega^{(2)}_n[F]\, =\, \int_{\C}\int_{\C} F(z,w) \mathcal{O}_n^{(2)}(z,w)\, d^2z\, d^2w\, .
\ee
Using these definitions, one may determine a relation between these expected overlap functions
and the correlation functions for the eigenvalues.
Define $\rho^{(1)}_n$ and $\rho^{(2)}_n$, analogously to $\omega^{(1)}_n$ and $\omega^{(2)}_n$
as
\be
\rho^{(1)}_n[f]\, =\, \frac{1}{n}\, \E\left[\sum_{k=1}^{n} f(\lambda_k)\right]\, ,\
\text{ and }\
\rho^{(2)}_n[F]\, =\, \frac{1}{n}\, \E\left[\sum_{j=1}^{n} \sum_{k\neq j} F(\lambda_j,\lambda_k)\right]\, .
\ee
Then there are functions $\mathcal{R}_n^{(1)} : \C \to \C$ and $\mathcal{R}_n^{(2)} : \C \times \C \to \C$ such that
\be
\rho^{(1)}_n[f]\, =\, \int_{\C} f(z) \mathcal{R}_n^{(1)}(z)\, d^2z\quad \text{and}\quad
\rho^{(2)}_n[F]\, =\, \int_{\C}\int_{\C} F(z,w) \mathcal{R}_n^{(2)}(z,w)\, d^2z\, d^2w\, .
\ee
Then
\be
\mathcal{O}_n^{(1)}(z) + \int_{\C} \mathcal{O}_n^{(2)}(z,w)\, d^2w\, =\, \mathcal{R}_n^{(1)}(z)\, .
\ee
In terms of these functions, for any nonnegative integers $p$ and $q$,
\be
\label{eq:mnpq}
m_n((p);(q))\, =\, \int_{\C} z^p \overline{z}^q \mathcal{O}_n^{(1)}(z)\, d^2z 
+ \int_{\C} \int_{\C} z^p \overline{w}^q \mathcal{O}_n^{(2)}(z,w)\, d^2z\, d^2w\, .
\ee
Therefore, the mixed matrix moments are calculable from the overlap functions.
Moreover, the limiting values of the moments give some constraints for the limiting behavior of the overlap functions.
It is easy to see that $\mathcal{O}_n^{(1)}(e^{i\theta} z) = \mathcal{O}_n^{(1)}(z)$
and $\mathcal{O}_n^{(2)}(e^{i\theta}z,e^{i\theta}w) = \mathcal{O}_n^{(2)}(z,w)$,
consistent with the fact that $m_n((p);(q))$ equals $0$ unless $p=q$.

\section{Formulas for the Overlap Functions}

Chalker and Mehlig were able to relate the overlap functions to expectations of functions involving all
the eigenvalues. The eigenvalue distribution for the complex Ginibre ensemble is well-known.
In fact it is one of the simplest of the various Gaussian ensembles.
For example, as Chalker and Mehlig also point out in their paper,
\be
\label{eq:RD}
\mathcal{R}_N^{(1)}(z)\, =\, \frac{N}{\pi N!}\, e^{-N |z|^2} D_{N-1}(z)\, ,
\ee
where $D_{N-1}(z)$ equals the determinant of the $(N-1)$-dimensional square 
matrix $\mathcal{D}_N(z)$ where the matrix entries are best indexed for $j,k \in \{0,\dots,N-2\}$ as 
\be
[\mathcal{D}_N(z)]_{jk}\, =\, \left(\frac{N^{j+k+4}}{\pi^2 (j!) (k!)}\right)^{1/2} \int_{\C} \overline{\lambda}^j \lambda^k
|z-\lambda|^2 \exp(-N|\lambda|^2)\, 
d^2\lambda\, .
\ee
By rotational invariance of all the terms in the integrand other than $|z-\lambda|^2$, which is only
quadratic, it happens that $\mathcal{D}_N(z)$ is a tridiagonal matrix.
Hence, Chalker and Mehlig point out that it is easy to derive a recursion relation for $D_{N-1}(z)$.
It is easier to define a new quantity $D_{N-1}(\sigma^{-2},z) = D_{N-1}(\sigma^{-1} N^{-1/2} z)$.
Then they show
\be
\label{eq:Dr}
D_{n+1}(\sigma^{-2},z)\, =\, (\sigma^{-2}|z|^2+n+1) D_n(\sigma^{-2},z) - \sigma^{-2} n |z|^2 D_{n-1}(\sigma^{-2},z)\, ,
\ee
and $D_0(\sigma^{-2},z)=1$, $D_1(\sigma^{-2},z) = 1+\sigma^{-2}|z|^2$. It turns out to be easy to solve this recurrence relation,
and Chalker and Mehlig give the formula
\begin{equation}
\label{eq:D}
D_{N-1}(\sigma^{-2},z)\, =\, (N-1)! \sum_{n=0}^{N-1} \frac{(\sigma^{-2}|z|^2)^n}{n!}\, ,
\end{equation}
which is the partial sum for the series for $(N-1)! \exp(\sigma^{-2}|z|^2)$. In order to obtain $D_{N-1}(z)$ one must take $\sigma^{-2}=N$.
One sees that the dividing line is $|z|<1$ versus $|z|>1$, as to whether enough terms have been included in the partial sum
to get essentially $\exp(N|z|^2)$ or not. From this it follows that the measure $\mathcal{R}_N^{(1)}(z)\, d^2z$ converges 
weakly to $\pi^{-1} \mathbf{1}_{[0,1]}(|z|^2)\, d^2z$, as $N \to \infty$.
The reason for going into so much detail in this example is that the other examples are similar, but harder.
In fact, some of the formulas are so complicated that so far they have eluded any explicit, exact formula (at least as far as we have been able to find in the literature).

Another easy result which follows from these explicit formulas, but which does not appear in the paper of Chalker and Mehlig, is the scaling formula
near the unit circle. Let us record this for later reference.
\begin{lemma} \label{lem:circle}
 For any $u \in \R$,
\be
\mathcal{R}_N^{(1)}(1-N^{-1/2}u)\, \stackrel{N\to\infty}{\longrightarrow}\, \pi^{-1} \Phi(2u)\quad 
\text{ where }\quad
\Phi(x)\, =\, \frac{1}{\sqrt{2\pi}}\, \int_{-\infty}^{x} e^{-z^2/2}\, dz\, .
\ee
\end{lemma}
\begin{proof}
Given the exact formula,
\be
\mathcal{R}_N^{(1)}(1-N^{-1/2}u)\, =\, \pi^{-1} \exp(-N+N^{1/2}u) \sum_{n=0}^{N-1} \frac{(N-N^{1/2}u)^n}{n!}\, ,
\ee
make the substitution $n=N-N^{1/2}x$ for $x \in \{N^{-1/2},2N^{-1/2},\dots,N^{1/2}\}$ and use Stirling's formula.
Then replace the sum by an appropriate integral in $x$ (of which it is a Riemann sum approximation with $\Delta x=N^{-1/2}$) by using the rigorous Euler-Maclaurin
summation formula.
\end{proof}
We may note that using the Euler-Maclaurin summation formula, one may obtain more terms as corrections of the leading-order term, just as one does
for the asymptotic series in Stirling's formula. Additionally, one may obtain formulas that are valid for more values of $u$:
one may obtain an asymptotic formula for $\mathcal{R}_N^{(1)}(z)-\pi^{-1}$ assuming that $|z|-1<CN^{-1/2}$ for some $C$, and another formula for $\mathcal{R}_N^{(1)}(z)$
assuming that $|z|-1>-CN^{-1/2}$ for some $C$: the difference in being whether one chooses to asymptotically evaluate the terms which are present in the partial sum for $\exp(N|z|^2)$
or whether one chooses to asymptotically evaluate the terms which are absent in that partial sum.

\subsection{More Involved Formulas:}
The formula for $\mathcal{O}_N^{(1)}$ is not much more complicated than the formula for $\mathcal{R}_N^{(1)}$, and Chalker and Mehlig gave the explicit answer.
It turns out that one may write $\mathcal{O}_N^{(1)}$ similarly to $\mathcal{R}_N^{(1)}$ as 
\be
\begin{gathered}
\mathcal{O}_N^{(1)}(z)\, 
=\, \frac{N}{\pi N!}\, \exp(-N|z|^2) G_{N-1}(z)\quad 
\text{ where }\quad
G_{N-1}(z)\, =\, \det[\mathcal{G}_{N-1}(z)]\, ,\\
\forall j,k \in \{0,\dots,N-2\}\, ,\quad \\
[\mathcal{G}_{N-1}(z)]_{jk}\, =\,  \left(\frac{N^{j+k+4}}{\pi^2 (j!) (k!) }\right)^{1/2} \int_{\C} \overline{\lambda}^j \lambda^k
(N^{-1}+|z-\lambda|^2) \exp(-N|\lambda|^2)\, 
d^2\lambda\, .
\end{gathered}
\ee
The matrix $\mathcal{G}_{N-1}(z)$ is also tridiagonal for the same reason as $\mathcal{D}_{N-1}(z)$.
In particular, there is again a recursion relation for $G_{N-1}(z)$.
Defining $G_{N-1}(\sigma^{-2},z) = G_{N-1}(\sigma^{-1}N^{-1/2}z)$, one may see the recursion formula
\begin{equation}
\begin{split}
G_{n+1}(\sigma^{-2},z)\, \\
&=\, [\mathcal{G}_n(\sigma^{-2},z)]_{nn} G_{n}(\sigma^{-2},z) - [\mathcal{G}_n(\sigma^{-2},z)]_{n,n-1} [\mathcal{G}_n(\sigma^{-2},z)]_{n-1,n} G_{n-1}(\sigma^{-2},z)\\
&=\, (\sigma^{-2}+n+2) G_{n}(\sigma^{-2},z)  - \sigma^{-2} n |z|^2 G_{n-1}(\sigma^{-2},z)\, ,
\end{split}
\end{equation}
with $G_0(\sigma^{-2},z) = 1$ and $G_1(\sigma^{-2},z) = 2 + \sigma^{-2}|z|^2$.
\begin{lemma}
\label{lem:G}
The exact solution to the recursion relation when $\sigma^{-2}=N$ is 
\be
\label{eq:lemG}
G_N(z)\, =\, (N-1)!\, \sum_{n=0}^{N-1} (N-n) \, \frac{(N|z|^2)^n}{n!}\, .
\ee
\end{lemma}
Using this formula, it is easy to see that $N^{-1} \mathcal{O}_N(z)\, d^2z$ converges weakly to $\pi^{-1} (1-|z|^2) \mathbf{1}_{[0,1]}(|z|^2)\, d^2z$, which is precisely the behavior
that Chalker and Mehlig found by other techniques. We will return to their approach, shortly.
For now, let us state the analogue of Lemma \ref{lem:circle}.
\begin{corollary}
For any $u \in \R$,
\be
\label{eq:Ocirc}
\mathcal{O}_N^{(1)}(1-N^{-1/2}u)\, \sim\, \frac{N^{1/2}}{\pi}\, \left[\frac{e^{-2u^2}}{\sqrt{2\pi}} - 2u\Phi(-2u)\right]\, ,\quad \text{ as $N \to \infty$.}
\ee
\end{corollary}
\begin{proof}
The proof is perfectly analogous to the proof of Lemma \ref{lem:circle}, except we start with Lemma \ref{lem:G} instead of equation (\ref{eq:D}).
\end{proof}
One also needs the two point function $\mathcal{O}_N^{(2)}$ in order to obtain any interesting moments.
The two-point function for the eigenvalues is easier to start with since its distribution is known exactly.
Using ideas related to the theory of orthogonal polynomials, one may see that $\mathcal{R}_N^{(2)}(z_1,z_2)$ is determinantal.
The canonical general reference for this is \cite{Mehta}.
One may write the formula as 
\be
\mathcal{R}_N^{(2)}(z_1,z_2)\, =\, \pi^{-2} e^{-N|z_1|^2} e^{-N|z_2|^2}
\det\left(K_N(z_j \overline{z}_k)\right)_{j,k=1}^{2}\quad \text{ for }\quad
K_N(z)\, =\, \sum_{n=0}^{N-1} \frac{(Nz)^n}{n!}\, .
\ee
From this one may determine the following asymptotics, proved in the same way as before.
\begin{lemma}
Define $\mathcal{C}_N^{(2)}(z_1,z_2) = \mathcal{R}_N^{(2)}(z_1,z_2) - \mathcal{R}_N^{(1)}(z_1) \mathcal{R}_N^{(1)}(z_2)$,
the corrected correlation function for the eigenvalues.
Then for any fixed $u_1,u_2 \in \C$
\be
\mathcal{C}_N^{(2)}(1-N^{-1/2}u_1,1-N^{-1/2}u_2)\, 
\sim\, \pi^{-2} e^{-|u_1-u_2|^2} |\Phi(-u_1-\overline{u}_2)|\, ,
\ee
where the definition of $\Phi$ is extended to the complex plane as $$\Phi(-u)\, =\, (2\pi)^{-1/2} e^{-u^2/2} \int_{0}^{\infty} e^{-x^2/2} e^{-ux}\, dx$$
\end{lemma}
We have stated a somewhat precise limit for $\mathcal{R}_N^{(2)}$, but we do not know how to get a precise limit for $\mathcal{O}_N^{(2)}$.
Let us state one of Chalker and Mehlig's main results as a conjecture. In other words, they give a good argument for the calculation of $\mathcal{O}_N^{(2)}$
which is highly plausible on the basis of mathematical reasoning, but to the best of our knowledge their result has not yet been fully rigorously proved.
\begin{conj}[Chalker and Mehlig]
(i) For any two points $z_1,z_2$ such that $|z_1|<1$, $|z_2|<1$ and $|z_1-z_2|>0$,
\be
\mathcal{O}_N^{(2)}(z_1,z_2)\, \stackrel{N\to\infty}{\longrightarrow}\,
-\frac{1}{\pi^2}\cdot \frac{1-z_1\overline{z}_2}{|z_1-z_2|^4}\, .
\ee
(ii)
For any $\omega \in \C$ and $z$ such that $|z|<1$,
\begin{multline}
\label{eq:CM2}
N^{-2} \mathcal{O}_N^{(2)}\Big(z+\frac{1}{2}N^{-1/2}\omega,z-\frac{1}{2}N^{-1/2}\omega\Big)\,
\sim\, -\pi^{-2} (1-|z|^2)\, \frac{1-(1+|\omega|^2)e^{-|\omega|^2}}{|\omega|^4}\, ,\qquad \\ \text{ as $N \to \infty$.}
\end{multline}
\end{conj}
Importantly, there is no asymptotic formula for $z_1$ and $z_2$ near the boundary of the circle.
For all the other cases, this regime gives lower-order corrections, beyond the leading order.

Chalker and Mehlig's approach is beautiful and compelling.
They calculated an explicit formula for $\mathcal{O}_N^{(2)}(0,z)$.
Note, for instance, that $\mathcal{R}_N^{(2)}(0,z) = \pi^{-1} (\mathcal{R}_N^{(1)}(z)-\pi^{-1}e^{-|z|^2})$,
so the formula simplifies when one of the arguments is $0$.
A similar fact holds for $\mathcal{O}_N^{(2)}(z_1,z_2)$, even though it seems that it is not determinantal
like $\mathcal{R}_N^{(2)}(z_1,z_2)$.
Then, Chalker and Mehlig considered a universality-type argument to see how the functional form should behave under transformations
of the point $0$ to other places on the circle.
Their argument is also a universal argument, applying to more ensembles than just the complex Ginibre ensemble,
but we will continue to consider just the complex Ginibre ensemble, here.

The second part of their argument is the key to their formula. 
The function $\mathcal{O}_N^{(2)}(z_1,z_2)$ may be expressed as the expectation
of a non-local function of all the eigenvalues of $A_n$.
Chalker and Mehlig observe that the function depends mainly on the eigenvalues in a core small
area around $z_1$ and $z_2$.
For this core, the distribution of the eigenvalues should be universal, not depending on the proximity
of $z_1$ and $z_2$ to the boundary of the disk, as long as they are not near the boundary.
Then outside the core there is a self-averaging contribution of all the other eigenvalues, which
may be reduced to a Riemann integral approximation, and calculated.
That part does depend on the geometry of the point configuration in the disk, but it is easily calculated.
Putting these two parts together with their formula for $\mathcal{O}_N^{(2)}(0,z)$,
they were able to arrive at (\ref{eq:CM2}).

The reader is advised most strongly to consult their beautiful paper.

Now we want to explain briefly the first part of their argument since it is a basis for a different proposal
we have for how to prove their conjecture.
Chalker and Mehlig point out that $\mathcal{O}_N^{(2)}(z_1,z_2)$ may be calculated as the determinant of a 5-diagonal matrix.
In fact, it is easier to start with $\mathcal{R}_N^{(2)}(z_1,z_2)$:
\be
\mathcal{R}_N^{(2)}(z_1,z_2)\, 
=\, \frac{N^3}{\pi^2 N!}\, |z_1-z_2|^2 e^{-N|z_1|^2} e^{-N|z_2|^2} F_{N-2}(z_1,z_2)\, ,
\ee
where $F_{N-2}(z_1,z_2)$ equals the determinant of the $(N-2)$-dimensional square matrix $\mathcal{F}_{N-2}(z_1,z_2)$, where
\be
[\mathcal{F}_{N-2}(z_1,z_2)]_{jk}\, =\, \left(\frac{N^{j+k+6}}{\pi^2 (j+1)! (k+1)! }\right)^{1/2} \int_{\C} \overline{\lambda}^j \lambda^k |z_1-\lambda|^2 |z_2-\lambda|^2 
\exp(-N|\lambda|^2)\, d^2\lambda\, ,
\ee
for $j,k=0,\dots,N-3$. Then the formula for $\mathcal{O}_N^{(2)}(z_1,z_2)$ is 
\be
\mathcal{O}_N^{(2)}(z_1,z_2)\, 
=\, -\frac{N^2}{\pi^2 N!}\, e^{-N|z_1|^2} e^{-N|z_2|^2} H_{N-2}(z_1,z_2)\, ,
\ee
where $H_{N-2}(z_1,z_2)$ equals the determinant of the $(N-2)$-dimensional square matrix $\mathcal{H}_{N-2}(z_1,z_2)$, where
\begin{multline}
[\mathcal{H}_{N-2}(z_1,z_2)]_{jk}\, =  \, \left(\frac{N^{j+k+6}}{\pi^2 (j+1)! (k+1)! }\right)^{1/2} \\  \int_{\C} \overline{\lambda}^j \lambda^k \Big[|z_1-\lambda|^2 |z_2-\lambda|^2 
+ N^{-1} \left(\overline{z}_1 - \overline{\lambda}\right)(z_2-\lambda)\Big]
\exp(-N|\lambda|^2)\, d^2\lambda\, ,
\end{multline}
for $j,k=0,\dots,N-3$.
These are naturally 5-diagonal because of rotational invariance. However, notice that if $z_1=0$ or $z_2=0$ then they become tri-diagonal again.
Hence, they are more easily calculable in that case.
That is why $\mathcal{O}_N^{(2)}(0,z)$ is calculable.

In a later section, we are going to propose another method to proceed.
We will write down the recursion relation for the 5-diagonal matrix, which is harder than for a tridiagonal
matrix.
Then, even if the formula is not exactly solvable, we argue that it should be asymptotically
solvable.
We give more details in a later section, in particular carrying out the asymptotic approach for the easier
problem of calculating $\mathcal{R}_N^{(2)}(z)$ (which we may check against the exact solution).

\section{Moments and Constraints on the Overlap Functions}
An ideal situation would be to find an explict sum-formula for $\mathcal{O}_N^{(2)}(z_1,z_2)$, just as Lemma \ref{lem:G}
provides for $\mathcal{O}_N^{(1)}(z)$,
but so far, this has not been discovered.
In the next section, we will suggest a rigorous approach which may work to give the asymptotics, even when no explicit formula is known.
For now, let us state the constraints imposed by the moment formula from before.

Recall from (\ref{eq:mnpq})
for any nonnegative integers $p$ and $q$,
$$
m_N((p);(q))\, =\, \int_{\C} z^p \overline{z}^q \mathcal{O}_N^{(1)}(z)\, d^2z 
+ \int_{\C} \int_{\C} z^p \overline{w}^q \mathcal{O}_N^{(2)}(z,w)\, d^2z\, d^2w\, .
$$
Moreover, from the discussion at the end of Section \ref{subsec:Argument},
$m_N((p);(q))$ equals $0$ unless $p=q$, and as noted at the end of Section \ref{sec:Expected}, this is already reflected in the rotational invariance properties of
$\mathcal{O}_N^{(1)}(z)$ and $\mathcal{O}_N^{(2)}(z_1,z_2)$.
Therefore, specializing, we see that
\be
\label{eq:MoTot}
\int_{\C} |z|^{2p} \mathcal{O}_N^{(1)}(z)\, d^2z 
+ \int_{\C} \int_{\C} z_1^p \overline{z}_2^p \mathcal{O}_N^{(2)}(z_1,z_2)\, d^2z_1\, d^2z_2\, =\, 1\, ,
\ee
for each nonnegative integer $p$.
This is the constraint formula.
Let us now analyze this formula, starting with the leading order terms, and going down in order.

\subsection{Cancelling Divergences at Leading Order}

For any fixed $z$ with $|z|<1$, we have
\be
\mathcal{O}_N^{(1)}(z)\, \sim\, N \pi^{-1} (1-|z|^2)\, ,
\ee
and the corrections are actually exponentially small in $N$ (since they arise as the deep part of the right tail of the series for the exponential).
Therefore, integrating, we obtain the leading-order part of the contribution from $\mathcal{O}_N^{(1)}(z)$ from the formula above
\be
\int_{\C} |z|^{2p} \mathcal{O}_N^{(1)}(z)\, d^2z\, \sim\, N \pi^{-1} \int_{\C} |z|^{2p}(1-|z|^2) \mathbf{1}_{[0,1)}(|z|^2)\, d^2z\, .
\ee
The corrections to this formula are not exponentially small, incidentally.
This is because the formula for $\mathcal{O}_N^{(1)}(z)$ is not exponentially close to the exact formula for all $z$ in the complex plane.
For a fixed $|z|>1$ it is easy to see that $\mathcal{O}_N^{(1)}(z)$ is exponentially small (hence exponentially close to the approximating function of $0$ there).
That is because one only has the series for the exponential up to a small number of terms, deep in the left tail.
Near the circle, there are algebraic corrections, not exponential ones.

Nevertheless, let us note that, by making a polar decomposition, $z=r e^{i\theta}$, we obtain
\be
N \pi^{-1} \int_{\C} |z|^{2p}(1-|z|^2) \mathbf{1}_{[0,1)}(|z|^2)\, d^2z\, 
=\, N \int_0^1 t^p (1-t)\, dt\, =\, \frac{N}{(p+1)(p+2)}\, .
\ee
Let us see how this cancels with the leading-order part of the $\mathcal{O}_N^{(2)}$ integral.

We will use Chalker and Mehlig's formula here for the leading-order part, even though we do not yet know
the corrections for the lower-order part near the circle.
Then we get
\begin{multline}
\int_{\C} \int_{\C} z_1^p \overline{z}_2^p \mathcal{O}_N^{(2)}(z_1,z_2)\, d^2z_1\, d^2z_2\\
\sim\, - \int_{\C} \int_{\C} \left(z+\frac{\omega}{2N^{1/2}}\right)^p \left(\overline{z}-\frac{\overline{\omega}}{2N^{1/2}}\right)^p N^2 \pi^{-2} (1-|z|^2) \mathbf{1}_{[0,1)}(|z|^2)\,\\ \frac{1-(1+|\omega|^2)e^{-|\omega|^2}}{|\omega|^4}\, N^{-1} d^2\omega\, d^2z\, ,
\end{multline}
where the $N^{-1}$ associated to the volume-element $d^2\omega$-times-$d^2z$ is to account for the Jacobian of the transformation from $(z_1,z_2)$ to $(z,\omega)$.
Now we will begin to separate this formula into even another decomposition into leading terms, and sub-leading terms.
This is because, in the formulas $z_1^p = (z+\frac{1}{2}N^{-1/2}\omega)^p$ and $\overline{z}_2^p = (\overline{z}-\frac{1}{2}N^{-1/2}\overline{\omega})^p$,
clearly the leading order arises by ignoring the contributions of $\omega$ which each are accompanied by negative powers of $N$.
We really obtain, what we might call the ``leading order, leading order'' term:
\begin{multline}
\int_{\C} \int_{\C} z_1^p \overline{z}_2^p \mathcal{O}_N^{(2)}(z_1,z_2)\, d^2z_1\, d^2z_2\,
\sim\, \\ - N \pi^{-2} \int_{\C} \int_{\C} |z|^{2p}(1-|z|^2) \mathbf{1}_{[0,1)}(|z|^2)\, \frac{1-(1+|\omega|^2)e^{-|\omega|^2}}{|\omega|^4}\, N^{-1} d^2\omega\, d^2z\, .
\end{multline}
Then it is easy to see that this splits. The integral over $\omega$ is 
\be
\int_{\C} \frac{1-(1+|\omega|^2)e^{-|\omega|^2}}{|\omega|^4}\, d^2\omega\,
=\, \pi \int_0^{\infty} \frac{1-(1+t)e^{-t}}{t^2}\, dt\,
=\, \pi \int_0^{\infty} \frac{1}{t^2} \left(\int_0^t s e^{-s}\, ds\right)\, dt\, .
\ee
Integrating-by-parts, it is easy to see that this gives $\pi$.
Therefore, we end up with the exact negative of the leading order contribution by $\mathcal{O}_N^{(1)}$:
\begin{multline}
\int_{\C} \int_{\C} z_1^p \overline{z}_2^p \mathcal{O}_N^{(2)}(z_1,z_2)\, d^2z_1\, d^2z_2\,
\sim\, - N \pi^{-1} \int_{\C} |z|^{2p}(1-|z|^2) \mathbf{1}_{[0,1)}(|z|^2)\,  d^2z\, \\ =\, -\frac{N}{(p+1)(p+2)}\, .
\end{multline}
The fact that these two terms cancel is good, because each diverges, separately; whereas, according to the formula, the exact answer is supposed to be $1$.

\subsection{The Sub-Leading Contribution from $\mathcal{O}_N^{(1)}$}
\label{subsec:sLO1}
For the first integral, we are fortunate that the exact correction is known near the circle.
We will not attempt to keep track of the exponentially-small corrections which are present away from the circle.
Near the circle, the exact corrections are relevant because they are not exponentially small.

Using Corollary 4.3, we know that
\begin{multline}
\int_{\C} |z|^{2p} \mathcal{O}_N^{(1)}(z)\, d^2z
- \frac{N}{(p+1)(p+2)}\, =\,\\ \frac{N^{1/2}}{\pi}\, \int_{\C} \left[\frac{e^{-2u^2}}{\sqrt{2\pi}} - 2 u \Phi(-2u)\right]\Bigg|_{u = N^{1/2} (1-|z|)} |z|^{2p}\, d^2z
+ o(1)\, ,
\end{multline}
where the small term $o(1)$ means that the remainder converges to $0$ as $N \to \infty$.
This remainder includes exponentially small corrections to $\mathcal{O}_N^{(1)}(z)$ away from the circle, as well as the systematic correction terms to the leading-order
behavior near the circle that arise from the Euler-Maclaurin series.
The reason that these correction terms to the Euler-Maclaurin summation formula are $o(1)$ will arise momentarily: even the leading order term is only order-1, constant.

Making the polar decomposition of $z$ and then rewriting $r = 1-N^{-1/2}u$ so that $dr=N^{-1/2}\, du$ (and reversing orientation of the integral), we have
\be
\begin{split}
\int_{\C} |z|^{2p} \mathcal{O}_N^{(1)}(z)\, d^2z
- \frac{N}{(p+1)(p+2)}\, \\
=\, 2 \int_{-\infty}^{N^{1/2}} \left[\frac{e^{-2u^2}}{\sqrt{2\pi}} - 2 u \Phi(-2u)\right](1-N^{-1/2}u)^{2p+1}\, du+ o(1)\\
=\, 2 \int_{-\infty}^{\infty} \left[\frac{e^{-2u^2}}{\sqrt{2\pi}} - 2 u \Phi(-2u)\right]\, du+ o(1)\, .
\end{split}
\ee
In particular, this correction is independent of $p$, modulo vanishingly small remainder terms which are accumulated in the $o(1)$.
Rewriting $u=x/2$ and integrating by parts gives a constant which is equal to $3/2$.

We will not be able to make it to the order-1, constant terms in the $N\to \infty$ asymptotics series (in decreasing powers of $N$).
The reason is that for $\mathcal{O}_N^{(2)}$, we do not have sufficiently precise asymptotics to get to that level.
Instead, what we will do next is to consider what constraints the formula for the moments imposes on $\mathcal{O}_N^{(2)}$.

\subsection{Sub-Leading Divergences in the $\mathcal{O}_N^{(2)}$ Term}

We have now accounted for all the non-vanishing contributions from the $\mathcal{O}_N^{(1)}$ term.
The leading-order divergence cancels with the leading-order divergence of the $\mathcal{O}_N^{(2)}$ term.
The sub-leading order part of the $\mathcal{O}_N^{(1)}$ contribution to the moment is already order-1, constant, and it is independent of $p$.
It equals $3/2$.
Note that the moment itself is also independent of $p$, it is $1$.

Since we do not know the actual formula for $\mathcal{O}_N^{(2)}$, our plan for this section is to consider the proposed formula
for $\mathcal{O}_N^{(2)}$ in the bulk.
That still leads to one other divergent contribution, diverging logarithmically in $N$.
What this must mean is that in the formula for $\mathcal{O}_N^{(2)}(z_1,z_2)$ for $z_1$ and $z_2$
close, and both near the circle, there must be an edge correction,
which leads to a counter-balancing divergence.
This is what we explain in some more detail, now.
This subsection is detailed and technical.

We consider the proposed formula for $\mathcal{O}_N^{(2)}$ that Chalker and Mehlig derived.
This is the correct formula in the bulk, following the argument of their paper, although there is a lower-order correction near the circle.
We will not include the correction on the circle. Instead our calculations will show constraints that must be satisfied for this correction formula.
We use $z_1 = z+\frac{1}{2}N^{-1/2}\omega$ and $z_2 = z-\frac{1}{2}N^{-1/2}\omega$ so that
\be
z_1 \overline{z}_2\,
=\, |z|^2 + \frac{i \mathrm{Im}[\omega \overline{z}]}{N^{1/2}} - \frac{|\omega|^2}{4N}\, .
\ee
Therefore, using the bulk formula we would have
\begin{multline}
\int_{\C} \int_{\C} z_1^p \overline{z}_2^p \mathcal{O}_N^{(2)}(z_1,z_2)\, d^2z_1\, d^2z_2\,\\
\approx\, -\frac{N^2}{\pi^2} \int_{\C} \int_{\C} \left[|z|^2 + \frac{i \mathrm{Im}[\omega \overline{z}]}{N^{1/2}} - \frac{|\omega|^2}{4N}\right]^p
\left(1 -  \left[|z|^2 + \frac{i \mathrm{Im}[\omega \overline{z}]}{N^{1/2}} - \frac{|\omega|^2}{4N}\right]\right)\\
\cdot \frac{1-(1+|\omega|^2)e^{-|\omega|^2}}{|\omega|^4}\,
\mathbf{1}_{[0,1]}\Big(\Big|z\pm\frac{1}{2}N^{-1/2}\omega\Big|^2\Big)\, 
N^{-1} d^2\omega\, d^2z\, ,
\end{multline}
where we use the approximation symbol $\approx$ to remind ourselves that this is only one part of the eventual formula.
Simplifying this, and writing $z = r e^{i\theta}$ and $\omega = \rho e^{i t}$, we have
\begin{multline}
\int_{\C} \int_{\C} z_1^p \overline{z}_2^p \mathcal{O}_N^{(2)}(z_1,z_2)\, d^2z_1\, d^2z_2\, \\ \approx\, -\frac{N}{\pi^2} \int_{\C} \int_{\C} \left[r^2 + \frac{i r\rho \sin(\theta-t)}{N^{1/2}} - \frac{\rho^2}{4N}\right]^p
\left(1 -  \left[r^2 + \frac{i r\rho \sin(\theta-t)}{N^{1/2}} - \frac{\rho^2}{4N}\right]\right)\\
\cdot \frac{1-(1+\rho^2)e^{-\rho^2}}{\rho^4}\,
\mathbf{1}_{[0,4N]}\Big(\Big|\rho \pm 2N^{1/2}re^{i(t-\theta)}\Big|^2\Big)\, 
r\rho\, dr\, d\rho\, d\theta\, dt\, .
\end{multline}

Let us denote $\phi = t-\theta$. Integrating over the extra angular variable, simplifying the power of $\rho$ in the second line, and simplifying the indicator in the second line, we obtain\begin{multline}
\int_{\C} \int_{\C} z_1^p \overline{z}_2^p \mathcal{O}_N^{(2)}(z_1,z_2)\, d^2z_1\, d^2z_2\,
 \\ \approx\, -\frac{2N}{\pi} \int\limits_{r \in [0,1]} \int\limits_{\rho>0} \int\limits_{\phi \in [0,2\pi)} \left[r^2 - \frac{i r\rho \sin(\phi)}{N^{1/2}} - \frac{\rho^2}{4N}\right]^p
\left(1 -  \left[r^2 - \frac{i r\rho \sin(\phi)}{N^{1/2}} - \frac{\rho^2}{4N}\right]\right)\\
\cdot \frac{1-(1+\rho^2)e^{-\rho^2}}{\rho^3}\,
\mathbf{1}_{[0,N^{1/2} R(r,\phi)]}(\rho)\, 
r dr\, d\rho\, d\phi\, ,
\end{multline}
 for
\be
R(r,\phi)\, =\, 2 \left(\sqrt{1-r^2\sin^2(\phi)}-r|\cos(\phi)|\right)\, ,
\ee
arising from the condition
$|\rho \pm 2 N^{1/2} r e^{i\phi}| \leq 2N^{1/2} \Leftrightarrow \rho \leq R(r,\phi) N^{1/2}$

Let us rewrite this once again, this time isolating different functional terms that we wish to consider in more detail:
\begin{multline}
\int_{\C} \int_{\C} z_1^p \overline{z}_2^p \mathcal{O}_N^{(2)}(z_1,z_2)\, d^2z_1\, d^2z_2\, \\ \approx\, -\frac{2N}{\pi} \int\limits_{r \in [0,1]} \int\limits_{\phi \in [0,2\pi)} \left(\int_0^{N^{1/2}R(r,\phi)} F_p(r,\phi,\rho) W(\rho)\, d\rho\right)\, r\, dr\, d\phi\, ,
\end{multline}
where
\be
\label{eq:Fp}
F_p(r,\phi,\rho)\, =\, \left[r^2 - \frac{i r\rho \sin(\phi)}{N^{1/2}} - \frac{\rho^2}{4N}\right]^p
\left(1 -  \left[r^2 - \frac{i r\rho \sin(\phi)}{N^{1/2}} - \frac{\rho^2}{4N}\right]\right)\, ,
\ee
and
\be
W(\rho)\, =\,  \frac{1-(1+\rho^2)e^{-\rho^2}}{\rho^3}\, .
\ee
Now we note that we can expand
\be
\label{eq:expand}
F_p(r,\phi,\rho)\, =\, \sum_{k=0}^{2p} f_p^{(k)}(r,\phi)\, \frac{\rho^k}{N^{k/2}}\, .
\ee
Odd powers of $k$ have $f_p^{(k)}(r,\phi)$ which is an odd function of $\sin(\phi)$. Since the rest of the integral will contribute even factors, this means all odd powers will integrate to zero,
so we only keep track of even powers.
We have already taken account of $f_p^{(0)}(r,\phi)$ which is just $f_p^{(0)}(r) = r^{2p} (1-r^2)$.
This was what gave us the leading order divergence we considered in a past subsection.

Moreover, starting from the even power $k=4$, we have
\be
-\frac{2N}{\pi} \int\limits_{r \in [0,1]} \int\limits_{\phi \in [0,2\pi)} \left(\int_0^{N^{1/2}R(r,\phi)} f_p^{(k)}(r,\phi)\, \frac{\rho^k}{N^{k/2}} W(\rho)\, d\rho\right)\, r\, dr\, d\phi\, 
=\, O(1)\, .
\ee
The reason is that $N \cdot N^{-k/2} = N^{-(k-2)/2}$ which is vanishing. This means that the lower limit of integration is actually contributing a negligible correction, asymptotically for large $N$.
Near the upper limit of integration, we may expand $W(\rho) \sim \rho^{-3}$.
Therefore we obtain near the upper limit, for $k=4,6,\dots$,
\begin{multline}
\label{eq:largek}
-\frac{2}{\pi N^{(k-2)/2}} \int\limits_{r \in [0,1]} \int\limits_{\phi \in [0,2\pi)} f_p^{(k)}(r,\phi)\,\left(\int_0^{N^{1/2}R(r,\phi)}  \rho^{k-3}\, d\rho\right)\, r\, dr\, d\phi\\
=\, -\frac{2}{\pi} \int\limits_{r \in [0,1]} \int\limits_{\phi \in [0,2\pi)} f_p^{(k)}(r,\phi)\, \frac{[R(r,\phi)]^{(k-2)/2}}{k-2}\, r\, dr\, d\phi
+o(1)\, =\, O(1)\, .
\end{multline}

This only leaves the term with $k=2$ which might diverge.
Indeed, for this, we just have
$$
-\frac{2}{\pi}\,
\int\limits_{r \in [0,1]} \int\limits_{\phi \in [0,2\pi)} f_p^{(2)}(r,\phi)\,\left(\int_0^{N^{1/2}R(r,\phi)}  \frac{1-(1+\rho^2)e^{-\rho^2}}{\rho}\, d\rho\right)\, r\, dr\, d\phi\, .
$$
The only divergent part of this arises near the upper limit for the $\rho$ integral which gives $\ln(N^{1/2} R(r,\phi)) = \frac{1}{2} \ln(N) + \ln(R(r,\phi))$,
so the logarithmic divergence is 
\begin{multline}
-\frac{2}{\pi}\,
\int\limits_{r \in [0,1]} \int\limits_{\phi \in [0,2\pi)} f_p^{(2)}(r,\phi)\,\left(\int_0^{N^{1/2}R(r,\phi)}  \frac{1-(1+\rho^2)e^{-\rho^2}}{\rho}\, d\rho\right)\, r\, dr\, d\phi\\
=\,
-\frac{\ln(N)}{\pi} \int\limits_{r \in [0,1]} \int\limits_{\phi \in [0,2\pi)} r f_p^{(2)}(r,\phi)\, dr\, d\phi
+ O(1)\, .
\end{multline}
It is easy to see that
\be
f_p^{(2)}(r,\phi)\, =\,
\frac{1}{4}\, r^{2p} - \frac{p}{4}\, r^{2p-2} (1-r^2) - \frac{p(p-1)}{2}\, r^{2p-2}(1-r^2) \sin^2(\phi) + p r^{2p}\sin^2(\phi)\, .
\ee
Therefore, we have
\be
\begin{split}
\int_{\phi \in [0,2\pi]} f_p^{(2)}(r,\phi)\, d\phi\, 
&=\, 2\pi
\left(\frac{1}{4}\, r^{2p} - \frac{p}{4}\, r^{2p-2} (1-r^2) - \frac{p(p-1)}{4}\, r^{2p-2}(1-r^2) + \frac{p}{2}\, r^{2p}\right)\\
&=\, \frac{\pi}{2} \left[(p+1)^2 r^{2p} - p^2 r^{2p-2}\right]\, .
\end{split}
\ee
Therefore, the sub-leading order divergence is now
\be
\label{eq:sL}
\int_{\C} \int_{\C} z_1^p \overline{z}_2^p \mathcal{O}_N^{(2)}(z_1,z_2)\, d^2z_1\, d^2z_2
+ \frac{N}{(p+1)(p+2)}\,
=\, -\frac{1}{4}\, \ln(N) + O(1)\, .
\ee
We may consider this particular form.
It is independent of $p$.
Near the circle, and for $z_1$ near $z_2$, the form of $z_1^p \overline{z}_2^p$, to leading order is just $|z|^{2p}$ which is just $1$, because $z$ is near the circle.
This is the same explanation for the reason that the order-1, constant term coming from $\mathcal{O}_N^{(1)}$ term is independent of $p$.
We also know that the moment must be independent of $p$.

One could also try to calculate the order-1 contributions at this point, coming just from
the bulk formula for $\mathcal{O}_N^{(2)}(z_1,z_2)$.
One could then check whether these combine to a constant independent of $p$.
That would be yet another strong check that Chalker and Mehlig's formula for $\mathcal{O}_N^{(2)}(z_1,z_2)$  is true to very high accuracy in the bulk, and only needs an edge correction
near the circle.

It would be best to have a sufficiently explict formula for $\mathcal{O}_N^{(2)}(z_1,z_2)$ to allow one to see the correction near the circle.
Then we could have an answer to settle this.
Next, we propose a method which we believe could potentially provide this.

\section{Proposal to Rigorously Approach Chalker and Mehlig's Result}
There are various ways to try to prove Chalker and Mehlig's formula for the bulk behavior of $\mathcal{O}_N^{(2)}$.
One way is to try to fill in the details to make Chalker and Mehlig's argument rigorous.
Their idea is to express $\mathcal{O}_N^{(2)}$ in terms of the expectation of a function of the eigenvalues, and then use the known eigenvalue marginal
for the complex Ginibre ensemble.

Here we want to propose a second method.
The formula for $\mathcal{O}_N^{(2)}(z_1,z_2)$ is the determinant of a 5-diagonal matrix.
One may express such a determinant through a recursion relation, although the recursion relation is significantly more complicated than in the tridiagonal case.
It is higher order, and it is a vector valued recursion relation for a vector with dimension greater than $1$.
We will not explicate this, here.
It is well-known, it just follows from Cramer's rule, and it is widely used in numerical codes.

Instead, what we want to advocate here is solving recursion relations, at least asymptotically for large $N$, using adiabatic theory.
We have not tried this yet for $\mathcal{O}_N^{(2)}(z_1,z_2)$.
There may be formidable difficulties which obstruct this approach, but let us demonstrate the idea for an easier problem: re-deriving the formula for $\mathcal{R}_N^{(1)}(z)$.
This leads to an easier problem.
The key trick for this particular problem is to realize that $\mathcal{R}_N^{(1)}(z)$, at least for the leading-order asymptotic formula,
is constant in $z$ for $|z|<1$.

\subsection{The Recurrence Relation for $\mathcal{R}_N^{(1)}(z)$ Using Matrices}

We are treating the case of $\mathcal{R}_N^{(1)}(z)$ as a simpler toy model,
in lieu of treating the real problem of interest which is $\mathcal{O}_N^{(2)}(z_1,z_2)$.
We hope to be able to handle $\mathcal{O}_N^{(2)}(z_1,z_2)$ later, in another paper.

Recall from (\ref{eq:RD}) that $\mathcal{R}_N^{(1)}(z) = \pi^{-1} [(N-1)!]^{-1} \exp(-N |z|^2) D_{N-1}(z)$, which means from Stirling's formula that
\be
\mathcal{R}_N^{(1)}(z)\, \sim\, \frac{1}{\pi} \cdot \frac{1}{\sqrt{2 \pi N}}\, e^{-(N-1)\ln(N)+N(1-|z|^2)} D_{N-1}(z)\, .
\ee
Moreover, recall that there is a recursion relation in (\ref{eq:Dr}). Namely, defining $D_{N-1}(\sigma^{-2},z) = D_{N-1}(\sigma^{-1}N^{-1/2}z)$,
it happens that
$$
D_{n+1}(\sigma^{-2},z)\, =\, (\sigma^{-2}|z|^2+n+1) D_n(\sigma^{-2},z)-\sigma^{-2}n|z|^2D_{n-1}(\sigma^{-2},z)\, .
$$
Let us fix $\sigma^2=N^{-1}$ as Chalker and Mehlig do.
Then
\be
D_{n+1}(N,z)\, =\, (N|z|^2+n+1) D_n(N,z)-Nn|z|^2D_{n-1}(N,z)\, .
\ee
Also, since the answer only depends on the magnitude of $z$, let us write $r=|z|$ so
\be
D_{n+1}(N,r)\, =\, (Nr^2+n+1) D_n(N,r)-Nnr^2D_{n-1}(N,r)\, .
\ee
We want to calculate $\mathcal{R}_N^{(1)}(r)$ which is asymptotically given by
\be
\mathcal{R}_N^{(1)}(r)\, \sim\, \frac{1}{\pi} \cdot \frac{1}{\sqrt{2 \pi N}}\, e^{-(N-1)\ln(N)+N(1-r^2)} D_{N-1}(N,r)\, .
\ee
Now since we have a second-order recursion relation, let us define a two-dimensional vector $v_n = [D_{n-1}(N,r),D_n(N,r)]^*$.
(All our vectors and matrices will be real but we use the adjoint instead of the transpose because we want to keep the symbol
$T$ for other purposes.)
Then the recursion relation says that
\be
v_{n+1}\, =\, A_n v_n\, ,\qquad
A_n\, =\, 
\begin{bmatrix} 0 & 1 \\
-Nnr^2 & N r^2+n+1
\end{bmatrix}\, ,
\ee
and we want $D_{N-1}(N,r) = e_2^* v_{N-1}$, where $\{e_1,e_2\}$ is the standard basis for $\R^2$.
In order to have a simpler formula, we note that we can write $v_1 = A_0 e_2$, for $A_0$ defined as above. 
In seeking $\mathcal{R}_N^{(1)}(r)$, we really have
\be
\mathcal{R}_N^{(1)}(r)\, \sim\, \frac{1}{\pi} \cdot \frac{1}{\sqrt{2 \pi N}} \, e^{-(N-1)\ln(N)+N(1-r^2)} e_2^* A_{N-2} \cdots A_1 A_0 e_2\, .
\ee
The idea is to try to express this using the spectral decomposition of the matrices $A_n$,
where we use the fact that the matrices $A_n$ are {\em varying slowly} in $n$, as much as possible.
This is why we call this the {\em adiabatic approach.}

\subsection{Spectral Formulas and Summary of Main Contribution}

We may summarize the spectral information as 
\be
\begin{gathered}
\lambda_n^{\pm}\, =\, \frac{N}{2} \left(r^2+\frac{n+1}{N} \pm \sqrt{\left(\frac{n+1}{N}-r^2\right)^2+4r^2N^{-1}}\right)\, ,\quad \\
V_n^{\pm}\, =\, \begin{bmatrix} 1 \\ \lambda_n^{\pm} \end{bmatrix}\, ,\quad
W_n^{\pm}\, =\, \pm\frac{1}{\lambda_n^+ - \lambda_n^-} \begin{bmatrix} - \lambda_n^{\pm} \\ 1 \end{bmatrix}\, ;\\
A_n V_n^{\pm}\, =\, \lambda_n^{\pm} V_n^{\pm}\, ,\qquad
(W_n^{\pm})^* A_n\, =\, \lambda_n^{\pm} (W_n^{\pm})^*\, ,\qquad
(W_n^{\sigma})^* V_n^{\tau}\, =\, \delta_{\sigma,\tau}\, ,\ \text{ for $\sigma,\tau \in \{+1,-1\}$.}
\end{gathered}
\ee
In particular, $A_n = \lambda_n^+ V_n^+ (W_n^+)^* + \lambda_n^- V_n^- (W_n^-)^*$.
Therefore, we can rewrite the conclusion of the recursion relation as 
\be
\begin{split}
\mathcal{R}_N^{(1)}(r)\, 
\sim\, \frac{1}{\pi} \cdot \frac{1}{\sqrt{2 \pi N}} \, e^{-(N-1)\ln(N)+N(1-r^2)} e_2^* A_{N-2} \cdots A_1 A_0  e_2\\
=\, \frac{1}{\pi} \cdot \frac{1}{\sqrt{2 \pi N}} \, e^{-(N-1)\ln(N)+N(1-r^2)}
\sum_{\sigma \in \{+1,-1\}^{N-1}} \left(\prod_{n=0}^{N-2} \lambda_n^{\sigma(n)}\right)
([W_0^{\sigma(0)}]^* e_2) (e_2^* V_{N-2}^{\sigma(N-2)}) \\
&\hspace{9cm}
\cdot \left(\prod_{n=0}^{N-1} [W_{n+1}^{\sigma(n+1)}]^* V_n^{\sigma(n)}\right)\, .
\end{split}
\ee
Anticipating that the main contribution to this sum will be $\sigma(0)=\dots=\sigma(N-1)=+1$, we may rewrite this as 
\be
\mathcal{R}_N^{(1)}(r)\, 
\sim\, \frac{1}{\pi} \cdot \frac{1}{\sqrt{2 \pi N}} \, e^{-(N-1)\ln(N)+N(1-r^2)} \mathcal{M}_N(r) \mathcal{P}_N(r)\, ,
\ee
where $\mathcal{M}_N(r)$ is the ``main term''
\be
\mathcal{M}_N(r)\,
=\,
\left(\prod_{n=0}^{N-2} \lambda_n^{+}\right)
([W_0^{+}]^* e_2) (e_2^* V_{N-2}^{+}) 
\cdot \left(\prod_{n=0}^{N-1} [W_{n+1}^{+}]^* V_n^{+}\right)\, ,
\ee
and $\mathcal{P}_N(r)$ will be a series of perturbations
\be
\mathcal{P}_N(r)\,
=\, 
\sum_{\sigma \in \{+1,-1\}^{N-1}} 
\left(\prod_{n=0}^{N-2} \frac{\lambda_n^{\sigma(n)}}{\lambda_n^+}\right)
\frac{([W_0^{\sigma(0)}]^* e_2) (e_2^* V_{N-2}^{\sigma(N-2)})}{([W_0^{+}]^* e_2) (e_2^* V_{N-2}^{+})}
\cdot \left(\prod_{n=0}^{N-1} \frac{[W_{n+1}^{\sigma(n+1)}]^* V_n^{\sigma(n)}}
{[W_{n+1}^{+}]^* V_n^{+}}\right)\, .
\ee
We know that we are trying to find that the leading order behavior of $\mathcal{R}_N^{(1)}(r)$ is as follows:
it is constant, equal to $\pi^{-1}$, for $r<1$, and it is exponentially small for $r>1$.
We will not try to recover the boundary behavior near $r=1$ in this note.
(In fact, what we hope to be able to do in a later paper is to calculate $\mathcal{O}_N^{(2)}(z_1,z_2)$
in a similar way, and especially to determine the edge behavior when $z_1$ and $z_2$ are near the circle.)
Let us quickly note how we may dispense with the $r>1$ case so that we may focus on $r<1$.

The largest contribution to $\mathcal{M}_N(r)$ comes from the product of eigenvalues
\begin{multline}
\left(\prod_{n=0}^{N-2} \lambda_n^{+}\right)
=\, \exp\left[\sum_{n=0}^{N-2} \ln( \lambda_n^{+})\right]\\
=\,e^{(N-1) \ln(N)} \exp\left(\sum_{n=0}^{N-2} \ln\left[ 
\frac{1}{2}\left(r^2+\frac{n+1}{N} + \sqrt{\left(\frac{n+1}{N}-r^2\right)^2+4r^2N^{-1}}\right)
\right]
\right)\, .
\end{multline}
Moreover, defining $t_{n+1} = (n+1)/N$, the sum is $(N-1)$ times a Riemann sum approximation so that:
\be
\begin{split}
\frac{1}{N}\ln\left[e^{-(N-1)\ln(N)} \left(\prod_{n=0}^{N-2} \lambda_n^{+}\right)\right]\, 
&=\,
\int_0^1 \ln\left(\frac{1}{2}\left(r^2+t+\sqrt{(t-r^2)^2+4r^2N^{-1}}\right)\right)\, dt + o(1)\\
&=\, \int_0^1 \ln\left(\frac{1}{2}\left(r^2+t+\sqrt{(t-r^2)^2}\right)\right)\, dt + o(1)\\
&=\, \int_0^1 \ln(\max\{r^2,t\})\, dt + o(1)\, ,
\end{split}
\ee
where the remainder term $o(1)$ is a quantity which converges to $0$ as $N \to \infty$.
Hence we may see, by integrating, that
\be
\lim_{N \to \infty} \frac{1}{N}\ln\left[e^{-(N-1)\ln(N)} \left(\prod_{n=0}^{N-2} \lambda_n^{+}\right)\right]\, 
=\, \begin{cases} \ln(r^2) & \text{ if $r\geq 1$,}\\
r^2-1 & \text{ if $r \in [0,1]$.}
\end{cases}
\ee
This means that, incorporating the exponential part of the prefactor for $\mathcal{R}_N^{(1)}(r)$,
\be
\lim_{N \to \infty} \frac{1}{N}\ln\left[e^{-(N-1)\ln(N)+N(1-r^2)} \left(\prod_{n=0}^{N-2} \lambda_n^{+}\right)\right]\, 
=\, \begin{cases} 0 & \text { if $r \in [0,1]$,}\\
\ln(r^2)-1+r^2 & \text{ if $r>1$,}
\end{cases}
\ee
and it is easy to see that $\ln(x) \leq x-1$ for all $x \in (0,\infty)$ by convexity of $-\ln(x)$.
For $r>1$ this is exponentially small: to leading order $e^{N[\ln(r^2)-1+r^2]}$.
We claim that no other factor is exponentially large, so that we obtain
\be
\lim_{N \to \infty} N^{-1} \ln(\mathcal{R}_N^{(1)}(r))\, 
=\, \begin{cases} 0 & \text { if $r \in [0,1]$,}\\
\ln(r^2)-1+r^2 & \text{ if $r>1$.}
\end{cases}
\ee
Therefore, we will henceforth assume $r<1$.

When $r<1$, we claim that we need to do a more careful analysis of the product.
The time scale $t_n=n/N$ is too rough when $t_n$ is near $r^2$.
The purely discrete scale $n$ is too fine.
Therefore, we use the intermediate time scale $T_n = (t_n-r^2)N^{1/2}$, instead.
Then we may rewrite
\begin{multline}
\lambda_n^{+}\,
=\, N \exp(\psi^{+}(T_{n+1}))\, ,\qquad \\
\psi^{+}(T_{n+1})\, =\, \ln\left(r^2 + \frac{1}{2}\, N^{-1/2} T_{n+1} + \frac{1}{2}\, N^{-1/2} \sqrt{T_{n+1}^2 + 4 r^2}\right)
\end{multline}
so that
\be
\left(\prod_{n=0}^{N-2} \lambda_n^{+}\right)\,
=\, N^{N-1} \exp\left(\sum_{n=0}^{N-1} \psi^+(T_{n+1})\right)\, .
\ee
Then we  use the Euler-Maclaurin summation
formula to obtain all other terms in the asymptotic series which are significant, including some boundary terms that come with the Euler-Maclaurin formula.
(To do an integral such that $\int_{-r^2 N^{1/2}}^{(1-r^2)N^{1/2}} \psi^+(T)\, dT$, one may find it useful to define $T=2r\sinh(x)$ so that
$dT = 2r \cosh(x)\, dx$ and $\sqrt{T^2+4r^2} = 2r \cosh(x)$, as well.
Doing all this leads to
\be
\left(\prod_{n=0}^{N-2} \lambda_n^{+}\right)\, \sim\, e \cdot r (1-r^2) e^{N\ln(N) - N(1-r^2)}\, .
\ee
It is also easy to use the definitions of $V_n^+$ and $W_n^+$ to show that
\be
[W_0^{+}]^* e_2\,
\sim\, N^{-1} r^{-2}\quad \text{ and } \quad
e_2^* V_{N-2}^+\, \sim\, N\, .
\ee
Using the Euler-Maclaurin summation formula, one may also prove that
\be
\prod_{n=0}^{N-1} [W_{n+1}^{+}]^* V_n^{+}\,
\sim\,  \frac{r}{1-r^2}\, N^{-1/2}\, .
\ee
The details of the Euler-Maclaurin summation formula for this product as well as for the product of the eigenvalues
are not trivial. (The product of the eigenvalues is harder than the product of the inner-products.)
They may be done, in particular, by using the intermediate time-scale parameter $T_n$.
Therefore, we obtain
\be
\mathcal{M}_N(r)\, =\, e \sqrt{N}\, e^{(N-1)\ln(N) - N(1-r^2)}\, .
\ee
Therefore, since $$\mathcal{R}_N^{(1)}(r) \sim \pi^{-1} (2\pi N)^{-1/2} \exp(-(N-1)\ln(N)+N(1-r^2)) \mathcal{M}_N(r) \mathcal{P}_N(r)$$ we see that
\be
\mathcal{R}_N^{(1)}(r)\, \sim\, \pi^{-1}\, \frac{e}{\sqrt{2\pi}}\,   \mathcal{P}_N(r)\, .
\ee
Now we will argue that $\mathcal{P}_N(r)$ is actually independent of $r$, to leading order.

\section{Invariance of the Perturbation Series $\mathcal{P}_N(r)$}

Let us write
\be
\mathcal{P}_N(r)\, =\, \sum_{\sigma \in \{+1,-1\}^{N-1}} \mathcal{P}_N(\sigma;r)\, ,
\ee
for 
\be
\mathcal{P}_N(\sigma;r)\,
=\, 
\left(\prod_{n=0}^{N-2} \frac{\lambda_n^{\sigma(n)}}{\lambda_n^+}\right)
\frac{([W_0^{\sigma(0)}]^* e_2) (e_2^* V_{N-2}^{\sigma(N-2)})}{([W_0^{+}]^* e_2) (e_2^* V_{N-2}^{+})}
\cdot \left(\prod_{n=0}^{N-1} \frac{[W_{n+1}^{\sigma(n+1)}]^* V_n^{\sigma(n)}}
{[W_{n+1}^{+}]^* V_n^{+}}\right)\, .
\ee
Let us think of $\sigma$ as a sequence of switches, from the $+$ state to the $-$ state, or vice-versa.

Using the notation $t_n = n/N$ and $T_n = N^{1/2} (t_n-r^2)$, we may write
\be
\begin{split}
[W_{n+1}^{\tau}]^* (V_n^{\sigma}-V_{n+1}^{\sigma})\,
&=\, -\frac{\tau N^{-1/2}}{2 \sqrt{T_{n+1}^2+4r^2}}
\left(1+\sigma\, \frac{T_{n}+T_{n+1}}{\sqrt{T_{n}^2+4r^2}+\sqrt{T_{n+1}^2+4r^2}}\right)\\
&=\, -\frac{\tau}{2 \sqrt{T_{n+1}^2+4r^2}}
\left(1+\sigma\, \frac{T_{n}+T_{n+1}}{\sqrt{T_{n}^2+4r^2}+\sqrt{T_{n+1}^2+4r^2}}\right)\, \Delta T_n\, ,
\end{split}
\ee
where we define $\Delta T_n = T_{n+1}-T_n=N^{-1/2}$.
This means that in a time $\Delta T_n$ there is a factor proportional to $\Delta T_n$ contributing to $\mathcal{P}_N(\sigma;r)$,
if we switch from $\sigma=+$ to $\tau=-$ or from $\sigma=-$ to $\tau=+$ because in these cases $[W_{n+1}^{\tau}]^*V_{n+1}^{\sigma}=0$.
This is representative of a Poisson process of jumps.

Moreover, if at $a$ one jumps from $+$ to $-$ and at $b$ one jumps back to $+$, then for all $n \in \{a,\dots,b-1\}$
there is a contribution to $\mathcal{P}_N(\sigma;r)$
equal to
\be
\begin{split}
\frac{\lambda_n^{-}}{\lambda_n^+}
\cdot \frac{[W_{n+1}^{-}]^* V_n^{-}}
{[W_{n+1}^{+}]^* V_n^{+}}\, 
&=\, \frac{r^2+\frac{1}{2}N^{-1/2}T_{n+1}-\frac{1}{2}N^{-1/2}\sqrt{T_{n+1}^2+4r^2}}
{r^2+\frac{1}{2}N^{-1/2}T_{n+1}+\frac{1}{2}N^{-1/2}\sqrt{T_{n+1}^2+4r^2}}\\
&\qquad \qquad
\cdot \frac{1 + \frac{1}{2\sqrt{T_{n+1}^2+4r^2}}
\left(1- \frac{T_{n}+T_{n+1}}{\sqrt{T_{n}^2+4r^2}+\sqrt{T_{n+1}^2+4r^2}}\right)\, \Delta T_n}
{1 - \frac{1}{2\sqrt{T_{n+1}^2+4r^2}}
\left(1+ \frac{T_{n}+T_{n+1}}{\sqrt{T_{n}^2+4r^2}+\sqrt{T_{n+1}^2+4r^2}}\right)\, \Delta T_n}\, ,
\end{split}
\ee
and this quantity is asymptotic to 
$\exp\left(-\left[\frac{\sqrt{T_{n+1}^2+4r^2}}{r^2+\frac{1}{2}N^{-1/2}T_{n+1}}
-\frac{1}{\sqrt{T_{n+1}^2+4r^2}}\right] \Delta T_n\right)$,
when one takes $N \to \infty$ if one also takes a sequence of $T_{n_N}$ such that $|T_{n_N}|/N^{1/2} \to 0$.
Moreover the product is decreasing very rapidly as $|T_n|$ gets large on an order-1 scale.
Therefore, the correction to this asymptotic formula is neglible, for the purpose of calculating the leading order behavior of $\mathcal{P}_N(r)$.
Therefore, defining $\mathcal{P}^{++}_N(r)$ to be the sum of those $\mathcal{P}_N(\sigma;r)$ with $\sigma$ starting at $+$ at the left endpoint
and returning to $+$ at the right endpoint, with some number of intervals of $-$ in between, we have the effect of switching from $+$ to $-$,
staying at $-$ for an interval, and then switching back.
This gives
\be
\begin{split}
\lim_{N \to \infty} \mathcal{P}^{++}_N(r)\,
&=\, 1 + \sum_{K=1}^{\infty} \int_{-\infty<S_1<\dots<S_{2K}<\infty}
\prod_{k=1}^{K} \left[\frac{1}{2\sqrt{S_{2k-1}^2+4r^2}}\left(1+\frac{S_{2k-1}}{\sqrt{S_{2k-1}^2+4r^2}}\right)\right]\\
&\hspace{3cm}
\exp\left(-\sum_{k=1}^{K}\int_{S_{2k-1}}^{S_{2k}} \left[\frac{\sqrt{s^2+4r^2}}{r^2} - \frac{1}{\sqrt{s^2+4r^2}}\right]\, ds\right)\\
&\hspace{3cm}
\prod_{k=1}^{K} \left[-\frac{1}{2\sqrt{S_{2k}^2+4r^2}}\left(1-\frac{S_{2k}}{\sqrt{S_{2k}^2+4r^2}}\right)\right]\, 
dS_1\, \cdots\, dS_{2n}\\
&=\, 1 + \sum_{K=1}^{\infty} (-1)^K \int_{-\infty<x_1<\dots<x_{2K}<\infty}
\prod_{k=1}^{K} \left(\frac{1}{[1+\exp(-2x_{2k-1})][1+\exp(2x_{2k})]}\right)\\
&\hspace{3cm}
\exp\left(-\sum_{k=1}^{K} \int_{x_{2k-1}}^{x_{2k}} [4 \cosh^2(x)-1]\, dx\right)\, dx_1\, \cdots\, dx_{2K}\, ,
\end{split}
\ee
where we made the change of variables $S_k = 2r \sinh(x_k)$, which is useful, as we have also mentioned before.
Let us comment on where the $r$-dependence went.
In fact the limits of integration for $S_1$ and $S_{2K}$ should be $-r^2N^{1/2}<S_1$ and $S_{2K}<(1-r^2)N^{1/2}$.
Since the exponentials are negative (and growing in magnitude), the integrand is converging rapidly.
Therefore, we can replace the limits of integration, by allowing integrals over all space, with a correction due to the tails of the integrals which are exponentially small.
Then the substitution we have made from $S_k$ to $x_k$ eliminates the $r$ dependence, entirely.
Finally, we mention that we can do the integral in the exponential to simplify the formula, a bit:
\begin{multline}
\lim_{N \to \infty} \mathcal{P}^{++}_N(r)\, \\
=\, 1 + \sum_{K=1}^{\infty} (-1)^K \int_{-\infty<x_1<\dots<x_{2K}<\infty}
\exp\left(\sum_{k=1}^{K} \left[
-\ln\left(1+e^{-2x_{2k-1}}\right)+\sinh(2x_{2k-1})-x_{2k-1}\right]\right)
\\
\exp\left(-\sum_{k=1}^{K} \left[
\ln\left(1+e^{2x_{2k}}\right)+\sinh(2x_{2k})-x_{2k}\right]\right)
\, dx_1\, \cdots\, dx_{2K}\\
=\, 1 + \sum_{K=1}^{\infty} (-1)^K \\ \cdot\int_{-\infty<x_1<\dots<x_{2K}<\infty}
e^{-\sum_{k=1}^{K} \left(\ln[\cosh(x_{2k-1})]+\ln[\cosh(x_{2k})]+\sinh(2x_{2k})-\sinh(2x_{2k-1})\right)}
\, dx_1\, \cdots\, dx_{2K}\, .
\end{multline}
Again, note that this is rapidly decreasing as $x_1 \to -\infty$ or $x_{2K} \to \infty$.
To get the analogous terms $\mathcal{P}^{+-}_N(r)$, $\mathcal{P}^{-+}_N(r)$
and $\mathcal{P}^{--}_N(r)$, we can just alter this formula essentially by taking $x_1 \to -\infty$
or $x_{2K} \to \infty$ or both.
(This is not entirely correct because we lose terms corresponding to the density for crossing, but morally it is still correct
because the terms remaining are certainly going to $0$.)
Therefore 
\be
\lim_{N \to \infty} \mathcal{P}_N(r)\,
=\, \lim_{N \to \infty} \mathcal{P}_N^{++}(r)\, .
\ee
Since we know that $\lim_{N \to \infty} \mathcal{R}_N^{(1)}(r)$ must equal $\pi^{-1}$ on the disk (for instance because the area of the disk is 1)
this leaves the calculation to show that
\begin{multline}
1  + \sum_{K=1}^{\infty} (-1)^K \int_{-\infty<x_1<\dots<x_{2K}<\infty}
e^{-\sum_{k=1}^{K} \left(\ln[\cosh(x_{2k-1})]+\ln[\cosh(x_{2k})]+\sinh(2x_{2k})-\sinh(2x_{2k-1})\right)}
dx_1\cdots dx_{2K}\\
\stackrel{?}{=}\, \frac{\sqrt{2\pi}}{e}\, .
\end{multline}
At this time we cannot see a direct method to prove this,
but we hope to explore it in a later paper.

\section{Summary and Outlook}

We have considered the complex Ginibre ensemble.
We consider the problem of calculating the mixed matrix moments to be a nice pedagogical problem.
It may be used to illustrate the method of using concentration of measure to derive nonlinear recursion relations.
This method is particularly important in spin glass theory, where it led to the Ghirlanda-Guerra identities, which are critical
to those models.

The most natural connection between spin glasses and random matrices are the spherical spin glasses
of \cite{KTJ} and \cite{CS}.
This has been studied vigorously with very detailed results.
See for example \cite{ABAC}.
The relation we have drawn between the overlaps in spin glasses and the moments in random
matrix theory is mainly illustrative, to suggest the central role of concentration-of-measure (COM).
In addition to spin glass theory and random matrix theory, the idea of using COM to derive low-dimensional nonlinear equations
to replace linear equations in high dimensions is helpful in a variety of contexts \cite{CK}.

The mixed matrix moments for the complex Ginibre ensemble are particularly nice moments to consider because the combinatorics
are as simple as possible. (Indeed it is somewhat simpler than the usual Catalan numbers that arise in the GUE/GOE moments
or the bipartite Catalan numbers that arise in the Mar\v{c}enko-Pastur law.)
Also, they are not as well-studied as the other moments for the classical Gaussian matrix ensembles,
but they are still well-studied.
However, an interesting facet which has not been exhaustively studied is their relation to the overlap functions defined by Chalker and Mehlig.

Chalker and Mehlig's papers are extremely interesting and introduce what certainly seems like a key 
object in random matrix theory that has not been taken up sufficiently yet by mathematicians.
It is recognized as a key result by theoretical and mathematical physicists.
See, for instance, the recent paper \cite{Burda}.

Chalker and Mehlig did not consider the application of calculating the mixed matrix moments from their overlap functions.
Indeed, since the mixed matrix moments are already known, the reverse problem seems more reasonable, but it would probably be very difficult to calculate the overlap functions just from the mixed matrix moments.
However, what is true is that, if one takes Chalker and Mehlig's formula for the bulk overlap functions,
then the mixed matrix moments do place some constraints on the edge behavior,
as we have shown. 

We have proposed a possible method for calculating $\mathcal{O}_N^{(2)}(z_1,z_2)$, asymptotically,
but we have not carried out this suggestion.
We did illustrate it by re-deriving $\mathcal{R}_N^{(1)}(z)$ by treating the second-order recursion formula
as an adiabatic matrix evolution problem.

Now we would like to suggest another interesting direction for further study.
Fyodorov and Mehlig, and Fyodorov and Sommers, calculated two very interesting examples
of non-Hermitian random matrices for which they obtained exact expressions for the overlap functions \cite{FY1,FY2}.
They did not yet calculate the mixed matrix moments for these random variables.
It would be an ideal problem to do so, and check the formulas linking the overlap functions and the mixed matrix moments.

In a private communication with Shannon Starr, Fyodorov has explained that the eigenfunction non-orthogonality
in the systems considered in \cite{FY1,FY2} has physical relevance.
The overlap was shown by Fyodorov and Savin to give the resonance shift if one perturbs a scattering
system \cite{FY3}.
This was even experimentally verified recently \cite{Gros}.

Finally, the first two overlap functions only help with calculating mixed matrix moments of the Ginibre ensemble of the 
form $\operatorname{tr}[A^p (A^*)^p]$ for $p=1,2,\dots$.
In order to calculate mixed matrix moments for more than two factors one needs higher order overlap functions.
Given the difficulty to calculate the first two, this is a formidable problem,
but it might be a reasonable exact calculation for the matrix ensembles considered by Fyodorov and his collaborators.

\chapter{Mallows Random Permutations}
\section{A $q$-Stirling's Formula}
Before we say anything about a $q$-deformed Stirling's formula, recall that Stirling's formula says that 
$$
n!\sim \sqrt{2\pi n}\frac{n^n}{e^n}
$$
This is an asymptotic formula.  We use the $\sim$ symbol to denote that 
$$
\lim_{n\rightarrow\infty} \frac{n!}{\sqrt{2\pi n}\frac{n^n}{e^n}}=1
$$
It is worthwhile to note that this formula can be proved using the Euler Maclaurin formula discussed in chapter 2.  

For fixed $0<q<1$, we define 
$$
[n]_q=\frac{1-q^n}{1-q}
$$
We can then define a $q$-deformed factorial as 
$$
[n]_q!=[n]_q[n-1]_q\dots[1]_q=\prod_{k=1}^n\frac{1-q^k}{1-q}
$$
For notational convenience, we will denote $[n]_q$ by $[n]$ and $[n]_q!$ as $[n]!$, suppressing the dependence on $q$.  In a work in progress with Shannon Starr, we require a Stirling type formula (or asymptotic formula) for $[n]!$.  A similar formula was first proved by Moak \cite{Moak}.  At the time this formula was proved, we were unaware of his work.  As our methods and approximation are slightly different, we include our verison and proof of the $q$-Stirling formula here.  
\begin{thm}
For $\beta\in\mathbb{R}$, let $q=\exp(-\beta/n)$.  Let $[n]!$ denote $[n]_q!$ for this particular $q$.  Then we have 
$$
\ln\left(\frac{[n]!}{n!}\right)=n\int_0^1\ln\left(\frac{1-e^{-\beta y}}{\beta{y}}\right)dy +\frac{\beta}{2}+\frac{1}{2}\ln\left(\frac{1-e^{-\beta}}{\beta}\right)+R_n(\beta)
$$
where $R_n(\beta)$ is a remainder term and $R_n(\beta)\rightarrow 0$ as $n\rightarrow\infty$. 
\end{thm}
\begin{proof}
First consider \[\ln\left(\frac{[n]!}{n!}\right)\] As mentioned previously, we want $q$ to be going to $1$ as $n\rightarrow\infty$, so we are looking at $q=e^{-\frac{\beta}{n}}$ for fixed $\beta$.  Notice that \[\ln\left(\frac{[n]!}{n!}\right)=\sum_{k=1}^n \ln\left(\frac{1-q^k}{(1-q)k}\right)\]  

To approximate this sum, we use the Euler-MacLaurin approach and compare the sum to

 \[\int_1^n \ln\left(\frac{1-q^x}{(1-q)x}\right)dx\]

For ease of notation, let \[f(x)=\ln\left(\frac{1-q^x}{(1-q)x}\right)\]
In order to make this comparison, we will first compare 
$\frac{1}{2}f(k+1)+\frac{1}{2}f(k)$ to \[ \int_k^{k+1} f(x)dx\] and then sum over the $k$'s.   Using the fundamental theorem of calculus, we know that $\frac{f(k)+f(k+1)}{2}$ is equal to \begin{equation}\label{eq1} \int_k^{k+1}\frac{d}{dx}\left[(x-k-\frac{1}{2})(f(x))\right]dx \end{equation}  Evaluating the derivative in the integrand of \ref{eq1} gives 

\begin{equation}\label{eq} \int_k^{k+1}f(x)+\left (x-k-\frac{1}{2}\right)f'(x) \:dx\end{equation} This can be broken up into two integrals, \[\int_k^{k+1}f(x)\:dx+\int_k^{k+1}  \left(x-k-\frac{1}{2}\right)f'(x)\: dx\]where the first integral is exactly what we wanted to compare to and the second integral is an error term.  Consider now only this error term \begin{equation}\label{eq2}\int_k^{k+1} \left(x-k-\frac{1}{2}\right)f'(x)\:dx\end{equation}Notice that \[\int_k^{k+1} \left(x-k-\frac{1}{2}\right)\:dx=0\]so we can add or subtract any constant from $f'(x)$ without changing the value of the integral.  Using this fact, \ref{eq2} can be written as \[\int_k^{k+1}\left(x-k-\frac{1}{2}\right)[f'(x)-f'(0)]\:dx\]  Since this equation is true for any $1\leq k\leq n-1$, we can simplify this and let $k=0$ and $k+1=1$, which gives \begin{equation}\label{eq3}\int_0^1\left(x-\frac{1}{2}\right)[f'(x)-f'(0)]\:dx\end{equation}This substitution will not cause any problems, because we can always replace $f(x)$ by $f(x+k)$ later. Since $f'(x)-f'(0)$ can be rewritten as \[\int_0^x f''(y)\:dy\] we can write \ref{eq3} as the double integral \[\int_0^1\int_0^x \left(x-\frac{1}{2}\right)f''(y)\:dy\:dx\]Switching the order of integration gives \[\int_0^1\int_y^1 \left(x-\frac{1}{2}\right)f''(y)\:dx\:dy\]
After integrating with respect to $x$ we are left with
 \[
\frac{1}{2}\int_0^1 y(1-y)f''(y)\:dy
\]
Using this combined with \ref{eq}, we have shown that \[\frac{1}{2}f(0)+\frac{1}{2}f(1)=\int_0^1 f(x)\:dx+\int_0^1 x(1-x)f''(x)\:dx\]At this point we can return our attention to the original problem, in which we need to sum up all of these integrals.  \begin{equation}\label{result}\sum_{k=1}^n f(k)=\frac{1}{2}f(n)-\frac{1}{2}f(1)+\sum_{k=1}^{n-1}\frac{1}{2}(f(k)+f(k+1))\end{equation}

From the previous calculation, 
\[\sum_{k=1}^n\frac{1}{2}(f(k)+f(k+1))=\int_1^nf(x)\:dx+\sum_{k=1}^n\int_0^1x(1-x)f''(k+x)\:dx\]Since we can move the summation inside the integral, we now turn our attention to $f''(x)$ to see if this sum will converge.  Calculating $f''(x)$ gives \[ \frac{-(\ln(q))^2q^x}{(1-q^x)^2}+\frac{1}{x^2}\]Substituting $e^{-\frac{\beta}{n}}$ for $q$ gives \[
\frac{-\beta^2e^{-\frac{\beta x}{n}}}{n^2(1-e^{-\frac{\beta x}{n}})^2}+\frac{1}{x^2}\]Multiplying top and bottom of the first fraction by $e^{\frac{\beta x}{n}}$\[\frac{-\beta^2}{n^2(e^{\frac{\beta x}{2n}}-e^{-\frac{\beta x}{2n}})^2}+\frac{1}{x^2}\]  Noticing that the bottom of the first fraction is equal to $(2n\sinh(\frac{\beta x}{2n}))^2$ leaves \[\frac{-\beta^2}{4n^2\sinh^2(\frac{\beta x}{2n})}+\frac{1}{x^2}\]This term will converge pointwise to $0$ by the dominated convergence theorem as $\frac{\beta}{n}\rightarrow 0$.  Going back to (\ref{result}) gives \[\ln\left(\frac{[n]!}{n!}\right)\sim\frac{1}{2}f(n)+\int_1^nf(x)dx+\frac{1}{2}\sum_{k=1}^{n-1}\int_0^1 x(1-x)f''(k+x)dx\] where this last term will converge  in the limit.  
\\
If we take $q = \exp(-\beta/n)$ for some $\beta \in \mathbb{R}$,
then we obtain
\begin{align*}
\ln\left(\frac{[n]!}{n!}\right)\, \Big|_{q = e^{-\beta/n}}\,
&=\, \frac{1}{2}\, \ln\left(\frac{1-e^{-\beta}}{(1-e^{-\beta/n})n}\right)
+ \int_1^n \ln\left(\frac{1-e^{-\beta x/n}}{(1-e^{-\beta/n})x}\right)\, dx\\
&\qquad + \frac{1}{2}\,\int_0^1 x(1-x)  \sum_{k=1}^{n-1} \left[\frac{1}{(k+x)^2} - \frac{\beta^2}{4 n^2 \sinh^2(\beta[k+x]/n)}\right]\, dx\, .
\end{align*}
We can rewrite this as
$$
\ln\left(\frac{[n]!}{n!}\right)\, \Big|_{q = e^{-\beta/n}}\,
=\, n A(\beta) + B(\beta) + R_n(\beta)\, ,
$$
where $A(\beta)$ and $B(\beta)$ do not depend on $n$ and $R_n(\beta)$ is a ``small'' remainder term,
which vanishes for $\beta$ fixed in the limit $n \to \infty$.
More precisely,
$$
A(\beta)\, =\, \frac{1}{n}\, \int_0^n  \ln\left(\frac{1-e^{-\beta x/n}}{\beta x/n}\right)\, dx\, ,
$$
which can be seen to be independent of $n$, by making a change of variables, $x = n y$ so that $dx = n dy$:
$$
A(\beta)\, =\, \int_0^1 \ln\left(\frac{1-e^{-\beta y}}{\beta y}\right)\, dy\, .
$$
We can write
$$
B(\beta)\, =\, \frac{\beta}{2} + \frac{1}{2}\, \ln\left( \frac{1-e^{-\beta}}{\beta} \right)\, .
$$
We throw all of the error terms we accumulated into the last term.
It is convenient to break it into three parts:
$$
R_n(\beta)\, =\, R_n^{(1)}(\beta) + R_n^{(2)}(\beta) + R_n^{(3)}(\beta)\, ,
$$
where
\begin{gather*}
R_n^{(1)}(\beta)\, =\, \frac{1}{2}\,\int_0^1 x(1-x)  \sum_{k=1}^{n-1} \left[\frac{1}{(k+x)^2} - \frac{\beta^2}{4 n^2 \sinh^2(\beta[k+x]/n)}\right]\, dx\, ,\\
R_n^{(2)}(\beta)\, 
=\, \int_1^n \ln\left(\frac{1-e^{-\beta x/n}}{(1-e^{-\beta/n})x}\right)\, dx - n A(\beta) - \frac{\beta}{2}\, ,\\
R_n^{(3)}(\beta)\, 
=\,  \frac{1}{2}\, \ln\left(\frac{1-e^{-\beta}}{(1-e^{-\beta/n})n}\right) - \frac{1}{2}\, \ln\left( \frac{1-e^{-\beta}}{\beta} \right)\, .
\end{gather*}
At this point, the proof of our theorem is complete, provided that we prove the following lemma.
\end{proof}

\begin{lemma}
For $\beta \in \mathbb{R}$ fixed, we have
$$
R_n^{(1)}(\beta)\, ,\ R_n^{(2)}(\beta)\, ,\ R_n^{(3)}(\beta)\, \to\, 0\, ,
$$
as $n \to \infty$.
\end{lemma}
\begin{proof}
We immediately know that $R_n^{(1)}(\beta)$
converges to $0$, as $n \to \infty$, since
$$
R_n^{(1)}(\beta)\, =\, \frac{1}{2}\,\int_0^1 x(1-x)  \sum_{k=1}^{n-1} \left[\frac{1}{(k+x)^2} - \frac{\beta^2}{4 n^2 \sinh^2(\beta[k+x]/n)}\right]\, dx\, ,
$$
and we may use the dominated convergence theorem.

For the next term, we notice
\begin{align*}
R_n^{(2)}(\beta)\, 
&=\, \int_1^n \ln\left(\frac{1-e^{-\beta x/n}}{(1-e^{-\beta/n})x}\right)\, dx - n A(\beta) - \frac{\beta}{2}\\
&=\, n \ln\left(\frac{\beta/n}{1-e^{-\beta/n}}\right) - \frac{\beta}{2} - \int_0^1 \ln\left(\frac{1-e^{-\beta x/n}}{\beta x/n}\right)\, dx\\
&=\, - n \ln \left(\frac{1-e^{-\beta/n}}{\beta/n}\right) - \frac{\beta}{2} - \int_0^1 \ln\left(\frac{1-e^{-\beta x/n}}{\beta x/n}\right)\, dx\\
&=\, - n \ln \left( \frac{1-e^{-\beta/n}}{\beta/n} \cdot e^{\beta/2n}\right) - \int_0^1 \ln\left(\frac{1-e^{-\beta x/n}}{\beta x/n}\right)\, dx\\
&=\, -n \ln \left( \frac{2n}{\beta}\, \sinh\left( \frac{\beta}{2n} \right) \right) -  \int_0^1 \ln\left(\frac{1-e^{-\beta x/n}}{\beta x/n}\right)\, dx\, .
\end{align*}
We know that the integrand converges to $\ln(1) = 0$ pointwise, so that the integral converges to 0
by DCT.
For the other term, we know that
$$
\frac{2n}{\beta}\, \sinh\left( \frac{\beta}{2n} \right)\, \to\, 1\, ,
$$
as $n \to \infty$.
Moreover, we have the Taylor expansion
$$
\sinh(x)\, =\, x + \frac{x^3}{3!} + \frac{x^5}{5!} + \dots + \frac{x^{2n+1}}{(2n+1)!} + \dots\, ,
$$
which means that
$$
\frac{\sinh(x)}{x}\, =\, 1 + \frac{x^2}{6} + \dots\, .
$$
This gives
$$
\frac{2n}{\beta}\, \sinh\left( \frac{\beta}{2n} \right)\, =\, 1 + \frac{\beta^2}{24n^2} + \dots\, =\,
1 + O(n^{-2})\, , \quad \text { as $n \to \infty$,}
$$
where $O(n^{-2})$ means that there is a function (which depends on $\beta$ as well as $n$)
which may be bounded by a finite constant $C$ (which is a function $C(\beta)$ depending on $\beta$)
times $n^{-2}$ for sufficiently large values of $n$.
This means
$$
\ln\left( \frac{2n}{\beta}\, \sinh\left( \frac{\beta}{2n} \right) \right)\,
=\, O(n^{-2})\, , \quad \text { as $n \to \infty$,}
$$
since $\ln(1+x) = x + O(x^2)$, as $x \to 0$.
Then
$$
-n \ln \left( \frac{2n}{\beta}\, \sinh\left( \frac{\beta}{2n} \right) \right)\, =\, O(n^{-1})\, , \quad \text { as $n \to \infty$,}
$$
which means in particular that it converges to $0$ as $n \to \infty$, (since $1/n$ does).
We have seen that $R_n^{(2)}(\beta)$ does indeed converges to $0$ as $n \to \infty$.

Finally, we have 
\begin{align*}
R_n^{(3)}(\beta)\, 
&=\, \frac{1}{2}\, \ln\left(\frac{1-e^{-\beta}}{(1-e^{-\beta/n})n}\right) - \frac{1}{2}\, \ln\left( \frac{1-e^{-\beta}}{\beta} \right)\\
&=\, \frac{1}{2}\, \ln\left( \frac{\beta}{(1-e^{-\beta/n})n}\right)\\
&=\, - \frac{1}{2}\, \ln\left( \frac{1-e^{-\beta/n}}{\beta/n} \right)\, .
\end{align*}
We know that
$$
\frac{1-e^{-\beta/n}}{\beta/n} \to 1\, ,
$$
as $n \to \infty$. Therefore, the logarithm converges to $0$.
\end{proof}

\section{The Mallows Measure}
Given a permutation of $n$ numbers $\pi\in S_n$, we define the inversion number $\mathrm{Inv}(\pi)$ to be 
$$
\mathrm{Inv}(\pi)=\#\{(i,j) \; : \;i<j \;\mathrm{and}\; \pi(i)>\pi(j)\}
$$
For each $q\in(0,1)$, the Mallows measure \cite{Mallows} is defined by 
$$
\mu_{n,q}(\pi)=\frac{q^{\mathrm{Inv}(\pi)}}{Z_{n,q}}
$$
where $Z_{n,q}$ is a normalization constant given by 
$$
Z_{n,q}=\sum_{\pi\in S_n} q^{\mathrm{Inv}(\pi)}=\prod_{k=1}^n\frac{1-q^k}{1-q}=[n]_q!
$$
where $[n]_q!$ is as stated in the previous chapter.  
The measure is related to the Iwahori-Hecke algebra as shown by Diaconis and Ram \cite{DiaconisRam}.  Note that for $q=1$, the Mallows measure is just the uniform measure on $S_n$, with all $n!$ permutations equally likely.  
\section{Fisher-Yates Algorithm}
The Fisher-Yates algorithm is a method of obtaining a uniform random permutation from a finite set.  The algorithm was first introduced by Fisher and Yates in \cite{FisherYates}.  Their original introduction of the algorithm was as a "paper and pencil" type algorithm for generating a random permutation of $n$ numbers by hand.  The algorithm was first presented as a computer algorithm by Durstenfeld \cite{Durstenfeld} and became more widely known in a work by Knuth \cite{Knuth}. 

The algorithm consists of the following four steps:
\newtheorem*{fy}{Fisher Yates Algorithm}
\begin{fy} 
\ \\
{\em
1) Set a counting variable $j$ to be equal to $1$.  Let $n$ denote the length of the desired sequence.  We will let $L$ be a sequence which holds our permutation.  We will begin by letting $L=(1)$.\\
2) Let $m=i+1$.  Pick an integer uniformly at random between $1$ and $m$.  Call this integer $k$.  \\
3) If $k=m$, then append $k$ to the end of the list $L$.  Otherwise, insert $m$ into $L$ at position $k$.  \\
4) Increase $i$ by $1$.  If $i<n$, then return to step 2.  Otherwise, the algorithm is complete.}  
\end{fy}

To see this algorithm in action, we will do an example for $n=4$.  
To begin with, $i=1$ and $L=(1)$.  To generate our random integers for this example, we used the Python generator {\em random.randint()}. \\
\\
{\bf Iteration 1} \\
1) $i=1$ and $L=1$.  \\
2) Since $m=i+1$, $m=2$.  Generating a random integer between $1$ and $2$, we get $k=2$. \\
3) Since $k=m$, we append $k$ to the end of $L$, giving $L=(1,2)$.   \\
4) We increase $i$ to $2$, and since $i<n$, we go back to step $2$ for another iteration. \\ \\
{\bf Iteration 2} \\
2) $i=2$, so $m=3$. Generating a random number between $1$ and $3$, we get $k=2$.  \\ 
3) Since $k<m$, we insert $m$ into position $k$ in the list.  This gives $L=(1,3,2)$. \\
4) Increasing $i$ by $1$, we have $i=3$, which is still less than $n$, so we go on for another iteration.
\\
\\
{\bf Iteration 3} \\
2) $i=3$, $m=4$. \\
3) We generate a random number between $1$ and $4$ and get $k=3$.  Since $k<m$, we insert $4$ into position $3$, which gives $L=(1,3,4,2)$.  \\
4) Once we increase $i$ by $1$, we notice that $i=4$, and so the algorithm terminates.
\\ 
\\
We end up with $L=(1,3,4,2)$ as our random permutation.  A Python code for performing this algorithm on a computer is given in the appendix.  
\\ \\
As mentioned, this algorithm shuffles the numbers $(1,\dots,n)$ uniformly, so that each permutation is equally likely.  Since we are trying to simulate a Mallows random permutation, we have adapted this algorithm to return a permutation distributed according to the Mallows measure.  
\newtheorem*{fy2}{Fisher Yates Algorithm for a Mallows Permutation}
\begin{fy2}
\ \\
{\em For a permutation of length $n$, with Mallows parameter $q$, we have the following algorithm to generate a Mallows distributed random permutation. \\ \\
1) Begin with $i=1$ and $L=(1)$. \\
2) Let $m=i+1$ and let $k$ be a random integer distributed according to a geometric distribution with probability $p=1-q$.   \\
3) Let $j=1+((k-1)\%m)$, where by $\%m$, we mean modulo $m$. \\
4) If $j=1$, append $m$ to the end of the list $L$.  Otherwise, insert $m$ into $L$ at position $m+1-j$.  \\
5) Increment $i$ by $1$.  If $i<n$, go back to step 2.  Otherwise, the algorithm terminates.

}
\end{fy2}
We will not go through an example here, as this algorithm runs very similarly to the uniform Fisher Yates algorithm.  Since it may not be obvious, we will prove why this modified algorithm generates a random permutation distributed according to the Mallows measure.  
\begin{thm}
The modified Fisher Yates algorithm stated above does give a permutation distributed according to the Mallows measure.
\end{thm}
\begin{proof}
Recall that the Mallows measure is given by 
$$
\mathbb{P}_{n,q}(\pi)=\frac{q^{\mathrm{Inv}(\pi)}}{Z_{n,q}}
$$
where
$$
Z_{n,q}=\prod_{k=1}^n\frac{1-q^k}{1-q}
$$
We will prove the theorem by induction.  We will start with the case $n=2$.  In this case, the only possible permutations are $(1,2)$ and $(2,1)$.  Based on the Mallows measure, 
$$
\mathbb{P}\left \{(1,2)\right\}=\frac{1-q}{1-q^2}
$$
and 
$$
\mathbb{P}\left\{(2,1)\right\}=\frac{q(1-q)}{1-q^2}
$$
Consider our algorithm.  We always start with $(1)$.  In this case, we will either be adding $2$ at the end of the permutation, or we will be inserting $2$ into slot $1$, giving us $(2,1)$.  Given the algorithm above, if $j=1$, then we will get $(1,2)$ and if $j=2$ we have $(2,1)$.  $j=1$ only if $k=1,3,5,...$.  Using the fact that $k$ is a geometric random variable, we have 
$$
\mathbb{P}\left\{(1,2)\right\}=\mathbb{P}\{j=1\}=\mathbb{P}\{k\; \mathrm{is} \;\mathrm{odd}\}
$$
$$
=\sum_{i=0}^{\infty}(1-q)q^{2i}
$$
$$
=\frac{1-q}{1-q^2}
$$
as desired.  On the other hand 
$$
\mathbb{P}\{(2,1)\}=\mathbb{P}\{j=2\}=\mathbb{P}\{k\; \mathrm{is}\;\mathrm{even}\}
$$
$$
=\sum_{i=0}^{\infty}(1-q)q^{2k+1}
$$
$$
=\frac{q(1-q)}{1-q^2}
$$
This completes the proof of the base case.   

For the inductive step, suppose that for a permutation of length $n$, the algorithm does in fact give a permutation distributed according to the Mallows measure.  In other words, letting $\pi$ be a permutation of length $n$, we know that 
$$
\mathbb{P}_{n,q}(\pi)=\frac{q^{\mathrm{Inv}(\pi)}}{\prod_{k=1}^n\frac{1-q^k}{1-q}}
$$
Suppose now that $\pi '$ is the same permutation as $\pi$, except with the element $n+1$ added in via the algorithm given.  We need to prove that 
$$
\mathbb{P}_{n+1}(\pi)=\frac{q^{\mathrm{Inv}(\pi)}q^{\mathrm{Inv}_{n+1}(\pi ')}}{\prod_{k=1}^{n+1}\frac{1-q^k}{1-q}}
$$
where 
$\mathrm{Inv}_{n+1}(\pi ')$ denotes the number of inversions caused by the element $n+1$.
We can assume that we have run our algorithm successfully up until $n$ and just need to perform the last step of the algorithm to add in $n+1$.  In this case, $i=n$ and $m=n+1$.  If $j=1$, we know that adding in $n+1$ will cause no additional inversions, so $q^{\mathrm{Inv}_{n+1}(\pi ')}=1$.    $j=1$ only if $k-1$ is a multiple of $n+1$.  Using the fact that $k$ is a geometric variable, we have 
$$
\mathbb{P}\{j=1\}=\mathbb{P}\{(k-1)\%(n+1)=0\}
$$
$$
=\sum_{i=0}^{\infty}q^{i(n+1)}(1-q)=\frac{1-q}{1-q^{n+1}}
$$
This implies that 
$$
\mathbb{P}_{n+1}(\pi')=\frac{q^{\mathrm{Inv}(\pi)}}{\prod_{k=1}^{n+1}\frac{1-q^k}{1-q}}
$$
as desired (since adding in the last point did not cause any additional inversions).

Now suppose that $j=2$.  This implies that $(k-1)\%(n+1)=1$.  This will occur only if $k=\ell(n+1)+1$ for some integer $\ell$.  If $j=2$, then $\mathrm{Inv}_{n+1}(\pi ')=1$, since $n+1$ will only cause an inversion with an element in the $n$th position.  We have
$$
\mathbb{P}\{j=2\}=\sum_{i=0}^{\infty} q^{i(n+1)+1}(1-q)=\frac{q(1-q)}{1-q^{n+1}}
$$
and in this case, we have 
$$
\mathbb{P}_{n+1}(\pi)=\frac{q^{\mathrm{Inv}(\pi)}q}{\prod_{k=1}^{n+1}\frac{1-q^k}{1-q}}
$$
as desired.  

This pattern will continue in general.  Suppose that adding in $n+1$ causes $I$ inversions.  Then, we know that it must have been added in at position $n+1-I$.  From the algorithm, this means that $j=I+1$.  This will occur only if $(k-1)\%(n+1)=I$.  In this case 
$$
\mathbb{P}\{j=I+1\}=\sum_{i=0}^{\infty}q^{i(n+1)+I}(1-q)=\frac{q^I(1-q)}{1-q^{n+1}}
$$
From this, we have 
$$
\mathbb{P}_{n+1}(\pi')=\frac{q^{\mathrm{Inv}(\pi)}q^{I}}{\prod_{k=1}^{n+1}\frac{1-q^k}{1-q}}
$$
which completes the proof.
\end{proof}

\section{Length of the Longest Increasing Subsequence}

Consider a permutation $\pi\in S_n$.  An increasing subsequence $i_1,i_2,\dots,i_k$ of a permutation $i\mapsto \pi(i)$ is a subsequence such that 
$i_1<\dots<i_k$ and $\pi(i_1)<\pi(i_2)<\dots<\pi(i_k)$.  We will be concerned with determining the length of the longest increasing subsequence in a given permutation.  Denote the length of the longest increasing subsequence of $\pi$ by $\ell(\pi)$.  

The following example is due to \cite{ADHamm}.  Consider the permutation given by 
$$
7\;\;\;2\;\;\;8\;\;\;1\;\;\;3\;\;\;4\;\;\;10\;\;\;6\;\;\;9\;\;\;5
$$
where $\pi(1)=4$, $\pi(2)=2$, $\pi(3)=5$, etc.  Then a longest increasing subsequence is given by 
$$
1\;\;\;3\;\;\;4\;\;\;6\;\;\;9
$$
In this case $\ell(\pi)=5$.  \\Notice that the longest increasing subsequence is not necessarily unique. 
$$
2\;\;\;3\;\;\;4\;\;\;6\;\;\;9
$$ is also an increasing subsequence of length $5$.  The longest increasing subsequence problem goes back to Ulam \cite{Ulam}.  Ulam asked what is the distribution of the length of the longest monotone (increasing or decreasing) subsequence of a uniform random permutation.  While we will not go into the history here, a detailed account of Ulam's problem and Monte Carlo methods can be found in \cite{Hammersley}.  Quite a bit of progress has been made concerning the distribution of the length of the longest increasing subsequence, provided that the permutation is uniformly distributed.  Hammersley \cite{Hammersley} showed that $\mathbb{E}\ell(\pi)\sim c\sqrt{n}$, where $n$ is the length of the permutation and $c$ is a constant.  Vershik and Kerov \cite{VershikKerov} and Logan and Shepp \cite{LS} proved that the constant $c$ is equal to $2$.  Their methods of proof relied on hard analysis of the asymptotics of Young tableau.  Aldous and Diaconis \cite{AldousDiaconis} give an interacting particle process argument for $c=2$.  In addition, Baik, Deift, and Johansson \cite{BDJ} showed that the fluctuations of the length of the longest increasing subsequence for a uniform permutation are Tracy-Widom, on the order of $n^{1/6}$.  More specifically, they show that 
$$
\frac{\ell_n-2\sqrt{n}}{n^{1/6}}\xrightarrow{d} \chi
$$
where $\chi$ is a random variable with Tracy Widom distribution.  The distribution function for the Tracy Widom distribution is 
$$
F(t)=\exp\left(-\int_t^{\infty}(x-t)u^2(x)dx\right)
$$
where $u(x)$ is the solution of the Painlev\'{e} equation 
$$
u_{xx}=2u^3+xu
$$
See \cite{BDJ} for more background on the Tracy Widom distribution.  

Much less is known about the distribution of the length of the longest increasing subsequence of a random permutation distributed according to the Mallows measure.  In \cite{MS}, Mueller and Starr proved a weak law of large numbers result analogous to the Vershik-Kerov and Logan-Shepp results for the uniform case.  To continue this work, we would like to bound the fluctuations of the length of the longest increasing subsequence of a Mallows permutation.  As a first step in this direction, we use the modified Fisher-Yates algorithm to generate a random Mallows permutation, then use an algorithm called patience sorting to compute the length of the longest increasing subsequence of the generated permutation.  

\section{Patience Sorting}
The presentation of patience sorting that we describe here follows the algorithm as given by Aldous and Diaconis in \cite{AldousDiaconis}.  Patience sorting is a type of one person card game.  Imagine that we have a deck of cards with the numbers $1,\dots,n$ on them.  We shuffle the deck thoroughly, and put the cards in a pile face down.  We turn the cards face up one at a time and put them into a pile according to the following rule: \\
{\em A low card may be placed on top of a higher card (i.e. a 2 on top of a 7), but a higher card must be placed into a new pile to the right of the current piles}. \\
The object of the "game" is to finish with the fewest piles.  

As a short example, let suppose that we have a pile of cards labeled $1,\dots,6$.  Let us shuffle them (uniformly at random) and suppose that we end up with the permutation $$4\;\;\;1\;\;\;3\;\;\;2\;\;\;6\;\;\;5$$ with the $4$ on the top of the deck, and the $5$ on the bottom.  To begin the patience sorting algorithm, we start with the card $4$, which will be the beginning of our first pile.  The next card that we draw is a $1$.  Since this is less than $4$, it can go on top of the four in the first pile, so that our piles look like 
$$
\begin{array}{c}
1 \\
4
\end{array}
$$
Next we draw a $3$.  Since this is larger than $1$, it cannot go to the top of the pile, it must start a new pile.  Now we have 
$$
\begin{array}{ccc}
1 & 3 \\
4 &
\end{array}
$$
We next add the $2$ to the top of the second pile, since $2>3$.  
$$
\begin{array}{ccc}
1 & 2 \\
4 & 3
\end{array}
$$
Adding in $6$ requires us to make a new pile
$$
\begin{array}{ccc}
1 & 2 & 6 \\
4 & 3 & 
\end{array}
$$
We can then place the last card, $5$, on top of the third pile giving us 
$$
\begin{array}{ccc}
1 & 2 & 5 \\
4 & 3 & 6
\end{array}
$$
Notice that we end up with $3$ piles.  Notice also that the length of the longest increasing subsequence of the permutation is $3$.  Once such subsequence is 
$$
1 \; \; \; 3 \; \; \; 6
$$
and there are several more, but none of length more than $3$.  It turns out that this is not a coincidence.  The following theorem is due to Aldous and Diaconis \cite{AldousDiaconis}
\begin{thm}
With a given deck $\pi$, patience sorting played with the greedy strategy ends with exactly $\ell(\pi)$ piles.  In addition, the game played with any legal strategy ends with at least $\ell(\pi)$ piles.  
\end{thm}

\begin{proof}
Suppose that we have cards $a_1 < a_2<\dots <a_k$ an increasing subsequence in our pile.  Then under any legal strategy, each $a_i$ must be placed in a stack to the right of $a_{i-1}$, since any card placed on top of $a_{i-1}$ must be less than the value on $a_{i-1}$.  This implies that the final number of piles must be at least $k$, and since this is the length of an arbitrary increasing subsequence, the number of piles must be at least $\ell(\pi)$.  Furthermore, suppose we choose the greedy strategy, where we {\em only} start a new pile if we are forced to.  Suppose each time we put a card $a$ into any pile other than the first pile, we place a pointer from that card to the card on the top of the pile immediately to the left.  Notice that this card will always be less than our current card.  At the end of the game, if we follow the pointers backward from the top card on the last pile, we will have an increasing subsequence whose length is the number of piles.  
\end{proof}

Using this theorem and the patience sorting algorithm, it is possible to have a computer compute the longest increasing subsequence of a Mallows permutation.  The Python code for such a program is included in the appendix.  The following figures show a histogram for the length of the longest increasing subsequence of a permutation of length $n=10,000$, run $200$ times under the uniform distrubution and the Mallows distribution for varying values of $q$.  
\begin{figure}
\centering
\includegraphics[width=4in]{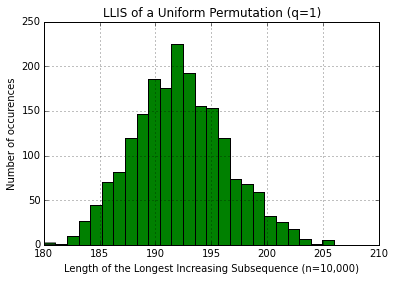}
\caption{Length of the longest increasing subsequence of a uniform permutation of length $10,000$}
\end{figure}

\begin{figure}
\centering
\includegraphics[width=4in]{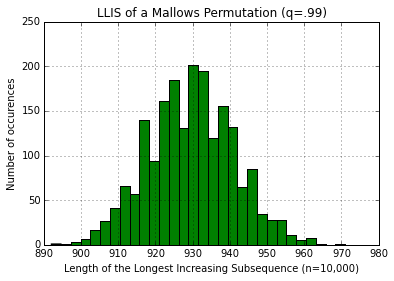}
\caption{Length of the longest increasing subsequence of a Mallows permutation of length $10,000$ with $q=0.99$}
\end{figure}

\begin{figure}
\centering
\includegraphics[width=4in]{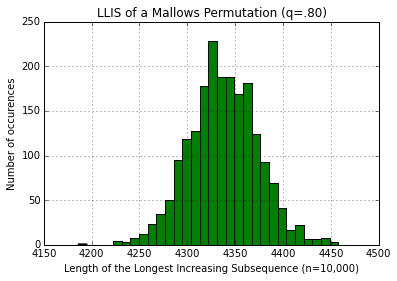}
\caption{Length of the longest increasing subsequence of a Mallows permutation of length $10,000$ with $q=0.88$}
\end{figure}

Unfortunately, we were unable to get too much useful information out of our simulations due to the fact that we did not have enough computing power to run the algorithm for big enough permutations.  We computed statistics (mean, variance, skewness, kurtosis) for our permutations in hopes of matching the experimental data to what was expected for a type of Tracy Widom distribution, but the results were inconclusive.  

\section{Four Square Problem}
We would now like to give an idea as to how the $q$-Stirling's formula arises in the problem of bounding the fluctuations of the length of the longest increasing subsequence in a Mallows random permutation.  This work is ongoing, but an important step is discussed in this section.  A random permutation (Mallows or otherwise), can be viewed as a set of points in a rectangle, in the following way:

Consider $n$ points $(x_i, y_i)$ in the rectangle $[0,1]\times [0,1]$ in $\mathbb{R}^2$ with all coordinates distinct.  The set of points specifies a permutation $\pi\in S_n$ by the rule: "The point with the $i$th smallest $y$ coordinate has the $\pi(i)$th smallest $x$-coordinate".  Hence, given a set of points in a box, we can obtain a permutation from these points.  Depending on how the points are distributed in the box, we can obtain permutations with different distributions.  As an example, if the points are uniformly distributed in the box, we obtain a uniform random permutation.  

To begin to bound the fluctuation of a Mallows random permutation, we assume that we have $n$ points in the unit square distributed so that they give a Mallows random permutation.  We then divide the square into a large number of small subsquares.   If the size of each subsquare is small enough, the points in the subsquare will be approximately uniformly distributed.  We then hope to couple our model to a model of Deuschel and Zeitouni \cite{DZ} to bound the fluctuations.  

\begin{figure}
\begin{center}
$$
\begin{tikzpicture}
\draw
(-2,-2)--(-2,2)--(2,2)--(2,-2)--(-2,-2);
\draw
(0,-2)--(0,2);
\draw
(-2,0)--(2,0); 
\node at (-1,-1) {$R_{11}$};
\node at (-1,1) {$R_{21}$};
\node at (1,-1) {$R_{12}$};
\node at (1,1) {$R_{22}$};
\end{tikzpicture}
$$
\caption{An example of a decomposition of $[0,1]^2$ into four rectangles $R_{11}, R_{12}, R_{21}, R_{22}$.}
\label{fig:FourSquare}
\end{center}
\end{figure}
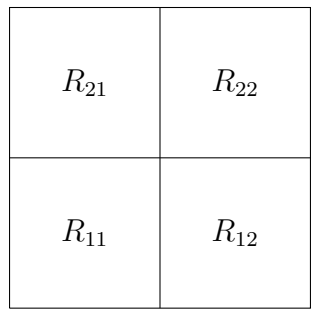

This argument will not be presented here.  For now, we will simply look at the division of the unit square into four subrectangles to illustrate the use of the $q$-Stirling's formula.  Suppose that we divide the unit square into four rectanges, which we refer to as $R_{11}, \; R_{12},\; R_{21},\;  R_{22}$.  See figure \ref{fig:FourSquare}.  Assuming that we have $n$ total points in the square, let $n_{11}$ denote the number of points in $R_{11}$, $n_{12}$ denote the number of points in $R_{12}$, and so on.  We then have $n_{11}+n_{12}+n_{21}+n_{22}=n$.  Denote the area of rectange $R_{ij}$ as $p_{ij}$.  Consider the distribution of the points $(x_1,y_1), (x_2,y_2),\dots, (x_n, y_n)$ in the square.  If the points are distributed uniformly (i.e. if $q=1$), then the probability of the event
$$
\cap_{i,j=1}^n \{\# \{k:(X_k,Y_k)\in R_{ij}=n_{ij}\}\}
$$
is given by the usual multinomial formula 
\begin{equation}
\label{com1}
\frac{n!}{n_{11}!n_{12}!n_{21}!n_{22}!} p_{11}^{n_{11}}p_{12}^{n_{12}}p_{21}^{n_{21}}p_{22}^{n_{22}}
\end{equation}
For $0<q<1$ and $q\not = 1$ (in other words, a Mallows random permutation), the correct probability is obtained by multiplying the above expression by the factor 
\begin{equation}
\label{com2}
q^{n_{12} n_{21}}\,
\frac{[n_{11}+n_{12}]![n_{11}+n_{21}]![n_{12}+n_{22}]![n_{21}+n_{22}]!}
{(n_{11}+n_{12})!(n_{11}+n_{21})!(n_{12}+n_{22})!(n_{21}+n_{22})!}
\cdot \frac{(n_{11})!(n_{12})!(n_{21})!(n_{22})!}{[n_{11}]![n_{12}]![n_{21}]![n_{22}]!}
\cdot \frac{n!}{[n]!}\, .
\end{equation}
where $[a]!$ denotes the $q$ factorial as defined earlier.  Combining \ref{com1} and \ref{com2} and using the notation $\{n\}!:=[n]!/n!$, we have
\begin{multline}
\label{4square}
\mathbb{P}\left(\cap_{i,j=1}^n \{\# \{k:(X_k,Y_k)\in R_{ij}=n_{ij}\}\right)= \\
\frac{n!}{\prod_{i,j=1}^2 (n_{ij})!}\left(\prod_{i,j=1}^2 p_{ij}^{n_{ij}}\right) \, 
\frac{\{n_{11}+n_{12}\}!\{n_{11}+n_{21}\}!\{n_{12}+n_{22}\}!\{n_{21}+n_{22}\}!}
{\{n_{11}\}!\{n_{12}\}!\{n_{21}\}!\{n_{22}\}!\{n_{11}+n_{12}+n_{21}+n_{22}\}!}\,q^{n_{21}n_{12}} 
\end{multline}
This formula is somewhat involved, but it is explicit.  By applying the $q$-Stirling's formula, we can obtain the exact asymptotics for the probability distribution in the limit $n\rightarrow\infty$, with $q=e^{-\beta/n}$.  

Recall that Stirling's formula says that 
$$
n!\sim e^{n\ln n-n}\sqrt{2\pi n}
$$
and the $q$-Stirling's formula states that for $q=e^{-\beta/n}$,
$$
\ln\left(\frac{[n]!}{n!}\right)=nA(\beta)+B(\beta)+R_n(\beta)
$$
where 
$$
A(\beta)=\int_0^1\ln\left(\frac{1-e^{-\beta x}}{\beta x}\right) dx
$$
$$
B(\beta)=\frac{\beta}{2}+\frac{1}{2}\ln\left(\frac{1-e^{-\beta x}}{\beta x}\right)
$$
and $R_n(\beta)$ is a remainder term which goes to zero as $n\rightarrow\infty$. Before we can apply these asymptotics to equation \ref{4square}, we need a preliminary lemma.
\begin{lem}
For $q=e^{-\beta/n}$, 
$$
\ln(\{n_{ij}\}!)=n_{ij} A\left(\frac{n_{ij}}{n}\, \beta\right) + B\left(\frac{n_{ij}}{n}\, \beta\right) + R_{n_{ij}}\left(\frac{n_{ij}}{n}\, \beta\right)\,
$$
\end{lem}
This lemma is easily proved using the $q$-Stirling's formula and rewriting $q$ as $q=e^{-\beta '/n_{ij}}$, where $\beta '=\frac{n_{ij}\beta}{n}$.  

To break down the asymptotics of \ref{4square} a bit, let 
$$
W_q = \frac{\{n_{11}+n_{12}\}!\{n_{11}+n_{21}\}!\{n_{12}+n_{22}\}!\{n_{21}+n_{22}\}!}
{\{n_{11}\}!\{n_{12}\}!\{n_{21}\}!\{n_{22}\}!\{n_{11}+n_{12}+n_{21}+n_{22}\}!}\,q^{n_{21}n_{12}}
$$
In addition, let $\nu_{ij}=\frac{n_{ij}}{n}$.  For this analysis, we will assume that all $\nu_{ij}$ are order $1$, so that we are not letting any of the squares be too small.  If this is the case, we can make the approximation 
$$
\ln\{n_{ij}\}!\approx n\nu_{ij}A(\beta v_{ij})
$$
Using this assumption, we have the following lemma
\begin{lem}
For $\nu_{11}, \nu_{12}, \nu_{21}, \nu_{22}>0$, we have 
\begin{multline}
\lim_{n\rightarrow\infty} \frac{1}{n}\ln(W_q)= -\beta \nu_{12} \nu_{21} + (\nu_{11}+\nu_{12}) A(\beta[\nu_{11}+\nu_{12}]) 
+ (\nu_{11}+\nu_{21}) A(\beta[\nu_{11}+\nu_{21}])\\
\qquad + (\nu_{12}+\nu_{22}) A(\beta[\nu_{12}+\nu_{22}])
+ (\nu_{21}+\nu_{22}) A(\beta[\nu_{21}+\nu_{22}])\\
\qquad - \nu_{11} A(\beta \nu_{11})
- \nu_{12} A(\beta \nu_{12})
- \nu_{21} A(\beta \nu_{21})
- \nu_{22} A(\beta \nu_{22})
- A(\beta)
\end{multline}
\end{lem}
\begin{proof}
By definition, 
$$
\lim_{n\rightarrow\infty}\frac{1}{n}\ln(W_q)
$$
$$
=\lim_{n\rightarrow\infty} \frac{1}{n}\ln\left(q^{\nu_{12}n\nu_{21}n}\frac{\{\nu_{11}n+\nu_{12}n\}!\{\nu_{11}n+\nu_{21}n\}!\{\nu_{12}n+\nu_{22}n\}!\{\nu_{21}n+\nu_{22}n\}!}{\{\nu_{11}n\}!\{\nu_{12}n\}!\{\nu_{21}n\}!\{\nu_{22}\}!\{\nu_{11}n+\nu_{12}n+\nu_{21}n+\nu_{22}n\}!}\right)
$$
Using the $q$-Stirling formula and the previous lemma, we get
\begin{multline}
=\lim_{n\rightarrow\infty}\frac{1}{n}( \ln\left(q^{\nu_{12}n\nu_{21}n}
\right)+ (\nu_{11}n+\nu_{12}n)A(\beta(\nu_{11}+\nu_{12}))+\dots + (\nu_{21}n+\nu_{22}n)A(\beta(\nu_{21}+\nu_{22})) \\
-\nu_{11}nA(\beta\nu_{11})-\nu_{12}nA(\beta\nu_{12})-\nu_{21}nA(\beta nu_{21})-\nu_{22}nA(\beta \nu_{22}) \\
-(\nu_{11}n+\nu_{12}n+\nu_{21}n+\nu_{22}n)A(\beta(\nu_{11}+\nu_{12}+\nu_{21}+\nu_{22})) )
\end{multline}
Distributing the $n$ and taking the limit immediately gives us what we need except for the first and last terms.  Consider just 
$$
\lim_{n\rightarrow\infty}\ln(q^{\nu_{12}n\nu_{21}n})
$$
Since $q=e^{-\beta/n}$, this is
$$
=\lim_{n\rightarrow\infty}\frac{1}{n}e^{-\beta(\nu_{12}n\nu_{21})}
$$
$$
=\frac{1}{n}(-\beta\nu_{12}n\nu_{21})=-\beta\nu_{12}\nu_{21}
$$
This gives us the first term in our lemma.  The last term is equal to 
$$
\lim_{n\rightarrow\infty}\frac{1}{n}(\nu_{11}n+\nu_{12}n+\nu_{21}n+\nu_{22}n)A(\beta(\nu_{11}+\nu_{12}+\nu_{21}+\nu_{22}))
$$
$$
=\lim_{n\rightarrow\infty}(\nu_{11}+\nu_{12}+\nu_{21}+\nu_{22})A(\beta(\nu_{11}+\nu_{12}+\nu_{21}+\nu_{22}))
$$
$$
=A(\beta)
$$
since $\sum_{i,j=1}^2\nu_{ij}=1$. Putting all terms together proves the lemma.  
\end{proof}
It is worth noting that the asymptotics for 
$$
\frac{1}{n}\ln W_q
$$
given by this lemma give us an equation analogous (and very similar) to equation (6) in \cite{starrWal}.  

Combining these asymptotics with the asymptotics for
$$
\frac{n!}{\prod_{i,j=1}^2 (n_{ij})!}\left(\prod_{i,j=1}^2 p_{ij}^{n_{ij}}\right)
$$
gives
\begin{equation}
\mathbb{P}\left(\cap_{i,j=1}^n \{\# \{k:(X_k,Y_k)\in R_{ij}=n_{ij}\}\right)=\frac{\sqrt{2\pi n}}{\prod_{i,j=1}^2 \sqrt{2\pi n_{ij}}}e^{\tilde{A}}
\end{equation}
where 
\begin{multline}
\tilde{A}=n(\ln(n)-\sum_{i,j=1}^n\nu_{ij}\ln(n_{ij}) \\ -\beta \nu_{12} \nu_{21} + (\nu_{11}+\nu_{12}) A(\beta[\nu_{11}+\nu_{12}]) 
+ (\nu_{11}+\nu_{21}) A(\beta[\nu_{11}+\nu_{21}])\\
\qquad + (\nu_{12}+\nu_{22}) A(\beta[\nu_{12}+\nu_{22}])
+ (\nu_{21}+\nu_{22}) A(\beta[\nu_{21}+\nu_{22}])\\
\qquad - \nu_{11} A(\beta \nu_{11})
- \nu_{12} A(\beta \nu_{12})
- \nu_{21} A(\beta \nu_{21})
- \nu_{22} A(\beta \nu_{22})
\end{multline}

\section{Conclusion and Outlook}
As previously mentioned, the results in this section are preliminary steps toward bounding the fluctuations of the length of the longest increasing subsequence of a Mallows permutation.  The next step is to use the approach of the four square problem to solve a nine square problem. Once the asymptotics are computed for that problem, we can generalize to a large number of small squares and obtain a local central limit theorem for the counts on small subsquares.  After that, we hope to couple our model to the model of Deuschel and Zeitouni \cite{DZ} and then use Talagrand's isoperimetric inequality to bound the fluctuations.  These results will appear in a future work.  

\appendix  
\chapter{Python code: Simulating a Mallows Random Permutation}

\definecolor{keywords}{RGB}{255,0,90}
\definecolor{commends}{RGB}{0,0,113}
\definecolor{red}{RGB}{160,0,0}
\definecolor{green}{RGB}{0,150,0}

\lstset{language=Python}
\begin{lstlisting}
''''
permutation.py

@author: Meg Walters
'''
import numpy as np
import random
import math
import bisect
import matplotlib.pyplot at plt

def patience_sort(list): 
#This function creates a multidimensional array 
#containing all of the stacks 
#of the patience sorting algorithm
#Input: 
#	list: list of numbers to sort
#Output:
#	len(stacks): returns the numbers of stacks
#		
    stacks = []
    len_stacks=[] #variable to keep track of number of stacks
    for x in list: #iterate through list of numbers
        temp_stack = [x] #put number in a temporary stack
        i = bisect.bisect_left(stacks, temp_stack) 
	#determines where number should be inserted if
	# it was to be inserted in order
        if i != len(stacks): 
	#if number is not larger then all numbers on top
	#of stacks
            stacks[i].insert(0, x) #put number on appropriate stack
            len_stacks.append(len(stacks)) #update length variable
        else:
            stacks.append(temp_stack) #create new stack
            len_stacks.append(len(stacks)) #update length variable

def fisher_yates(length): 
#uses fisher yates algorithm to create random permutation
#Input:
#	length: desired length of permutation
#Output:
#	L: random permutation

    L=[1] #begin with only 1 in the list
    
    for i in xrange(length-1): 
    #iterate to create a list of length 'length'
        m=i+2 #initialize/update m
        k=random.randint(1,m) 
          #generate a random integer between 1 and m.  
        if k==m:
            L.append(m) #append m to the end of the list 
        else:
            L.insert(k-1,m) #insert m in k-1 place in list
    
    return L #return random permutation

def mallows(length,q):
#uses the mallows measure to create a permutatioin
#Input:
#   length: desired length of permutation
#   q: 1-probability
#Output:
#   L: permutation

    L=[1] #begin with only 1 in the list
    
    for i in xrange(length-1): 
    #iterate to create list of length 'length'
        m=i+2 #initialize/update m
        x=np.random.geometric(p=1-q,size=1)
          #generate a geometric random integer, probability p
        y=1+((x-1)%m) #find y based on mallows
        if y==1: 
            L.append(m) #append m to end of the list
        else:
            L.insert(m+1-y,m) 
              #insert m at the m+1-y position in the list
    return L

length_list=10000  #change this variable to change n
length_data=200 #change this value to change 
		#number of times program should
                     #run to collect data
    
data=[] #initialize array to hold data

q=.8 #change this value to change the Mallows q

#create data
for i in xrange(length_data):
    data.append(patience_sort(mallows(length_list,q)))

#create histogram for given data
fig=plt.figure() 
ax=fig.add_subplot(111)
n, bins, patches = ax.hist(data,30,normed=False,...
 ...facecolor='green', histtype='bar',align='mid')
ax.grid(True)
plt.title('LLIS of a Mallows Permutation')
plt.xlabel('Length of the Longest Increasing Subsequence')
plt.ylabel('Number of occurences')
plt.show()

\end{lstlisting}
\bibliographystyle{urcsbiblio}
\bibliography{mybiblio}

\appendix

\end{document}